\documentclass{article}
\RequirePackage[OT1]{fontenc}
\RequirePackage{amsthm,amsmath}
\usepackage{hyphenat}
\usepackage{mathrsfs}
\RequirePackage%[numbers]
{natbib}
\RequirePackage[%hidelinks
colorlinks,citecolor=blue
,urlcolor=blue
]{hyperref}
\usepackage{framed,color}
\usepackage{etoolbox}
\usepackage{marginnote}
\BeforeBeginEnvironment{shaded}{\begin{minipage}{\linewidth}}
  \AfterEndEnvironment{shaded}{\end{minipage}\par}
\usepackage{bm}
\usepackage{bbold}
\usepackage{prodint}
\usepackage{listings}
\usepackage{lipsum,calc}
\usepackage{mathtools}
\usepackage[usestackEOL]{stackengine}
\usepackage[shortlabels]{enumitem}
\usepackage{array}
\usepackage{framed}
\usepackage[T1]{fontenc}
\usepackage{adjustbox}
\usepackage{booktabs}
\usepackage{accents}
\usepackage[utf8]{inputenc}
\usepackage{graphicx}
\usepackage{grffile}
\usepackage{longtable}
\usepackage{wrapfig}
\usepackage{rotating}
\usepackage[normalem]{ulem}
\usepackage{amsmath}
\usepackage{textcomp}
\usepackage{amssymb}
\usepackage{capt-of}
\newcommand{\ubar}[1]{\underaccent{\bar}{#1}}
\newcommand{\EE}{\mathbb{E}}

\newcommand{\mm}{m}

\newcommand{\cadlag}{càdlàg }
\newcommand{\banachk}{\mathbb{D}^k([0, \eta])}

\newcommand{\FF}{\mathscr{F}}
\newcommand{\vv}{\mathrm{v}}

\newcommand{\F}{\mathcal{F}}

\newcommand{\QQ}{Z}

\newcommand{\R}{\mathbb{R}}
\newcommand{\N}{\mathbb{N}}
\newcommand{\1}{\mathbb{1}}

\newcommand{\logloss}[1]{\ensuremath{\mathscr{L}_{t,#1}}}
\newcommand{\intlogloss}[1]{\ensuremath{\bar{\mathscr{L}}_{#1}}}

\newcommand{\tmax}{\tau}
\newcommand{\minus}{-}

\newcommand{\white}{\color{white}}

\newcommand{\logit}{\text{logit}}
\newcommand{\expit}{\text{expit}}
\newcommand{\argmax}{\text{argmax}}
\newcommand{\argmin}{\text{argmin}}

\newcommand{\eps}{\varepsilon}
\newcommand\independent{\protect\mathpalette{\protect\independenT}{\perp}}\def\independenT#1#2{\mathrel{\rlap{$#1#2$}\mkern2mu{#1#2}}}
\newtheorem{thm}{Theorem}
\newtheorem{cor}{Corollary}
\newtheorem{defi}{Definition}

\newtheorem{lemma}{Lemma}

\newtheorem{remark}{Remark}
\newtheorem{assumption}{Assumption}

\def\:{\hskip0pt}
\begin{document}

\title{Continuous-time targeted minimum loss-based estimation of intervention-specific mean outcomes}
  % \runtitle{Continuous-time TMLE}
\author{\small Helene C. W. Rytgaard$^{1,*}$, 
Thomas A. Gerds$^{1}$, and Mark J. van der Laan$^{2}$ \\
\small $^{1}$Section of Biostatistics, University of Copenhagen, Denmark\\ \small$^{2}$Devision of Biostatistics, University  of California, Berkeley}
  
% \address{H. C. Rytgaard, T. A. Gerds\\
  % Section of Biostatistics \\
  % University of Copenhagen \\ 
  % Øster Farimagsgade 5 \\
  % 1014 København K\\
  % Denmark \\
  % \printead{e1}\\\phantom{E-mail: } 
  % \printead*{e3}}

% \address{M. J. van der Laan\\
  % Division of Biostatistics\\
  % and\\
  % Center for Targeted Machine \\
  % Learning and Causal Inference\\
  % 101 Haviland Hall \\
  % Berkeley, California, 94720 \\
  % USA \\
  % \printead{e2}\\
  % \phantom{E-mail:\ }}

\maketitle
\begin{abstract}
  \noindent This paper studies the generalization of the targeted
  minimum loss-based estimation (TMLE) framework to estimation of
  effects of time-varying interventions in settings where both
  interventions, covariates, and outcome can happen at subject-specific
  time-points on an arbitrarily fine time-scale. TMLE is a general
  template for constructing asymptotically linear substitution
  estimators for smooth low-dimensional parameters in
  infinite-dimensional models. Existing longitudinal TMLE methods are
  developed for data where observations are made on a discrete
  time-grid.

  We consider a continuous-time counting process model where intensity
  measures track the monitoring of subjects, and focus on a
  low-dimensional target parameter defined as the
  intervention-specific mean outcome at the end of follow-up. To
  construct our TMLE algorithm for the given statistical estimation
  problem we derive an expression for the efficient influence curve
  and represent the target parameter as a functional of intensities
  and conditional expectations. The high-dimensional nuisance
  parameters of our model are estimated and updated in an iterative
  manner according to separate targeting steps for the involved
  intensities and conditional expectations.

  The resulting estimator solves the efficient influence curve
  equation. We state a general efficiency theorem and describe a
  highly adaptive lasso estimator for nuisance parameters that allows
  us to establish asymptotic linearity and efficiency of our estimator
  under minimal conditions on the underlying statistical model.
\end{abstract}

% \begin{keyword}
  % \kwd{targeted minimum loss-based estimation (TMLE)}
  % \kwd{time-varying confounding} \kwd{continuous-time interventions}
  % \kwd{semiparametric model} \kwd{efficient estimation} \kwd{causal
    % inference.}
% \end{keyword}

\section{Introduction}
\label{sec:introduction}

We consider a continuous-time longitudinal data structure of
\(n \in \N\) independent and identically distributed observations of a
multivariate process on a bounded interval of time \([0,\tmax]\) with
distribution \(P_0\) belonging to a semiparametric statistical model
\(\mathcal{M}\). We are interested in assessing the effect of
interventions on an outcome of interest under interventions that can
happen at arbitrary points in time and are subject to time-dependent
confounding. Our focus is on the construction of an asymptotically
efficient substitution estimator of intervention-specific mean
outcomes represented as a parameter
\(\Psi \, : \, \mathcal{M} \rightarrow \R\).

In all that follows, we use the words `intervention' and `treatment'
synonymously. Recent developments in the field of causal inference
have produced numerous methods to deal with effects of time-varying
treatments in presence of time-varying confounding
\citep{robins1986new,robins1987addendum,
  robins1989analysis,robins1989control,
  robins1992estimation,robins1998marginal,robins2000marginal,robins2000marginala,robins2000robust,bang2005doubly,
  robins2008estimation,van2010targeted,van2010targetedII,petersen2014targeted}.
These methods deal with settings with a fixed number of time-points at
which subjects of a population are all measured and can be intervened
upon.

In the present work we consider a continuous-time model, utilizing a
counting process framework \citep{andersen2012statistical} where
intensity processes define the rate of a finite number of continuous
monitoring times for each subject conditional on their observed
history, see Figure \ref{fig:time:illustration}. Our approach is
closely related to the work of \cite{lok2008statistical} and of
\cite{roysland2011martingale,roysland2012counterfactual}, who propose
continuous-time versions of structural nested models and marginal
structural models
\citep{robins1989analysis,robins1989control,robins1992estimation,robins1998marginal,robins2000robust,robins2000marginal,robins2000marginala,robins2008estimation},
respectively, using counting processes and martingale theory.  Our
parameter \(\Psi(P_0)\) is defined via the g-computation formula
\citep{robins1986new} in terms of interventions on the product
integral representing the data-generating distribution.

To estimate \(\Psi(P_0)\), we proceed on the basis of the targeted
minimum loss-based estimation (TMLE) framework \citep{van2006targeted,
  van2011targeted, van2018targeted}. TMLE is a general methodology for
constructing regular and asymptotically linear substitution estimators
for smooth low-dimensional parameters in infinite-dimensional models,
combining flexible ensemble learning and semiparametric efficiency
theory in a two-step procedure. The earliest of the TMLE developments
for estimation of effects of time-varying treatments in longitudinal
data structures (LTMLE) involve full likelihood estimation and
targeting
\citep{van2010targeted,van2010targetedII,stitelman2012general} whereas
the later \citep{van2012targeted,petersen2014targeted} are based on
the techniques of sequential regression originating from
\cite{bang2005doubly}. All these methods rely on a discrete
time-scale.

To construct our continuous-time TMLE algorithm we derive an
expression for the efficient influence curve.  The efficient influence
curve is the canonical gradient of the functional \(\Psi\) and is a
central component in the construction of locally efficient estimators
of a target parameter in general
\citep{bickel1993efficient}. Semiparametric efficiency theory yields
that, given a statistical model and a target parameter, a regular and
asymptotically linear estimator is efficient if and only if its
influence curve is equal to the efficient influence curve. Our
estimation procedure is based on a representation of the target
parameter in terms of intensities and conditional expectations for
which we need initial estimators and a targeting updating
algorithm. We propose such a targeting algorithm based on separate
targeting steps for the intensities and for the conditional
expectations that are iterated until convergence.  The resulting
estimator solves the efficient influence curve equation.

Our TMLE relies on initial estimators for the intensities and the
conditional expectations that constitute our nuisance parameters. The
TMLE framework allows us to take advantage of flexible and
data-adaptive nuisance parameter estimation through super learning
\citep{van2007super}. A super learner is based on a library of
candidate estimators for each nuisance parameter, and uses
cross-validation to select the best combination of estimators. The
general oracle inequality for cross-validation shows that the super
learner performs asymptotically as well as the best combination of
estimators in the library \citep{van2003unified,van2006oracle}. In
particular, we discuss the highly adaptive lasso (HAL)
\citep{van2017generally} for all likelihood components. If this HAL
estimator is included in the library of the super learner used for
initial estimation, we can show that our TMLE estimator is
asymptotically linear and efficient under minimal conditions on the
model \(\mathcal{M}\).

\begin{figure}
  \centering \includegraphics[width=1\textwidth,angle=0]
  {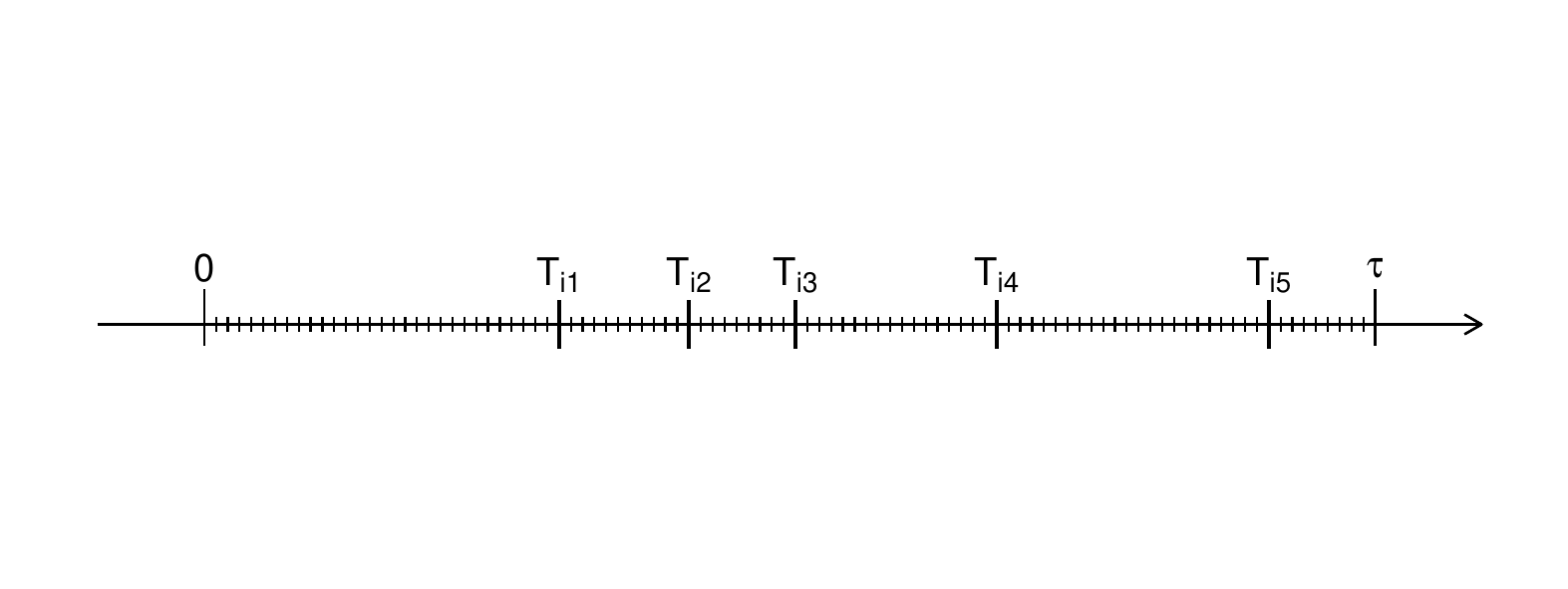}
  \caption{Observations measured on a continuous scale, with five
    random time-points \(T_{i1}, \ldots, T_{i5}\) between time \(0\)
    and time \(\tau\), referred to as monitoring times, where actual
    changes for the given subject \(i\) are measured.}
  \label{fig:time:illustration}
\end{figure}

Our methods extend the existing longitudinal TMLE (LTMLE) to the
continuous\hyp{}time case. While the theory and asymptotic performance
of LTMLE remain valid on any arbitrarily fine time-scale as long as it
is discrete, we note the following. Sequential regression based LTMLE
works by iterating through a sequence of regressions across all
time-points and has been a popular choice over the full
likelihood-based LTMLE that requires modeling of densities of
potentially high-dimensional covariates.  Our framework provides a
unified methodology that covers both the continuous and the
discrete-time case, where intensities of monitoring times are modeled
separately and the regression approach can be used to deal with
high-dimensional covariates.  The substantial difference between LTMLE
and our continuous-time TMLE lies in the limitation of LTMLE that one
needs sufficiently many events at each time-point to fit the
regressions. In contrast, our method can be applied when there are
monitoring times with few events, or just one event.

\subsection{Motivating applications}
\label{sec:motivating:example}

The methods developed here are applicable to a large variety of
problems in pharmacoepidemiology. In this field of research, hazard
ratios are often used as measures of the association of time-dependent
exposure with time-to-event outcomes \cite[see,
e.g.,][]{andersen2013trimethoprim,karim2016comparison,kessing2019depression}.
However, the interpretation of hazard ratios as the measure of causal
treatment effects is hampered for many reasons
\citep{hernan2010hazards,martinussen2018subtleties}. 
Furthermore, the time-dependent Cox model cannot be used in the
presence of time-dependent confounding
\citep{robins1986new}.  In many applications, it is thus
of great interest to formulate and estimate statistical parameters
under hypothetical treatment interventions that have a causal
interpretation under the right assumptions.

Specifically, it may be of interest to assess the % causal
effect of a (dynamic) drug treatment regime on the \(\tau\)-year risk
of death. As an example, let \(N^a(t)\) denote the process counting
visits to a medical doctor who can prescribe the drug. Let the process
\(A(t)\) be information on the type and dose level of the drug
prescribed at patient-specific times
\(T_1^a < T_2^a< \ldots < T^a_{N^{a}({\tmax})}\) during the study
period \([0,\tau]\). Further, let \(N^{\ell}(t)\) denote the process
counting doctor visits where health information \(L(t)\) is
collected. At visit \(T^{a}_{k}\), the doctor considers the baseline
characteristics \(L_0\) of the patient, the treatment so far
\(\{(N^{a}(t), A(t)):t\in[0,T^{a}_{k-1}]\}\), and evaluates changes of
covariates \(\{(N^{\ell}(t), L(t)):t\in[0,T^{a}_k)\}\) to decide to
continue or to change the treatment. We are interested in the result
of a hypothetical experiment that we would have liked to have
conducted but did not. In our example, the causal effect resulting
from such a hypothetical experiment could be the difference of the
\(\tau\)-year risk of death under different (dynamic) treatment
regimes.

A treatment regime is any a priori defined rule which can be applied
at doctor visits during the study period to decide if the treatment
should be changed or continued. For example, the intervention could
state that the patient stays on an initially randomized drug
throughout the entire study period. A dynamic treatment regime is an
example of an adaptive intervention which allows the decision to
depend on the current history of the patient.  The rule could be that
whenever the measurement of a blood marker value exceeds a certain
threshold the dose level of the drug should be adapted. By comparison
of the effects of different regimes we may further learn optimal
strategies for treatment interventions based on patients' medical
past.  In this article, we focus on interventions of the treatment
process keeping the doctor visits where they naturally occur.

The data analysis is complicated by the fact that the observed data
are subject to time-dependent confounding: At any point in time,
doctors and patients make treatment decisions for reasons depending on
treatment and covariate history. Furthermore, past treatment may
further influence future values of covariates, such as biomarkers and
diagnoses, and future treatment decisions. Additionally, a proportion
of subjects in the population may be lost to follow-up (censored)
which can likewise depend on prior treatment decisions and covariate
values. Our estimation framework allows us to control for the history
of all observed time-varying covariates and treatment choices that we
believe could be predictive of both treatment decisions and censoring,
and of the risk of death.  Importantly, in the case where neither the
treatment decision nor the censoring mechanism depend on any
unrecorded health information, i.e., the observed history at any point
in time is sufficient to predict the next treatment decision and the
censoring at that time, the no unmeasured confounding assumption
holds, and the estimated effects of hypothetical interventions based
on the observed data can be interpreted causally.

\subsection{Organization of article}

The article is organized as follows. We first define the statistical
estimation problem. Section \ref{sec:setting:notation} introduces the
continuous-time longitudinal setting and presents the model of a
single subject along with the likelihood. In particular, Section
\ref{sec:interventions} defines interventions on the likelihood,
Section \ref{sec:target:parameter} defines the target parameter, and
Section \ref{sec:eff:influence:function} provides the efficient
influence function for the estimation problem.  In Section
\ref{sec:causal}, we briefly review the assumptions under which the
target parameter identifies the causal effect of interest.  Section
\ref{sec:target:parameter:parametrization} presents a representation
of the target parameter in terms of intensities and conditional
expectations for which we need to construct an estimation
procedure. Section \ref{sec:eff:thm:tmle} states a general theorem for
asymptotically efficient substitution estimation.  Section
\ref{sec:motivating:tmle} introduces targeted minimum loss-based
estimation (TMLE) for the considered estimation problem. Section
\ref{sec:initial:estimation} describes initial estimation of the
likelihood components needed for the targeting updating steps and for
estimation of the target parameter that fulfills the criteria of the
efficiency theorem.  Section \ref{sec:TMLE} presents a particular TMLE
algorithm for estimation of the target parameter in the
continuous-time setting that involves separate targeting for
conditional expectations and intensities, pooled over time.  Section
\ref{sec:TMLE:inference} reviews inference for the TMLE estimator.
Section \ref{sec:simulation:study} presents the results of a
simulation study as a demonstration of the methods and a
proof-of-concept. Section \ref{sec:discussion:future} closes with a
discussion.

\section{Formulation of the statistical estimation problem}
\label{sec:setting:notation}

We start out defining the data structure, the statistical estimation
problem and the target parameter.  An overview of our notation can be
found in Appendix \ref{app:sub:overview:notation}.

\subsection{Notation and setup}

We represent the subject-specific information in terms of a counting
process model \citep{andersen2012statistical}. Suppose \(n\in\N\)
subjects of a population are followed in a bounded interval of time
\([0, \tmax]\). Let \((N^a ,N^{\ell} , N^c, N^d)\) be a multivariate
counting process generating random times at which treatment,
covariates, censoring status and survival status may change. At jump
times of \(N^a\) we observe changes in a treatment regime \(A(t)\)
taking values in finite set \(\mathcal A\) and at jump times of
\(N^{\ell}\) we observe changes of a covariate vector \(L(t)\) with
values in a compact subset of \(\R^d\). The processes \(N^c\) and
\(N^d\) generate changes in censoring status and death status,
respectively. Furthermore, \(L_0\) denotes a vector of baseline
covariates measured at time \(0\), and \(Y\) a real valued outcome
variable measured at time \(\tmax\). We assume that there are no
events at time zero, and that realizations of all processes are càdlàg
functions on \([0,\tmax]\). Let \((\F_t)_{t\ge 0 }\) denote the
filtration generated by the history of the observed processes up to
time \(t\). Specifically, we have \(\F_0=\sigma(L_0)\).

Let \(T_1^a < T_2^a< \ldots < T^a_{N^{a}({\tmax})}\) and
\(T_1^{\ell} < T_2^{\ell}< \ldots < T^{\ell}_{N^{\ell}({\tmax})} \) be
the random times at which the treatment regime \(A\) and the covariate
process \(L\) may change, and let \(T^c\) and \(T^d\) be the
right-censoring time and the survival time,
respectively. By definition, \(N^{a}(t)\) and
\(N^{\ell}(t)\) are the subject-specific counts of treatment and
covariate monitoring times in \([0,t]\).  The actual end
of follow-up for a subject is \(\min(T^c,T^d,\tmax)\) and the total
number of unique event times before time \(t\) is,
\begin{equation}
  K(t)=
  \# \big\lbrace\lbrace T_1^a, \ldots, T^a_{N^{a}(t)}\rbrace \cup \lbrace T_1^{\ell},
  \ldots, T^{\ell}_{N^{\ell}(t)}\rbrace \cup  \lbrace T^c, T^d\rbrace\big\rbrace.
    \label{eq:K:t}
\end{equation}
For each subject, the number of change points of the multivariate
counting process on \([0,\tmax]\), \(K=K({\tau})\), is finite.

For \(x=a,\ell,d,c\), we denote by \(\Lambda_0^x\) the cumulative
intensity that characterizes the compensator of \(N^x\). We further
denote by \(M^x= N^x-\Lambda_0^x\) the corresponding
martingale.  Heuristically, we have that
\begin{align*}
  P( N^x(dt)=1 \, \vert \, \F_{t-}) = \EE[ N^x(dt) \, \vert \,
  \F_{t-}]
  = d\Lambda_0^x (t \,
  \vert \, \F_{t-}),
\end{align*}
where the increment \( N^x(dt) \) is non-zero and equal to 1 if and
only if there is a jump of \(N^x\) in the infinitesimal interval
\([t,t+dt)\) \citep{andersen2012statistical,gill1994lectures}.

We assume that the processes \(A\) and \(L\) only change at monitoring
times. The distribution of the treatment \(A(t)\) at any time \(t\)
where \(N^a(t)\) jumps is denoted
\(\pi_{0,t}(a \,|\, \F_{t-} ) =P(A(t)=a \,|\, \F_{t-})\),
\(a\in\mathcal{A}\), and the distribution of covariates \(L(t)\) at
any time \(t\) where \(N^{\ell}(t)\) jumps is characterized by a
conditional density \(\mu_{0,t}(\ell \, \vert\, \F_{t-} ) \) with
respect to a dominating measure \(\nu_L\).  We assume that
\(A(t)=A(t-)\) and \(L(t)=L(t-)\) otherwise.  Lastly, the probability
distribution of baseline covariates \(L_0\) is characterized by the
conditional density \(\mu_{0,L_0}\) with respect to a dominating
measure \(\nu_{L_0}\).

\subsubsection{Observations}
\label{sec:obs}

For subject \(i\), with \(i=1,\dots,n\) independent subjects, let
\(\{T_{i,k}\}_{k=1}^{K_{i}(t)}\) denote the ordered set of unique
event times up to time \(t\). The observed data for subject \(i\) in a
bounded interval \([0,t]\), \(t \le \tau\), is given by:
\begin{align}
  \bar{O}_i(t) = \big\lbrace \big(L_{0,i}, s, N_i^a(s),A_i(s),N_i^{\ell}(s)
  ,L_i(s),N_i^d(s),N_i^c(s)\big) \, : \, s\in \{T_{i,k}\}_{k=1}^{K_{i}(t)}
  \big\rbrace.
  \label{eq:observed:data}
\end{align}
We also use the shorthand notation \(O_i=\bar{O}_i(\tmax)\) and denote
by \(\mathcal{O}\) the space where \( O_i\) takes its values. Let
\( \mathbb{P}_n\) denote the empirical distribution of the data
\(\{O_i\}_{i=1}^n\).  For a dataset with \(n\) observations, we
use,
\begin{align}
  \qquad 0={t}_0 <{t}_1 < \cdots <{t}_{{K}_n} ,
  \label{eq:observed:event:times}
\end{align}
with \({K}_n= \sum_{i=1}^n K_{i}\), to denote the ordered sequence of
unique times of changes
\(\cup_{i=1}^n \cup_{k=1}^{K_{i}}\{T_{i,k}\} \).

The distribution \(P_0\) of the observed data \(O\) factorizes
according to the time-ordering, going from one infinitesimal time
interval to the next with \( d\Lambda_0^x (t \, \vert \, \F_{t-})\)
representing the conditional probability of an event of \(N^x\) in
\([t,t+dt)\), for \(x=a, \ell, d,c\) \citep[][Section
II.7]{andersen2012statistical}. Accordingly, we express the likelihood
as
\begin{align}
  \begin{split} 
    dP_0(O) = \qquad\qquad\qquad\qquad\qquad\qquad\qquad\qquad
    \qquad\qquad\qquad\qquad\qquad\qquad\qquad \\ \mu_{0,L_0}(L_0)
    \Prodi_{t \in (0,\tmax]} \big( d\Lambda_0^{\ell} (t \, \vert \,
    \F_{t-}) \, \mu_{0,t} (L(t)\, \vert \, \F_{t-})
    \big)^{N^{\ell}(dt)} \big( 1-d\Lambda_0^{\ell} (t \, \vert \,
    \F_{t-})  \big)^{1-N^{\ell}(dt)}\\
    \Prodi_{t \in (0,\tmax]} \big( d\Lambda_0^{a} (t \, \vert \,
    \F_{t-})  \, \pi_{0,t} (A(t)\, \vert \, \F_{t-}) \big)^{N^{a}(dt)}
    \big( 1-d\Lambda_0^{a} (t \, \vert \,
    \F_{t-}) \big)^{1-N^{a}(dt)}\\
    \Prodi_{t \in (0, \tmax]} \big(d\Lambda_0^{c} (t \, \vert \,
    \F_{t-})\big)^{N^c(dt)} \big( 1 - d\Lambda_0^{c} (t \, \vert \,
    \F_{t-}) \big)^{1-N^c(dt)} \\
    \Prodi_{t \in (0, \tmax]} \big(d\Lambda_0^{d} (t \, \vert \,
    \F_{t-})\big)^{N^d(dt)} \big( 1 - d\Lambda_0^{d} (t \, \vert \,
    \F_{t-}) \big)^{1-N^d(dt)},
    \end{split}\label{eq:like}
\end{align}
with \( \prodi\) denoting the product integral \citep[][Section
II.6]{gill1990survey,andersen2012statistical}. A
particularly nice aspect of the product-integral representation is
that it gives a unified presentation for the discrete and the
continuous time case.

\subsection{Interventions}
\label{sec:interventions}

We are interested in estimating the effect of dynamic treatment
regimes
\citep{robins2002analytic,murphy2001marginal,hernan2006comparison,van2007causal}
corresponding to hypothetical experiments, under hypothetical
interventions, where data had been generated differently. An
intervention defines a rule specifying treatment at each intervention
time point given the data so far.  In our setting, we allow the number
and schedule of the intervention time-points to be subject-specific
and to occur in continuous time. We thus distinguish between
interventions that control the treatment decision mechanism, but not
the conditional distribution governing the schedule of the
intervention time points, and interventions that control the
distribution of treatment decisions as well as the distribution of the
intervention times.   In our motivating
pharmacoepidemiological applications from Section
\ref{sec:motivating:example}, an intervention only on the treatment
decision mechanism could for example specify a treatment regime where
a particular drug treatment should be continued or discontinued, while
the frequency of doctor visits are kept as they would naturally occur
for the subjects of the population had they followed that treatment
regime.  In addition to the treatment regimes of interest,
interventions always control the censoring mechanism, such that, in
the hypothetical experiment, all subjects are followed for the entire
study period \([0,\tau]\).

The observed data are generated by the distribution \(P_0\) which
factorizes as displayed in \eqref{eq:like}.  We define interventions
directly on \(P_0\), by replacing a subset of its components by an
intervention-specific choice. To formulate this, we decompose the
observed data distribution \(P_0\) into two parts which we refer to as
the interventional part (\(G_0\)) and the non-interventional part
(\(Q_0\)), respectively. We parametrize \(P_0\) accordingly,
\begin{align*}
  d{P}_0=d{P_{Q_0,G_0}} = \Prodi_{t \in [0, \tmax]} dQ_{0,t} \, dG_{0,t},
\end{align*}
with \(dG_{0,t}(o)\) and \(dQ_{0,t}(o)\) denoting the conditional
measures, given the observed history, corresponding to the
interventional part and the non\hyp{}interventional part at time
\(t\), respectively.  We use \(g_{0,t}\) to denote the density of
\(G_{0,t}\) and likewise \(q_{0,t}\) to denote the density of
\(Q_{0,t}\), both with respect to appropriate dominating measures.

Now, an intervention involves replacing \(G_0\) by some \(G^*\)
encoding how treatment and censoring is generated conditional on the
available history in the hypothetical experiment.  In its generality,
this is what is referred to as a \textit{randomized plan} in
\citet[][Sections 6 and 7]{gill2001causal}, or a \textit{stochastic}
intervention \citep{robins2004effects,dawid2010identifying}, but it
includes \textit{static} and \textit{dynamic} interventions
\citep{hernan2006comparison,chakraborty2013statistical} plans as
special cases as we explain below.  Which components we include in
\(G_0\), and thus intervene upon, depends on what kinds of effects we
are interested in and thus what scientific question we wish to
address. Consider the following options.

\begin{defi}[Interventions on treatment assigned]
  An intervention on treatment assigned involves replacing
  \(\pi_{0,t}\) by some choice \(\pi_{t}^*\). Thus, the interventional
  part includes the treatment and the censoring mechanism:
  \begin{align}
  \begin{split}
  dG_{0,t}(O) = \big(  \pi_{0,t} (A(t)\, \vert \, \F_{t-})
  \big)^{N^{a}(dt)} \big(d\Lambda_0^{c} (t \, \vert \,
  \F_{t-}) \big)^{N^c(dt)} \\
  \big( 1 - d\Lambda_0^{c} (t \, \vert \,
  \F_{t-}) \big)^{1-N^c(dt)}.
  \end{split}
  \label{eq:g:0}
\end{align}
The intervention prevents censoring and specifies the treatment regime
\(\pi_t^*\), such that
\begin{equation*}
  dG_t^*(O) = \big( 1- N^c(t)) \,\pi_{t}^* (A(t)\, \vert \, \F_{t-}).
\end{equation*}
The distribution of the intervention times is not intervened upon,
i.e., \(\Lambda_0^{a}\) is included in the non-interventional part.
\label{defi:treatment:assigned}
\end{defi}

When the interventional treatment distributions are degenerated and,
for example, set to a single value \(a^* \in\mathcal{A}\) throughout
the entire study period,
\begin{align}
  \pi_{t}^* (A(t)\, \vert \, \F_{t-}) = \1 \lbrace A(t)=a^*\rbrace .
  \label{eq:static} 
\end{align}
We refer to \(\pi_{t}^*\) as a static intervention since it
deterministically sets \(A(t)=a^*\) (at treatment monitoring times). A
dynamic intervention, on the other hand, defines \(\pi^*_t\) that
assigns \(A(t)\) deterministically dependent on the subject's observed
past.

\begin{defi}[Intervention on treatment and schedule]
  The interventional part includes the treatment, the schedule of
  intervention times, and the censoring mechanism:
  \begin{align*}
  \begin{split}
  dG_{0,t}(O) =  \big( d\Lambda_0^{a} (t \, \vert \,
    \F_{t-})  \, \pi_{0,t} (A(t)\, \vert \, \F_{t-}) \big)^{N^{a}(dt)}
    \big( 1-d\Lambda_0^{a} (t \, \vert \,
    \F_{t-}) \big)^{1-N^{a}(dt)} \\
    \big(d\Lambda_0^{c} (t \, \vert \,
  \F_{t-}) \big)^{N^c(dt)} 
  \big( 1 - d\Lambda_0^{c} (t \, \vert \,
  \F_{t-}) \big)^{1-N^c(dt)}.
  \end{split}
\end{align*}
The intervention prevents censoring and specifies a treatment regime
by \(\pi_t^*\) and \(\Lambda^{a,*}\), such that,
\begin{align*}
  dG^*_{t}(O) = \big( d\Lambda^{a,*} (t \, \vert \,
    \F_{t-})  \, \pi^*_{t} (A(t)\, \vert \, \F_{t-}) \big)^{N^{a}(dt)}
    \big( 1-d\Lambda^{a,*} (t \, \vert \,
    \F_{t-}) \big)^{1-N^{a}(dt)}.
\end{align*}
 \label{defi:treatment:times}
\end{defi}
An intervention on the schedule of treatment decisions involves
replacing the intensity \(\Lambda_0^a\) by some choice
\(\Lambda^{a,*}\). In the context of the applications described in
Section \ref{sec:motivating:example}, this could be to decrease or
increase the frequency of doctor visits, for instance to ensure at
least a monthly visit.

To focus our presentation, we continue with \(dG_{0,t}\) as defined by
\eqref{eq:g:0} according to Definition
\ref{defi:treatment:assigned}, considering interventions only on
the treatment decision \(\pi_t\) and on the censoring mechanism
\(\Lambda^c\).  Note that this defines the non-interventional part as
\begin{align*}
  dQ_{0,t}(O) &=  \big( d\Lambda_0^{a} (t \, \vert \,
                \F_{t-}) 
                \big)^{N^{a}(dt)} \big( 1-d\Lambda_0^{a} (t \, \vert \,
                \F_{t-}) \big)^{1-N^{a}(dt)}\\
              & \qquad  \big( d\Lambda_0^{\ell} (t \, \vert \,
                \F_{t-}) \, \mu_{0,t} (L(t)\, \vert \,
                \F_{t-}) \big)^{N^{\ell}(dt)} \big(
                1-d\Lambda_0^{\ell} (t \, \vert \,
                \F_{t-}) \big)^{1-N^{\ell}(dt)}\\
              &\qquad  \big(d\Lambda_0^{d} (t \, \vert \,
                \F_{t-})\big)^{N^d(dt)} \big( 1 - d\Lambda_0^{d} (t \, \vert \,
                \F_{t-}) \big)^{1-N^d(dt)},
\end{align*}
for \(t>0\).

\subsection{Statistical model}
\label{sec:statistical:model}

Let \(\mathcal{Q}\) denote the parameter set for the
non\:-\:interventional part \(Q= ( \mu_{t}, {\Lambda}^{a}(t),\)
\(\Lambda^{\ell}(t), \Lambda^d(t) )_{t\in [0,\tau]}\) and
\(\mathcal{G}\) the parameter set for the interventional part
\(G= ( \pi_{t}, {\Lambda}^{c}(t) )_{t\in [0,\tau]}\). We consider a
statistical model \(\mathcal{M}\) as follows:
\begin{align}
  \mathcal{M} = \bigg\lbrace P \, : \, dP = dP_{Q,G}= \Prodi_{t \in [0, \tmax]}  dQ_t dG_t ,\,  \, 
  G \in \mathcal{G}, \, Q \in \mathcal{Q}\bigg\rbrace . 
  \label{eq:statistical:model}
\end{align}
In Section \ref{sec:initial:estimation} we summarize results that
require \(Q,G\) to be contained in the set of \cadlag functions, i.e.,
functions that are right-continuous with left-hand limits, with
bounded sectional variation norm.  For now we leave
\(\mathcal{Q},\mathcal{G}\) unspecified.

\subsection{Target parameter}
\label{sec:target:parameter}

Suppose an intervention \(G^*\) is given.  We define the
post-interventional distribution for any \(P\in\mathcal{M}\) by
replacing \(G\) by \(G^*\) in \(P_{Q,G}\). The resulting \(P_{Q,G^*}\)
is commonly referred to as the \textit{g-computation formula}
\citep{robins1986new,gill2001causal}. We will also denote this by
\(P^{G^*}\). Based on the data \(O=\bar{O}(\tmax)\), our overall aim
is to estimate the expectation of the outcome of interest \(Y\) under
the g-computation formula. That is, we are interested in the parameter
\(\Psi^{G^*} \, : \, \mathcal{M} \rightarrow \R\) given by
\begin{align}
  \Psi^{G^*} (P)  = \EE_{P^{G^*}} \big[ Y \big]
  = \int_{{\mathcal{O}}} y \, \Prodi_{t \in [0, \tmax]} dQ_{ t}(o) \, dG_t^*(o),
  \label{eq:target:parameter}
\end{align}
where the notation \(\EE_{P^{G^*}}\) refers to the expectation
operator with respect to the post-interventional measure
\(P^{G^*}= P_{Q,G^*}\).

In this paper we focus on an all-cause mortality outcome,
\(Y= {N}^d(\tmax)\), so that \eqref{eq:target:parameter} is the
expected risk of dying by time \(\tau\) under the distribution defined
by the g-computation formula. We emphasize, however, that in principle
\(Y\) can be defined as any mapping of the observed past \(\F_{\tau}\)
as long as \(Y\) takes value in a compact set. We denote
by \(\psi_0 := \Psi^{G^*} (P_0)\), the true value of the target
parameter.  The target parameter is identifiable from the observed
data under the following positivity assumption.

\begin{assumption}[Positivity]
  We assume absolute continuity of \(P_{Q,{G}^{*}}\) with respect to
  \(P_{Q,{G}}\), i.e., \(P_{Q,{G}^{*}} \ll P_{Q,{G}}\). This implies
  existence of the Radon-Nikodym derivative
  \(dP_{Q,{G}^{*}} / dP_{Q,{G}} = \prodi_{s < \tau} dG^*_s /
  dG_s\). 
  \label{ass:identifiability}
\end{assumption}

The g-computation formula arising from replacing the observed \(G_0\)
by \(G^*\) and the resulting target parameter in
\eqref{eq:target:parameter} are well-defined statistical quantities by
Assumption \ref{ass:identifiability}.

\subsubsection{Causal parameter and causal interpretability}
\label{sec:causal}

Causal interpretability of the g-computation formula in the setting of
discrete data is provided by Robins' work
\citep{robins1986new,robins1987addendum,robins1989analysis} under the
assumptions of no unmeasured confounding and positivity \cite[for a
nice review, see,][Part III]{hernanrobins}. This work is further
generalized by \cite{gill2001causal} to continuously varying
covariates and treatments. Our setting differs from theirs in that
subjects are measured at random times on a continuous scale: Our
g-computation formula is represented as a product integral over times
where something actually happens, whereas the g-computation formula of
\cite{gill2001causal} consists of a finite product over times of a
discrete grid.  Nevertheless, note that the counting processes only
have finitely many changes in the compact time interval \([0,\tau]\),
and that interventions on the treatment decision are in fact only
applied at a finite number of (random) times.  Thus, the `traditional'
causal assumptions as stated by \cite{gill2001causal} can be applied
at the random times to ensure the causal interpretation of the
g-computation formula \citep[][Theorem
2]{gill2001causal}.

Particularly, consider \(G^*\) that defines an intervention on the
treatment assigned according to a distribution \(\pi^*_t\). Let
\(O^{G^*}\) be the counterfactual random variable representing the
data that would have been observed, had \(G^*\) been adhered to rather
than the factual \(G\) throughout the follow-up period. The causal
parameter of interest is \(\EE [ Y^{G^*}]\), the expected value of
\(Y\) when imposing the intervention \(G^*\). For our setting with
random treatment monitoring times,
\(T_1^a < T_2^a< \ldots < T^a_{N^{a}({\tmax})}\), we formulate the no
unmeasured confounding assumption for the treatment decision as
follows:

\begin{assumption}[No unmeasured confounding]
  \(Y^{G^*} \! \independent  A(T^a_k) \, \vert \, \bar{O}(T^a_k-)\), for all
  \(k=1,\ldots, N^{a}({\tmax})\). 
  \label{ass:sra}
\end{assumption}

Assumption \ref{ass:identifiability} and Assumption \ref{ass:sra}
together yield that our target parameter identifies the expected value
of the counterfactual outcome, i.e.,
\(\Psi^{G^*} (P) =\EE [ Y^{G^*}]\), had the treatment decisions been
governed by \(\pi^*_t\).

In settings where one intervenes not only on the treatment decisions
but also on the mechanism of the timing of treatment monitoring, an
additional assumption equivalent to Assumption \ref{ass:sra} is
needed.  Specifically, causal interpretability of \(\Psi^{G^*} (P) \)
is in that case achieved if the counterfactual outcome \(Y^{G^*}\) is
independent of the time to the next treatment monitoring event
conditional on the observed history up to any event time \(T_k\).

\subsection{Canonical gradient}
\label{sec:eff:influence:function}

The canonical gradient of the pathwise derivative of the target
parameter characterizes the information bound of the estimation
problem relative to the statistical model \(\mathcal{M}\)
\citep{bickel1993efficient}.  The canonical gradient is also known as
the \textit{efficient influence curve}.  The following theorem
provides a representation of the canonical gradient for our
statistical model \(\mathcal{M}\) and target parameter
\(\Psi^{G^*} \, : \, \mathcal{M} \rightarrow \R\).  We will use this
representation to construct an asymptotically efficient estimator.

We follow a general recipe to derive our expression for the canonical
gradient. For this, consider the submodel
\(\mathcal{M}_G \subset \mathcal{M}\) that assumes \(G\) to be known
and let \(\mathscr{T}_G(P)\) denote the tangent space at \(P\) in the
submodel \(\mathcal{M}_G\). The canonical gradient can be found as the
projection of any other gradient of the pathwise derivative of
\(\Psi \, : \, \mathcal{M}_G\rightarrow \R\) at \(P\) onto the tangent
space \(\mathscr{T}_G(P)\) \citep{vanRobins2003unified}. Thus, we need
to characterize the tangent space \(\mathscr{T}_G(P)\) and construct
an initial gradient to project onto \(\mathscr{T}_G(P)\). Note that
the tangent space \(\mathscr{T}_G(P)\) is generated by fluctuations
only of the conditional distributions of \(Q\) since \(G\) is
known. To construct an initial gradient, we can use the influence
curve of an inverse probability weighted estimator of \(\psi_0\) in
\(\mathcal{M}_G\). The detailed derivations can be found in Appendix
\ref{app:canonical:gradient}.

\begin{thm}[Canonical gradient] The canonical gradient \(D^*(P)\) at
  \(P\) in \(\mathcal{M}\) can be represented as follows:
\begin{align*}
  &D^*(P) =  \EE_{P^{G^*}} [ Y \, \vert \,  \F_0] - \Psi^{G^*} (P) \\
  & \,\,\, + \,  \int_0^{\tmax}\Prodi_{s <t } \frac{dG^*_s}{dG_s}
    \Big(  \EE_{P^{G^*}} [Y \, \vert\,  L(t), N^{\ell}(t)
    , \F_{t-}] -
    \EE_{P^{G^*}} [Y \, \vert\,  N^{\ell}(t),
    \F_{t-}] \Big) N^{\ell}(dt)   \\
  & \, + \, 
    \int_0^{\tmax}\Prodi_{s <t } \frac{dG^*_s}{dG_s}
    \Big(\EE_{P^{G^*}} [Y \, \vert\, \Delta N^{\ell}(t)=1, \F_{t-}]  - 
    \EE_{P^{G^*}} [Y \, \vert\,  \Delta N^{\ell}(t)=0,  \F_{t-}]\Big)
    M^{\ell}(dt) 
  \\
  & \, + \, 
    \int_0^{\tmax}\Prodi_{s <t } \frac{dG^*_s}{dG_s}
    \Big(  \EE_{P^{G^*}} [Y \, \vert\, \Delta  N^{a}(t)=1, \F_{t-}] - 
    \EE_{P^{G^*}} [Y \, \vert\, \Delta N^{a}(t)=0,  \F_{t-}]\Big)
    M^{a}(dt) 
  \\
  & \, + \, 
    \int_0^{\tmax}\Prodi_{s <t } \frac{dG^*_s}{dG_s}
    \Big(
    1 -  \EE_{P^{G^*}} [Y \, \vert\,  N^{d}(t)=0,  \F_{t-}] \Big)
    M^{d}(dt) . 
\end{align*}
We here follow notation of \cite{andersen2012statistical} and use
\(\Delta\) to denote the difference operator defined by
\(\Delta X = X-X_{-} \) for a \cadlag process \(X\).
\label{thm:eff:ic}
\end{thm}

\section{Representation of the target parameter by iterated
  expectations}
\label{sec:target:parameter:parametrization}

Estimation of the target parameter requires evaluation of a large
integral.  We here present a parametrization of the target parameter
in terms of a nested sequence of conditional expectations which is
central for our estimation procedure. As for the discrete-time
analogue \citep{bang2005doubly,robins2000robust}, the parametrization
is defined backwards through time, starting at the end of the study
period \(\tau\).  The main difference to the discrete-time
representation is that the time-points where events happen are random.

The notation that we present in this section will be used throughout
the remainder of the paper.  For \(P\in\mathcal{M}\) and a fixed
regime \(G^*\) according to Definition \ref{defi:treatment:assigned},
we define for \(t\in [0,\tau]\):
\begin{align}
  Z_{t}^{G^*}& :=
               \EE_{P^{G^*}} \big[  Y \, \big\vert
               \, L(t),  N^{\ell}(t),N^a(t),
               N^d(t), \F_{t-}\big] =
               \int
               Y \,\Prodi_{s \ge t} dG^*_s\, \Prodi_{s > t}dQ_s  .
               \label{eq:def:Zt}
\end{align}
Note that the conditioning set of \(Z_{t}^{G^*}\) excludes the
interventional part at time \(t\).  We further denote the conditional
expectation where \(L(t)\) has been integrated out by:
\begin{align}
  \begin{split}
    Z_{t,L(t)}^{G^*} := & \, \EE_{P^{G^*}} \big[ Y \, \big\vert \,
    N^{\ell}(t),N^a(t),
    N^d(t), \F_{t-}\big] \\[0.3em]
    =& \,\EE_{P^{G^*}} \big[ Z_{t}^{G^*} \, \big\vert \,
    N^{\ell}(t),N^a(t),
    N^d(t), \F_{t-}\big]\\
    =& \, \int \bigg( \int Y \,\Prodi_{s \ge t} dG^*_s\, \Prodi_{s >
      t}dQ_s \bigg) d\mu_{0,t} (L(t) \, \vert \,\F_{t-}). \end{split}
               \label{eq:def:Zt:Lt}
\end{align} 
The notation using the subscript `\(L(t)\)' in \(Z_{t,L(t)}^{G^*}\) to
refer to \(L(t)\) being integrated out follows the notation of
\cite{van2012targeted}.

\begin{lemma} 
  Define, for any \(P\in\mathcal{M}\):
\begin{align}
  {\bm{Z}} =  {\bm{Z}}(P) :=  \big( {Z}^{G^*}_{t}, {Z}^{G^*}_{t,L(t)},
  {\Lambda}^{\ell}(t), {\Lambda}^{a}(t), {\Lambda}^{d}(t)
  \big)_{t\in [0,\tmax]}  . 
  \label{eq:rel:part:Z}
\end{align}
Particularly, let \(\bm{Z}_0 = {\bm{Z}}(P_0) \).  The target parameter
\(\Psi^{G^*}\) defined in \eqref{eq:target:parameter} can be
represented as a functional of \(P\) only through \(\bm{Z}\).
\label{lemma:representation:Psi:Z}
\end{lemma}

\begin{proof}
See Appendix \ref{app:proof:of:lemma:Psi:Z}.
\end{proof}

In line with the sequential regression representation of
\cite{bang2005doubly}, we may  thus  utilize that
\(\bm{Z}\) rather than \(P\) itself can be used to evaluate the target
parameter.  Our aim is to estimate \(\bm{Z}\) in an optimal way. The
canonical gradient presented in Theorem \ref{thm:eff:ic} can be
represented as a functional of \(\bm{Z}\) and \(G\); this will guide
the construction of estimators for \(\bm{Z}\) as a result of our
efficiency theorem in Section \ref{sec:eff:thm:tmle}.

\section{Efficiency theorem for substitution estimation}
\label{sec:eff:thm:tmle}

In the previous section we have demonstrated that the target parameter
can be represented as a functional of \(P\in\mathcal{M}\) through
\(\bm{Z}\). We define a substitution estimator of the target parameter
\(\psi_0 = \Psi( P_0)\)  based on an estimator
\(\hat{\bm{Z}}_n\) for \(\bm{Z}_0\).  The following theorem states the
general conditions for asymptotic efficiency of such substitution
estimation of \(\psi_0 \). Particularly, efficient estimation requires
estimation of the relevant parts of the data-generating distribution
that enter the expression of the efficient influence curve. Thus, an
efficient estimator \(\hat{P}_n \) is characterized by both
\(\hat{\bm{Z}}_n\) and \( \hat{G}_n\).  The proof of the theorem
follows proofs in similar work \citep[see, e.g.,][Theorem
A.5]{van2006targeted,van2017generally,van2011targeted}, is short and
relies on an analysis of the separate terms of a von Mises expansion
of the target parameter. In Section \ref{sec:initial:estimation} we
summarize results on initial estimators that meet the regularity
conditions under minimal smoothness conditions on
\(\mathcal{M}\). Throughout we use the notation \(P f = \int f dP\)
and \( \Vert f \Vert_{P} = \sqrt{ P f^2}\).

We show in Appendix \ref{sec:remainder} that the difference
\(\Psi^{G^*}(P) - \Psi^{G^*}(P_0)\) admits the following presentation:
\begin{align}
  \Psi^{G^*}(P) - \Psi^{G^*}(P_0) = - P_0   D^*(P) + R_2(P,P_0), \qquad
  P \in \mathcal{M},
  \label{eq:rep:R2}
\end{align}
where \(D^*(P)\) is the canonical gradient and the second-order
remainder \(R_2(P,P_0) \) is given by:
\begin{align}
  \begin{split}
    & R_2(P,P_0)
    = \Psi^{G^*}(P) - \Psi^{G^*}(P_0) + P_0 D^*(P) \\
    &\,= \int_0^{\tmax} \int_{\mathcal{O}} Y \,
    \frac{1}{{\bar{g}}_{t}} \big( \bar{g}_{0,t} - {\bar{g}}_{t} \big)
    \Prodi_{s \le \tmax} dG_{s}^* \Prodi_{s < t} dQ_{0,s} \,\big( d
    Q_{0,t} - d{Q}_{t} \big) \Prodi_{s >t} d{Q}_{s} .
  \end{split}
      \label{eq:second:order}
\end{align}
Here we have used the notation:
\begin{align*}
  \bar{g}_{t} = \Prodi_{s < t} g_s, \quad
  \text{and,} \quad \bar{g}_{0,t} = \Prodi_{s < t} g_{0,s},
\end{align*}
with \( {g}_{t}\) being the density of the interventional part
\(dG_t\) as defined in Section \ref{sec:interventions}.

\begin{thm}
  Consider an estimator \(\hat{P}_n\) for \(P_0\), such that
\begin{align}
  \mathbb{P}_n D^*(\hat{P}_n
  ) = o_P(n^{-1/2}).
  \label{eq:key:EIC:eq}
\end{align}
If the following conditions 1 and 2 hold true,
\begin{enumerate}[label={Condition }{{\arabic*}}),
  leftmargin=\widthof{[condition xx]}+\labelsep]
\item \(R_2(\hat{P}_n, P_0 ) = o_P(n^{-1/2})\),\vspace{0.2cm}
\item \(D^*(\hat{P}_n )\) belongs to a Donsker class, and
  \(P_0\big( D^*(\hat{P}_n ) - D^*(P_0 )\big)^2\) converges to zero in
  probability,
\end{enumerate}
then
\begin{align*}
  \Psi^{G^*} (\hat{P}
  _n) -  \Psi^{G^*} (P
  _0) =\mathbb{P}_n D^* (P
  _0) + o_P(n^{-1/2}),
\end{align*}
that is, \( \Psi^{G^*} (\hat{P }_n)\) is asymptotically linear at
\(P_0\) with influence curve \(D^*(P_0)\) and is thus asymptotically
efficient among all locally regular estimators at \(P_0\).
\label{thm:eff:estimator}
\end{thm}

\begin{proof} 
  The proof of the theorem relies on the expansion
  \eqref{eq:rep:R2}. Applying the representation at \(\hat{P}_n^*\)
  and using Equation \eqref{eq:key:EIC:eq} from the theorem yields
\begin{align*}
  & \Psi^{G*}(\hat{P
    }_n) - \Psi^{G*}(P
    _0)
    =  {(\mathbb{P}_n - P_0) D^*(\hat{P
    }_n
    )} +
    R_2 (\hat{P
    }_n
    , P
    _0
    ) + o_P(n^{-1/2})
    . 
\end{align*}
It is now a result of empirical process theory \citep[see,
e.g.,][Lemma 19.24]{van2000asymptotic} that condition 2 implies
\begin{align*}
  {(\mathbb{P}_n - P_0) \big( D^*(\hat{P
  }_n
  ) -
  D^*({P
  }_0
  ) \big)}= o_P(n^{-1/2}),
\end{align*}
which finishes the proof. \end{proof}

If we are able to construct estimators \(\hat{P}_n \) that meet the
conditions of Theorem \ref{thm:eff:estimator} and solve Equation
\eqref{eq:key:EIC:eq}, the so-called efficient influence curve
equation, then the resulting substitution estimator
\(\Psi^{G^*}(\hat{P }_n)\) is asymptotically linear and efficient. In
the following, we comment on conditions 1 and 2.

\begin{remark}[Second-order remainder]
  The second-order remainder expressed in \eqref{eq:second:order}
  displays a double robustness structure \citep{vanRobins2003unified}
  in the sense that,
\begin{align*}
  R_2(P,P_0)
  =0 \quad \text{if} \quad
  \bar{g}_{t}= \bar{g}_{0,t} \quad \text{or} \quad q_t = q_{0,t}
  .
\end{align*}
When \( {\bar{G}}_{t}\) is bounded away from zero, the product
structure of the remainder \( R_2(P,P_0) \) yields, by use of the
Cauchy-Schwartz inequality, an upper bound in terms of the
\(L_2(P_0)\)-norm of \(( \bar{g}_{0,t} - {\bar{g}}_{t} ) \) and the
\(L_2(P_0)\)-norm of \(( {q}_{0,t} - {{q}}_{t} ) \):
\begin{align*}
  \int_0^{\tau} \Vert  \bar{g}_{0,t} - \bar{g}_{t} \Vert_{P_0} \,
  \Vert  q_{0,t} - q_{t} \Vert_{P_0} 
  \,   dt
  .
\end{align*}
Here we have further used that \( Y \in [0,1]\).  The required
convergence rate \(R_2 (\hat{P }_n , P _0 ) = o_P(n^{-1/2})\) will for
example be achieved if we estimate both parts at a rate faster than
\(o_P(n^{-1/4})\). In Section \ref{sec:initial:estimation} we apply
recent results on highly adaptive lasso (HAL) estimation
\citep{van2017generally} to show that such estimators do exist.
\label{remark:R2}
\end{remark}

\begin{remark}[Donsker class conditions]
  In Section \ref{sec:initial:hal} we give conditions under which
  \(D^*(\hat{P}_n ) \in \mathscr{F}\) where
  \( \mathscr{F} = \lbrace D^*(P) \, : \, P \rbrace\) is a Donsker
  class.  The key is that there is a finite number of
  monitoring times per subject, so that any subject contributes with
  finitely many terms to the likelihood.  Then the canonical
  gradient can be written as a well-behaved mapping of the nuisance
  parameters such that Donsker properties of the nuisance parameters
  will translate into Donsker properties of the efficient influence
  curve.  An important Donsker class is the class of càdlàg functions
  with finite variation.
\end{remark}

Theorem \ref{thm:eff:estimator} tells us what we need for efficient
estimation of our target parameter.  The next sections deal with
construction of such an estimator.

\section{Targeted minimum loss-based estimation (TMLE)}
\label{sec:motivating:tmle}

We present a TMLE procedure for construction of an estimator that will
satisfy the conditions of Theorem \ref{thm:eff:estimator}.  In
summary, this consists of, first, constructing initial estimators for
the components of \(\bm{Z}\) and \(G\), and, second, setting up an
algorithm for performing an update of the collection of initial
estimators that guarantees that it solves the efficient influence
curve equation \eqref{eq:key:EIC:eq}.

We also refer to the second step as the targeting step or the
targeting algorithm. It involves for each component of \(\bm{Z}\) a
choice of a loss function and a corresponding path indexed by
\(\eps\in\R\) through the initial estimator of that component such
that the generated score at \(\eps=0\) gives a desired part of the
canonical gradient. The general targeting algorithm involves iterative
updating steps that are repeated until convergence, at which point the
efficient influence curve equation \eqref{eq:key:EIC:eq} is solved.

A certain amount of extra notation is needed.  For
\(t \in (0,\tau]\), we define the ``clever weights'' as the
Radon-Nikodym derivative
\begin{align}
  h^{G^*}_t & = \Prodi_{s <t } \frac{dG^*_s}{dG_s},
                 \label{eq:clever:weight}
\end{align}
that only depends on the \(G\)-part of the likelihood, and the
``clever covariates'',
\begin{align}
  h^{\ell}_t & =
               \EE_{P^{G^*}} [Y \, \vert\, \Delta N^{\ell}(t)=1, \F_{t-}]
               \label{eq:h:ell}- 
               \EE_{P^{G^*}} [Y \, \vert\, \Delta N^{\ell}(t)=0,  \F_{t-}],
  \\
  h^{a}_t & =  \label{eq:h:a}
            \EE_{P^{G^*}} [Y \, \vert\,  \Delta N^{a}(t)
            =1, \F_{t-}]-
            \EE_{P^{G^*}} [Y \, \vert\, \Delta N^{a}(t)=0,  \F_{t-}]
            ,
  \\
  h^{d}_t &= 
            1 -  \EE_{P^{G^*}} [Y \, \vert\,  N^{d}(t)=0,  \F_{t-}] ,
            \label{eq:h:d}
\end{align}
that only depend on \(\bm{Z}\). This now allows us to write the
canonical gradient as,
\begin{align}
  & D^*(P)
    = Z^{G^*}_{t=0}- \Psi^{G^*} (P)\notag\\
  & \qquad + \, \int_0^{\tmax}h^{G^*}_t \big( Z_{t}^{G^*} -
    Z_{t,L(t)}^{G^*}\big) dN^{\ell}(t) + \sum_{x\in\lbrace a, \ell,
    d\rbrace} \int_0^{\tmax}h^{G^*}_t\, h^{x}_t \, dM^{x}(t).
  \label{eq:integral:1}
\end{align}
In the following, we define loss functions and one-dimensional
parametric submodels for each component of \(\bm{Z}\) such that the
scores are equal to the respective terms of \eqref{eq:integral:1}.
These will be used to construct our targeting algorithm in Section
\ref{sec:TMLE}. We consider each term of \eqref{eq:integral:1}
separately and define loss functions and submodels for
\(Z^{G^*}_{t,L(t)}\) (Section \ref{ssec:tmle:cond:means}) and for each
of the conditional intensities \(\Lambda^x\), \(x=a, \ell, d\)
(Section \ref{ssec:tmle:intensities}).

Before we do so, we remark upon another particular approach taken in
the discrete-time setting \citep{van2010targeted,stitelman2012general}
that could also be applied to our problem.  Note that, in this paper,
we focus on the parametrization of the target parameter
\( \Psi^{G^*}\) in terms of \(\bm{Z}\) as presented in Section
\ref{sec:target:parameter:parametrization}, but we could instead
construct initial estimators for all likelihood components, i.e., all
conditional densities and intensities, and construct a targeted
updating step for each of these components as is done in the
discrete-time case. Here we focus on \(\bm{Z}\) rather than the full
likelihood so that we can avoid targeting the conditional density of
the potential high-dimensional \(L(t)\). This will make the targeting
step simpler and computationally more efficient.

\subsection{Loss function and submodel for \(Z^{G^*}_{t,L(t)}\)}
\label{ssec:tmle:cond:means}

We need a loss function and a submodel for \(Z^{G^*}_{t,L(t)}\) such
that the score of the submodel equals the first term of
\eqref{eq:integral:1}.  This term consists of an integral over a
difference between \( Z_{t}^{G^*}\) and \(Z_{t,L(t)}^{G^*}\); thus, we
will define our loss function and submodel for \(Z^{G^*}_{t,L(t)}\)
for a given \(Z_{t}^{G^*}\).  Particularly, we define the time-point
specific logarithmic loss for \({\QQ}_{t,L(t)}^{G^*}\), indexed by
\({\QQ}_{t}^{G^*}\):
\begin{equation*}
  \logloss{{\QQ}_{t}^{G^*}}\big({\QQ}_{t,L(t)}^{G^*}\big) =
  {\QQ}^{G^*}_{t} \log {\QQ}^{G^*}_{t, L(t)} +
  \big(1- {\QQ}^{G^*}_{t} \big) \log \big( 1-  {\QQ}^{G^*}_{t, L(t)}\big) .
\end{equation*} 
With this loss function, the parametric submodel
\begin{equation*}
  \logit \,  {\QQ}^{G^*}_{t,L(t)}(\eps) =\logit\,
  {\QQ}^{G^*}_{t,L(t)} + \eps h_t^{G^*}, 
\end{equation*} 
has the desired property that
\begin{equation*}
  \frac{d}{d\eps} \,\logloss{{\QQ}_{t}^{G^*}}\big({\QQ}_{t,L(t)}^{G^*}(\eps)\big)\,\Big\vert_{\eps=0} =
  {h}^{G^*}_t \, \big( {\QQ}^{G^*}_{t} - {\QQ}^{G^*}_{t,L(t)}  \big) . 
\end{equation*}
Correspondingly, we define the integrated loss for
\({\bar{\QQ}}^{G^*}_{\tau, L(\tau)} = ({\QQ}^{G^*}_{s, L(s)} \, : \, 0
\le s \le \tau)\), indexed by
\({\bar{\QQ}}^{G^*}_{\tau} = ({\QQ}^{G^*}_{s} \, : \, 0 \le s \le
\tau)\),
\begin{align*}
  \intlogloss{{\bar{\QQ}}_{\tmax}^{G^*}}\big({\bar{\QQ}}_{\tau, L(\tau)}^{G^*}(\eps)\big) =
  \int_0^{\tmax} \logloss{{\QQ}_{t}^{G^*}}\big({\QQ}_{t,L(t)}^{G^*}(\eps)\big) \, dN^{\ell}(t),
\end{align*}
for which we have that
\begin{align*}
  \frac{d}{d\eps}  \intlogloss{{\bar{\QQ}}_{\tmax}^{G^*}}\big({\bar{\QQ}}_{\tau, L(\tau)}^{G^*}(\eps)\big)
  \,\Big\vert_{\eps=0}
  & =
    \int_0^{\tmax} \frac{d}{d\eps} \,
    \logloss{{\QQ}_{t}^{G^*}}\big({\QQ}_{t,L(t)}^{G^*}(\eps)\big)\,\Big\vert_{\eps=0} \, dN^{\ell}(t)
  \\
  & = \int_0^{\tmax}{h}^{G^*}_t\, \big( {\QQ}^{G^*}_{t} -
    {\QQ}^{G^*}_{t,L(t)}\big) \, dN^{\ell}(t) ,
\end{align*}
which, as we wanted, equals the first term of the part of the
canonical gradient expressed in \eqref{eq:integral:1}.

\subsection{Loss function and submodel for the intensities
  \({\Lambda}^x(t)\)}
\label{ssec:tmle:intensities}

We consider the case that \(\Lambda^x\) is absolutely continuous and
denote by \(\lambda^x\) the intensity rate such that
\(\Lambda^x(t \, \vert \, \F_{t-}) \)
\( = \int_0^t \lambda^x (s \,\vert\, \F_{s-}) ds \).  For each
\(x=a,\ell,d\), we specify a proportional hazard type submodel for
\(\Lambda^x\) with time-dependent covariates as follows:
\begin{align}
  \begin{split}
      d\Lambda^x (t; \eps) & = d\Lambda^x(t) \exp ( \eps
    {h}^{G^*}_t {h}^{x}_{t}),
\end{split}
  \label{eq:lambda:fluc}
\end{align}
together with the partial log-likelihood loss function:
\begin{align*}
  \mathscr{L}_{x}({\Lambda}^x)
  &= \log \bigg( \Prodi_{t\in [0,\tmax]} \big(\lambda^x(t )\big)^{N^x(dt)}
    \big(1- d\Lambda^x(t  )\big)^{1-N^x(dt)} \bigg) \\
  &= \int_0^{\tmax}  \log  \lambda^x(t)\, dN^x(t) -
    \int_0^{\tmax} d\Lambda^x(t )   
    .
\end{align*}
For this pair of submodel and loss function we have the desired
property that
\begin{align*}
  \frac{d}{d\eps}\, \mathscr{L}_{x}({\Lambda}^x (\cdot\,; \eps))\,\Big\vert_{\eps=0}
  & = 
    \frac{d}{d\eps}\, \bigg(\int_0^{\tmax} \Big( \log {\lambda}^x(t)
    +  \eps {h}^{G^*}_t 
    {h}^{x}_{t}\Big)\, dN^x(t) \,  -\\
  &\qquad\qquad\qquad\qquad\quad
    \int_0^{\tmax}  \exp \big( \eps {h}^{G^*}_t 
    {h}^{x}_{t}\big) {d\Lambda}^x(t)   \bigg)\bigg\vert_{\eps=0} 
  \\
  & = \int_0^{\tmax} {h}^{G^*}_t 
    {h}^{x}_{t}\, dN^x(t) -
    \int_0^{\tmax} {h}^{G^*}_t 
    {h}^{x}_{t} \, {d\Lambda}^x(t) \\
  & = \int_0^{\tmax}  {h}^{G^*}_t 
    {h}^{x}_{t} \, \big(  dN^x(t) - d{\Lambda}^x(t) \big),  
\end{align*}
is equal to the corresponding terms of the canonical gradient as
expressed in \eqref{eq:integral:1}.

\section{Initial estimation of nuisance parameters}
\label{sec:initial:estimation}

To carry out our targeting algorithm, we need initial estimators for
the following time-sequences of conditional densities, conditional
expectations, and conditional intensities:
\begin{equation}
  \begin{split}
    G &= \big(
    \pi_{t}, {\Lambda}^{c}(t)  \big)_{t\in [0,\tau]},  \qquad \text{and,}\\
    {\bm{Z}} &= \big( {Z}^{G^*}_{t}, {Z}^{G^*}_{t,L(t)},
    {\Lambda}^{\ell}(t), {\Lambda}^{a}(t), {\Lambda}^{d}(t)
    \big)_{t\in [0,\tmax]} .
  \end{split}
  \label{eq:init:components}
\end{equation}
To establish conditions 1 and 2 of Theorem \ref{thm:eff:estimator}
under as weak as possible restrictions on our model \(\mathcal{M}\),
which in turn yields the asymptotic efficiency of our TMLE estimator,
we propose highly adaptive lasso (HAL) estimation
\citep{van2017generally} for each of the quantities displayed in
\eqref{eq:init:components}.  In the present work, the aim of
introducing HAL estimation is mainly theoretical; we can apply earlier
results established by \cite{van2017generally} on convergence
properties for the HAL estimator under one key assumption that the
nuisance parameters can be parametrized by functions that are \cadlag
with finite sectional variation norm.

We refer to Appendix \ref{app:smoothness:hal} for a more detailed
characterization of HAL estimation and specifically the conditions on
our model \(\mathcal{M}\) that are sufficient to establish asymptotic
efficiency of the TMLE estimator when initial estimators are
constructed using HAL.  The present section gives a short account of
the HAL results that can be summarized as follows:
\begin{enumerate}
\item The nuisance parameters that characterize the statistical model
  \(\mathcal{M}\) can be discontinuous or non-differentiable, we only
  need them to be parametrizable by functions that are \cadlag and
  have finite sectional variation norm. Such functions generate a
  signed measure \citep{gill1995inefficient,van2017generally} and can
  be represented in terms of its measures over sections
  \citep{gill1995inefficient} corresponding to an infinite linear
  combination of indicator basis functions.
\item The HAL estimator is defined as the minimizer of the
  loss-specific empirical risk over the space of \cadlag functions
  with sectional variation norm bounded by a constant.  In practice
  this can be solved by \(L_1\)-penalized (Lasso) regression
  \citep{tibshirani1996regression} where indicator basis functions are
  the covariates and the point-mass assigned to support points are the
  corresponding coefficients. Specifically, the sectional variation
  norm equals the sum of the absolute values of the coefficients,
  i.e., the \(L_1\)-norm of the vector of coefficients.
\item The highly adaptive lasso estimators attain the sufficient
  \(n^{-1/4}\) convergence rate \citep[a rate that has been further
  improved to \(n^{-1/3}\log(n)^{d/2}\) for dimension \(d\)
  by][]{2019arXiv190709244B}.  Thus, by Remark \ref{remark:R2} on the
  double robustness of the estimation problem, the use of HAL for
  estimation of the nuisance parameters in \eqref{eq:init:components}
  allows us to establish condition 1 of Theorem
  \ref{thm:eff:estimator}.
\item The class of functions that are \cadlag and have finite
  sectional variation norm is a Donsker class and the canonical
  gradient is a well-behaved mapping of the nuisance parameters. This
  yields that \(D^*(P ) \in \mathscr{F}\) where
  \( \mathscr{F} = \lbrace D^*(P) \, : \, P \rbrace\) is a Donsker
  class and thus provides the basis to establish condition 2 of
  Theorem \ref{thm:eff:estimator}.
\end{enumerate}

To further optimize the estimation of the quantities in
\eqref{eq:init:components}, we use a super learner that selects the
best performing estimator from a prespecified library of candidate
algorithms by minimizing the cross-validated empirical risk.  The
super learner performs asymptotically no worse than any algorithm
included in its library, a property known as the oracle inequality
\citep{van2003unified,van2006oracle}. Accordingly, a super learner
that includes the HAL estimator in its library will converge  at the
same rate as the HAL estimator. Appendix \ref{app:overview:diagrams}
provides a short description of super learning.

\section{Targeting algorithm}
\label{sec:TMLE}

Based on the loss functions and submodels defined in Sections
\ref{ssec:tmle:cond:means}--\ref{ssec:tmle:intensities}, we here
present a targeting algorithm for updating the collection of initial
estimators for \(\bm Z_0\) for a given estimator \(\hat{G}_n\) for
\(G_0\) on \( [0,\tmax]\). We index the initial estimator for
\(\bm{Z}\) by \(\mm=0\) as follows:
\begin{align*}
  \hat{\bm Z}_n^{\mm=0}=\Big(\hat{Z}^{G^*}_{t,\mm=0},
  \hat{Z}^{G^*}_{t,L(t),\mm=0},
  \hat{\Lambda}^{\ell}_{\mm=0}(t),
  \hat{\Lambda}^{a}_{\mm=0}(t), \hat{\Lambda}^d_{\mm=0}(t)
  \Big)_{t\in [0,\tmax]}.
\end{align*}
Our targeting algorithm involves separate updating steps for current
estimators \(\hat{\QQ}^{G^*}_{t,L(t),\mm}\) and
\(\hat{\Lambda}^x_{\mm}(t)\), \(x=a,\ell,d\), from \(\mm\) to
\(\mm+1\), \(\mm=0,1,\ldots\), carried out simultaneously for all
time-points.  Our algorithm is overall centered around the
representation of the target parameter by iterated expectations that
we introduced in Section \ref{sec:target:parameter:parametrization},
with the separate updating steps ensuring that we solve the individual
terms of the efficient influence curve equation. To describe the
algorithm, recall that, as introduced in Display
\eqref{eq:observed:event:times},
\begin{align*}
  \qquad 0={t}_0 <{t}_1 < \cdots <{t}_{{K}_n} ,
\end{align*}
denotes the ordered sequence of unique times of changes
\(\cup_{i=1}^n \cup_{k=1}^{K_{i}}\{T_{i,k}\}\) across all subjects of
the dataset. Evaluated at time \(t_r\), the conditional expectations
\(Z^{G^*}_t\) and \(Z^{G^*}_{t,L(t)}\) from Section
\ref{sec:target:parameter:parametrization} can be written:
\begin{align}
  & {Z}_{t_{r}}^{G^*}
    = \EE_{P^{G^*}} [ Y \, \vert \, L(t_{r}),
    N^{\ell}(t_{r}), N^{a}(t_{r}), N^{d}(t_{r}),  \F_{t_{r-1}}] , \label{eq:Z:tilde}  \\
  & {Z}_{t_{r},L(t_{r})}^{G^*}
    = \EE_{P^{G^*}} [ Y \, \vert \,
    N^{\ell}(t_{r}), N^{a}(t_{r}), N^{d}(t_{r}),  \F_{t_{r-1}}]. \label{eq:Z:tilde:L}
\end{align}
Further, at each \(t_r\), we need to estimate the sequence of
conditional expectations: 
\begin{align}
  &                  \EE_{P^{G^*}} [ Y \, \vert \,  
    N^{a}(t_{r}), N^{d}(t_{r}),  \F_{t_{r-1}}] , \label{eq:Z:tilde:ell}\\
  &  \EE_{P^{G^*}} [ Y \, \vert \,
    N^{d}(t_{r}),  \F_{t_{r-1}}], \label{eq:Z:tilde:a} \\
  &\EE_{P^{G^*}} [ Y \, \vert \,
    \F_{t_{r-1}}] .\label{eq:Z:tilde:d}
\end{align}
Based on current estimators \(\hat{\bm{Z}}^{{\mm}}_n\), we construct
estimators for \eqref{eq:Z:tilde}--\eqref{eq:Z:tilde:d} for all
\(r=1, \ldots,{K}_n\), which further yield estimators for
\(\hat{h}^{\ell}_{t_r,\mm},\hat{h}^{a}_{t_r,\mm},\hat{h}^{d}_{t_r,\mm}\)
for the clever covariates \(h^{\ell}_{t_r},h^{a}_{t_r}\),
\(h^{d}_{t_r}\). A more detailed description of this procedure is
given in Appendix \ref{app:overview:targeting}.  Estimators
\(\hat{h}_{t_r}^{G^*}\), \(r=1, \ldots,{K}_n\), for the clever weights
are obtained by substituting the estimator \(\hat{G}_n\) for \({G}\)
in \eqref{eq:clever:weight}.

\subsection{Updating the estimator for \({\QQ}^{G^*}_{t_r,L(t_r)}\)}
\label{sec:7:update:Z:L}

The updating step for the time-sequence
\(( \hat{\QQ}^{G^*}_{t_r, L(t_r), \mm} )_{1 \le r \le {K}_n}\) of
estimators for the conditional expectations \eqref{eq:Z:tilde:L} is
defined according to the loss function and submodel from Section
\ref{ssec:tmle:cond:means}, given the estimated clever weights
\((\hat{h}_{t_r}^{G^*})_{1 \le r \le {K}_n}\), as:
\begin{align*}
  \hat{\bar{\QQ}}^{G^*}_{\tau,L(\tau),\mm+1} := \hat{\bar{\QQ}}^{G^*}_{\tau,L(\tau),\mm} (\hat{\eps}_n). 
\end{align*}
Here, \(\hat{\eps}_n\) is estimated from the data by:
\begin{align*}
  \hat{\eps}_n := \underset{\eps}{\argmin} \, \mathbb{P}_n  \,
  \intlogloss{\hat{\bar{\QQ}}_{\tmax,\mm}^{G^*}}\big(\hat{\bar{\QQ}}_{t_r, L(t_r),\mm}^{G^*}(\eps)\big),
\end{align*}
and the now updated estimators
\(\hat{\bar{\QQ}}^{G^*}_{\tau,L(\tau),\mm+1}=\big(
\hat{\QQ}^{G^*}_{t_r, L(t_r), \mm+1} \big)_{1 \le r \le {K}_n}\)
solves the desired part of the efficient influence curve equation
\begin{align*}
  \mathbb{P}_n \int_0^{\tmax}\hat{h}^{G^*}_t
  \big( \hat{\QQ}^{G^*}_{t,\mm} -  \hat{\QQ}^{G^*}_{t, L(t), \mm+1}\big)\, d{N}^{\ell}(t)
  = 0. 
\end{align*}
Notably, this updating step is carried out only at subject-specific
time-points
\(t_r \in \lbrace T^{\ell}_1, \ldots,
T^{\ell}_{N^{\ell}(\tau)}\rbrace\), where changes in \(L(t)\) are
observed.

\subsection{Updating the estimators for the intensities}
\label{sec:7:update:intensity}

For each of \(x=\ell,a,d\), the updating step for the current
estimator \(\hat{\Lambda}^{x}_{\mm}\) for the intensity \(\Lambda^x\)
uses the estimated time-sequence of clever weights
\((\hat{h}_{t_r}^{G^*})_{1 \le r \le {K}_n}\) and the time-sequence
\((\hat{h}^{x}_{t_r,\mm})_{1 \le r \le {K}_n} \) of current estimators
for the relevant clever covariate.  Based on the loss function
and submodel as defined in Section \ref{ssec:tmle:intensities},
\(\hat{\eps}_x\) is now estimated from the observed data by:
\begin{align*}
  \hat{\eps}_{x} := \underset{\eps}{\argmax} \,\, \mathbb{P}_n  \,
  \mathscr{L}_{x} \big( \hat{\Lambda}^{x}_\mm ( \cdot\, ; \eps)\big) .
\end{align*}
We denote the corresponding updated intensity by
\( \hat{\Lambda}^{x}_{\mm+1} = \hat{\Lambda}^{x}_{\mm} (\cdot\,;
\hat{\eps}_{x} )\), which now solves
\begin{align*}
  \mathbb{P}_n  \int_0^{\tmax}  \hat{h}^{G^*}_t 
  \hat{h}^{x}_{t,\mm} \, \big(  dN^x(t) - d\hat{\Lambda}_{\mm+1}^x(t) \big) =0,
\end{align*}
the equation of interest.

\subsection{Iterating the targeting steps}

 The updated estimators for \({\QQ}^{G^*}_{t,L(t)}\) and the
intensities across time yield updated estimators for the sequence of
conditional expectations \eqref{eq:Z:tilde}--\eqref{eq:Z:tilde:d} and
thus for the clever covariates.  This process constitutes the
targeting iteration from \(\mm\) to \( \mm+1\), corresponding to
updating \(\hat{\bm Z}_n^{\mm}\) into \(\hat{\bm Z}_n^{\mm+1}\). The
process is now repeated iteratively for \(\mm= 0,1,2,\ldots\), moving
from a current collection of estimators \(\hat{\bm Z}_n^{\mm}\) to an
updated collection of estimators \(\hat{\bm Z}_n^{\mm+1}\).  At each
step \(\mm\), the efficient score equation is evaluated
\begin{align*}
  \mathbb{P}_n D^*(\hat{\bm{Z}}_n^{\mm}, \hat{G}_n)
  &= \mathbb{P}_n \, \bigg(  \hat{Z}_{0,\mm}^{G^*}  + \,                  
    \int_0^{\tmax}\hat{h}_t\, \big( \hat{Z}_{t,\mm}^{G^*} -
    \hat{Z}_{t,L(t),\mm}^{G^*}\big) \,d N^{\ell}(t) \\
  &\quad\qquad + \sum_{x\in\lbrace a, \ell, d\rbrace}
    \int_0^{\tmax}\hat{h}_t\, \hat{h}^{x}_{t,\mm} \,
    \big(dN^{x}(t) - d\hat{\Lambda}^x_{\mm}(t)\big)\bigg) - \Psi^{G^*} (\hat{\bm{Z}}_n^{\mm}) , 
\end{align*}
and the iterations from \(\mm\) to \(\mm+1\) are continued until
\begin{align*}
  \vert \, \mathbb{P}_n D^*(\hat{\bm{Z}}_n^{\mm}, \hat{G}_n) \, \vert < s_n,
\end{align*}
for some choice of stopping criterion \(s_n = o_P(n^{-1/2})\). We can
for example use \(s_n = \sigma / (n^{-1/2} \log n) \), where
\(\sigma^2\) is the variance of the efficient influence function which
we can estimate by substituting the current estimators
\(\hat{\bm{Z}}_n^{\mm}\). Letting
\(\mm^* := \min (\mm \, : \, \vert \, \mathbb{P}_n
D^*(\hat{\bm{Z}}_n^{\mm}, \hat{G}_n) \, \vert < s_n)\), we denote the
final estimator by \(\hat{\bm{Z}}^*_{n}=\hat{\bm{Z}}^{\mm^*}_{n}\),
where
\begin{align*}
  \hat{\bm{Z}}^{\mm^*}_n =  \Big(  \hat{Z}^{G^*}_{t, \mm^*}, \hat{Z}^{G^*}_{t,L(t),\mm^*},
  \hat{\Lambda}_{\mm^*}^{\ell}(t),
  \hat{\Lambda}_{\mm^*}^{a}(t),  \hat{\Lambda}_{\mm^*}^{d}(t)
  \Big)_{t\in[0,\tmax]} .
\end{align*}

\section{Inference for the targeted substitution estimator}
\label{sec:TMLE:inference}

In Section \ref{sec:initial:estimation} we have summarized results
that yield the existence of initial estimators that fulfill the
regularity conditions of Theorem \ref{thm:eff:estimator}. In Section
\ref{sec:TMLE} we have presented a targeting algorithm that maps the
initial estimator \(\hat{\bm{Z}}_n^{\mm=0}\) into a targeted estimator
\(\hat{\bm{Z}}_n^{*}\) that solves the efficient influence curve
equation \eqref{eq:key:EIC:eq}. In the end we estimate the target
parameter by:
\begin{align*}
  \hat{\psi}^{G^*}_n  =
  \mathbb{P}_n \, \hat{\QQ}^{G^*}_{t=0,\mm^*},
\end{align*}
but we note that \(\hat{\psi}^{G^*}_n\) is equal to the substitution
estimator \(\Psi^{G^*} (\hat{P}^{*}_{n})\), where \(\hat{P}^*_n \) is
characterized by \((\bm{Z}^*_n, \hat{G}_n)\).  Specifically, the
components of \(\hat{\bm{Z}}_n^{*}\) are compatible with a probability
distribution \(\hat{P}_n^{*}\) whose conditional expectations of \(Y\)
given the relevant histories coincide with those of
\(\hat{\bm{Z}}_n^{*}\).

Theorem \ref{thm:eff:estimator} implies asymptotic efficiency of
\( \hat{\psi}^{G^*}_n\), and we can use the asymptotic normal
distribution
\begin{align*}
  \sqrt{n}\, \big(  \Psi^{G^*} (\hat{P}^{*}_n) 
  - \Psi^{G^*} (P_0) \big)
  \overset{\mathcal{D}}{\rightarrow} \mathcal{N} (0, {P}_0  D^* ( {P}_0
  )^2), 
\end{align*}
to provide an approximate two-sided confidence interval. Here
\begin{align}
  \hat{\sigma}_n^2 := \mathbb{P}_n \lbrace D^* (
  \hat{P}^{*}_n
  )\rbrace ^2,
  \label{eq:hat:sigma:n}
\end{align}
can be used to estimate the variance of the TMLE estimator.

\section{Empirical study}
\label{sec:simulation:study}

We demonstrate our methods by describing the application to simulated
data, using the proposed algorithm to estimate the contrast between a
treatment rule that sets \(A(t)=1\) and a treatment rule that sets
\(A(t)=0\). One could think of a randomized trial where subjects in
the study population are initially randomized to receive treatment or
no treatment, but may switch treatment over time depending on the
value of time-varying covariates.  This is a setting where a standard
Cox regression, as mentioned in Section \ref{sec:motivating:example},
does not apply. Our focus is here to confirm our theory by evaluating
the targeting algorithm, and, further, to compare our new algorithm to
the existing longitudinal TMLE (LTMLE).

We consider a setting where subjects of a population are followed for
\(\tau\) days of follow-up time. On any given day, subjects may change
treatment, covariates, may be lost to follow-up (right-censored), or
may experience the outcome of interest. Both the treatment and the
censoring mechanisms are subject to time-dependent confounding.  The
data are simulated such that the number of monitoring times per
subject are approximately the same across different \(\tau\). Thus,
the larger \(\tau \) is, the less events are observed at single
monitoring times.  Throughout we let \(n=1,000\).

\subsection{Setup} In all simulations, we generate data from a
sequence of logistic regressions allowing time-varying treatment and
covariates to affect one another, keeping effects
time-constant. Throughout, we let
\(L_0\in \lbrace 1, \ldots, 6\rbrace\),
\(A_0\in \lbrace 0, 1\rbrace\), \(A(t)\in \lbrace 0, 1\rbrace \) and
\(L(t)\in \lbrace 0, 1\rbrace \). We draw observations \(A_0\) such
that large values of \(L_0\) increase the probability of
\(A_0=1\). Subjects for which \(A_0=0\) and with current covariate
value equal to 1 are more likely to begin treatment by time
\(t\). Further, current treatment increases the probability of
\(L(t)=1\).  Both the baseline treatment \(A_0\) and the time-varying
treatment \(A(t)\) have a negative main effect on the outcome process
\(N^d(t)\). Moreover, \(A_0\) has a different effect within levels of
\(L(t)\), and \(L(t)\) has a positive main effect. At last, the
censoring process \(N^c(t)\) depends on both \(A_0\) and current
covariates.

\subsection{Parameter of interest}
Our parameter of interest is the contrast between the
intervention-specific mean outcomes under the regime that imposes
\(A(t)=1\) (subjects adhering to treatment) to the regime that imposes
\(A(t)=0\) (subjects adhering to no treatment). Both regimes also
impose no censoring throughout follow-up. Thus, our interventions are
specified as follows, following Definition 1:
\begin{align*}
 & \text{Intervention 0: } \qquad
  dG^{0}_{t}(O) = \big(1- A(t) \big)^{N^{a}(dt)} \big( 1 -
  N^c(t)
  \big),  \\
 & \text{Intervention 1: } \qquad
  dG^{1}_{t}(O) = \big(A(t) \big)^{N^{a}(dt)} \big( 1 -
  N^c(t)
  \big).  \qquad\qquad\qquad\quad
\end{align*}
We let \(\psi^1\) denote the intervention-specific mean outcome under
Intervention 1 and \(\psi^0\) the intervention-specific mean outcome
under Intervention 0. Our target parameter is the corresponding
difference:
\begin{align*}
  \Psi (P)= \int  Y  dP_{Q,G^{1}} -\int  Y  dP_{Q,G^{0}}
  = \psi^1 - \psi^0.
\end{align*}
As usual, the true value is denoted by \(\psi_0=\Psi (P_0)\).  There
is no unmeasured confounding, so that our parameter can be interpreted
causally as the treatment effect we would see in the real world if
subjects had been treated (\(A(t)=1\)) compared to not
(\(A(t)=0\)). Note that \(\psi_0<0\) reflects a protecting effect of
the treatment. Throughout, we estimate \(\psi_0^1\) and \(\psi_0^0\)
separately and report results for the estimated difference
\(\psi_0 = \psi_0^1 - \psi_0^0\).

\subsection{Simulations} Overall, we seek to investigate the
following:

\begin{enumerate}
\item The distribution of \(\sqrt{n}\, ( \hat{\psi}^*_n - \psi_0) \)
  under correctly specified parametric models for initial estimation
  of the nuisance parameters.
\item Comparison to LTMLE estimation when data are given on a discrete
  grid with small \(\tau\).
\item Robustness properties of our TMLE estimator: What happens when
  we misspecify the distribution of, for example, the outcome process.
\end{enumerate}

As in Section \ref{sec:TMLE:inference}, we use \({\sigma}^2\) to
denote the variance of the canonical gradient and \(\hat{\sigma}^2_n\)
its estimator. We are generally interested in the coverage of
confidence intervals based on
\(\hat{\sigma}^2_n\).

\subsubsection*{Comparison to LTMLE}

For small \(\tau\), we can compare our estimation procedure to the
results from the existing LTMLE implementation \citep{ltmleRpackage}.
Table \ref{table:ctmle:ltmle:results:2} shows the results of
estimating the contrast \(\psi_0=\psi_0^1-\psi_0^0\) when \(\tau = 5\)
and when \(\tau = 50\), using LTMLE and using our algorithm. The table
illustrates the deterioration of LTMLE when \(\tau\) increases; for
the simulations with \(\tau =50\), the estimated standard error from
LTMLE is completely off.

% Tue May 12 14:32:17 2020
\begin{table}[!t]
%  \centering
  \begin{tabular}{cc}
  \begin{minipage}{.5\linewidth}
\begin{tabular}{rrrr}
  \hline\hline\\[-0.9em]
  \multicolumn{1}{l}{LTMLE}  & \(\tau=5\) &\(\tau=30\) & \(\tau=50\)  \\
  \hline\\[-0.9em]
  $\psi_{0}$ & -0.135 & -0.138 & -0.148 \\ 
  $\hat{\psi}^{\mathrm{ltmle}}_{n,M}$ & -0.135 & -0.136 & -0.149 \\ 
  Bias & -0.001 & 0.002 & -0.001 \\ 
  Cov (95\%) & 0.941 & 0.972 & 0.986 \\ 
  $\sqrt{\mathrm{MSE}}$ & 0.034 & 0.048 & 0.039 \\ 
  $\hat{\sigma}_{n,M}$ & 0.034 & 0.051 & 0.052 \\  
  \hline\\
\end{tabular}
\end{minipage} &
  \begin{minipage}{.5\linewidth}
\begin{tabular}{rrrr}
  \hline\hline\\[-0.9em]
  \multicolumn{1}{l}{contTMLE}  & \(\tau=5\) &\(\tau=30\) & \(\tau=50\)  \\
  \hline\\[-0.9em]
  $\psi_{0}$ & -0.135 & -0.138 & -0.148 \\ 
  $\hat{\psi}^{*}_{n,M}$ & -0.137 & -0.137 & -0.147 \\ 
  Bias & -0.001 & 0.001 & 0.001 \\ 
  Cov (95\%) & 0.947 & 0.956 & 0.953 \\ 
  $\sqrt{\mathrm{MSE}}$ & 0.033 & 0.034 & 0.035 \\ 
  $\hat{\sigma}_{n,M}$ & 0.033 & 0.035 & 0.034 \\ 
  \hline\\
\end{tabular}
\end{minipage}
  \end{tabular}
  \caption{ %\raggedright
    % \begin{flushleft}
    Results from a simulation study with \(\tau=5,30,50\) days of follow-up. Left: Results from applying
    existing LTMLE software. Right: Results from applying
    our algorithm (contTMLE). The row \(\psi_0\) shows the
    true value of the target parameter, and the rows
    \(\hat{\psi}_{n,M}\) and \(\hat{\sigma}_{n,M}\) show averages of
    estimates across the \(M=1000\) simulation repetitions. 
      %       \end{flushleft}
  }
\label{table:ctmle:ltmle:results:2}
\end{table}

\subsubsection*{Performance of estimation and double robustness}

We investigate the distribution of
\(\sqrt{n}\, ( \hat{\psi}^{*}_n - \psi_0) \) under correctly specified
and misspecified parametric models used for the initial
estimation. For the latter, we leave out all time-varying variables
when estimating the outcome distribution.  Table
\ref{table:ctmle:misspecify:results:1} shows the results of estimating
the contrast \(\psi_0=\psi_0^1-\psi_0^0\) with \(\tau=30, 50, 100\)
days of follow-up.  In the last setting, LTMLE cannot be
applied due to too few events at single monitoring times.  The results
in the table illustrate that our algorithm achieve appropriate
coverage across \(\tau\), and that the targeting step of our algorithm
removes bias from the initial estimation. In these simulations we
further see that our TMLE estimation based on misspecified initial
estimators is able to achieve very similar results to our TMLE
estimators based on correctly specified initial estimators.

\begin{table}[ht]
  \centering
  \begin{tabular}{cc}
  \begin{minipage}{.5\linewidth}
    \begin{tabular}{rrrr}
      \multicolumn{4}{l}{Correctly specified initial estimation}\\
      \hline\hline\\[-0.9em]
      \multicolumn{1}{l}{}  & \(\tau=30\) &\(\tau=50\) & \(\tau=100\)  \\
      \hline\\[-0.9em]
      $\psi_{0}$ & -0.138 & -0.148 & -0.130 \\ 
      $\hat{\psi}^{\mathrm{init}}_{n,M}$ & -0.138 & -0.147 & -0.129 \\ 
      $\hat{\psi}^{*}_{n,M}$ & -0.137 & -0.147 & -0.129 \\ 
      Bias (init) & 0.000 & 0.001 & 0.001 \\ 
      Bias (tmle) & 0.001 & 0.001 & 0.001 \\ 
      Cov (95\%) & 0.956 & 0.953 & 0.945 \\ 
      $\sqrt{\mathrm{MSE}}$ & 0.034 & 0.035 & 0.035 \\ 
      $\hat{\sigma}_{n,M}$ & 0.035 & 0.034 & 0.034 \\ 
      \hline\\
    \end{tabular}
\end{minipage} &
  \begin{minipage}{.5\linewidth}
\begin{tabular}{rrrr}
  \multicolumn{4}{l}{Misspecified initial estimation}\\
  \hline\hline\\[-0.9em]
  \multicolumn{1}{l}{}  & \(\tau=30\) &\(\tau=50\) & \(\tau=100\)  \\
  \hline\\[-0.9em]
  $\psi_{0}$ & -0.138 & -0.148 & -0.130 \\ 
  $\hat{\psi}^{\mathrm{init}}_{n,M}$ & -0.119 & -0.135 & -0.128 \\ 
  $\hat{\psi}^{*}_{n,M}$ & -0.138 & -0.147 & -0.129 \\ 
  Bias (init) & 0.019 & 0.013 & 0.001 \\ 
  Bias (tmle) & 0.000 & 0.001 & 0.000 \\ 
  Cov (95\%) & 0.936 & 0.955 & 0.944 \\ 
  $\sqrt{\mathrm{MSE}}$ & 0.035 & 0.033 & 0.034 \\ 
  $\hat{\sigma}_{n,M}$ & 0.035 & 0.034 & 0.034 \\  
  \hline\\
\end{tabular}
\end{minipage}  \end{tabular}
\caption{%\raggedright
  Results from a simulation study with \(\tau=30,50,100\) days of
  follow-up. The initial estimator is denoted
  $\hat{\psi}^{\mathrm{init}}_{n}$ and the targeted estimator is
  denoted $\hat{\psi}^{*}_{n}$.  Left: Results from applying our
  algorithm with correctly specified initial estimation. Right:
  Results from applying our algorithm with misspecified initial
  estimation.  The row \(\psi_0\) shows the true value of the target
  parameter, and the rows \(\hat{\psi}^{*}_{n,M}\),
  \(\hat{\psi}^{\mathrm{init}}_{n,M}\) and \( \hat{\sigma}_{n,M}\)
  show averages of estimates across the \(M=1000\) simulation
  repetitions. As desired, the targeting algorithm corrects the bias
  of the initial estimator. It is further noted that the results for
  the TMLE estimator with misspecified and correctly specified initial
  estimation are very similar.  }
\label{table:ctmle:misspecify:results:1}
\end{table}

\subsection{Conclusions on the empirical findings}

Based on our simulations we draw the following conclusions:
\begin{enumerate}
\item The confidence intervals based on the estimate of the efficient
  influence curve have valid coverage in the simulations when the
  initial estimation is correctly specified.
\item Our new estimation algorithm produces similar results to the
  existing LTMLE for the settings where LTMLE applies.  When \(\tau\)
  gets larger, and thus observation sparsity increases, LTMLE breaks
  down.
\item Misspecification of the \(Q\)-part of the likelihood, such as
  the distribution of the outcome process, leads to biased initial
  estimation. This bias is corrected for by our targeting procedure.
\end{enumerate}

\section{Discussion}
\label{sec:discussion:future}

The current paper lays the groundwork for targeted minimum loss based
estimation of intervention-specific mean outcomes based on a
continuous-time counting process model.  Our generalization of the
TMLE methodology is motivated by the problems arising when estimating
interventional effects from observational data, where observations are
made without a schedule and potentially over very long time
horizons. We have developed our TMLE based on a continuous-time model
to cover any time-scale.  The advantage is that we can track the
information of changes in continuous time, hereby preserving the
original time-order of treatment and covariates.

We have derived the efficient influence function for the estimation
problem, and we have proposed a particular targeting algorithm to
construct an estimator that solves the efficient influence curve
equation.  Contrary to the existing discrete-time longitudinal TMLE
method (LTMLE) that involves a regression step separately for every
time-point, even if nothing is observed at that time-point, our
proposed TMLE algorithm allows us to smooth information across time.
In a sense, the estimation procedure we have presented amounts to
doing a sequential regression for each infinitesimal time interval,
but simultaneously by pooling over all \(t\).

A substantial advantage of the proposed setting for estimation of
interventional effects is that it provides a framework for studying
various types of interventions.  In this paper, we have focused on
interventions on censoring (\(\Lambda^c\)) and treatment decision
(\(\pi_t\)), and in our simulations we considered only an intervention
that imposed `treatment' compared to `no treatment'.  As discussed in
Section 1.1 and Section 2.2, other interventions on treatment decision
and interventions on the treatment monitoring intensity \(\Lambda^a\)
are possible within the same framework: The targeting procedure
remains the same, although, for the latter, \(\Lambda^a\) is taken out
of the non-interventional part, and thus out of the targeting
procedure. In future work we will consider stochastic and optimal
interventions
\citep{murphy2003optimal,murphy2005experimental,hernan2006comparison,zhang2013robust,zhang2012robust}
both on the treatment decision and on the treatment monitoring. One
intervention of interest could be to replace \(\Lambda^a\) by
\(\Lambda^{a,*}\) that only depends, for example, on the time since
last treatment monitoring to ensure a regime where subjects are
monitored regularly.

It is evident from the general analysis of TMLE that the performance
of the final estimator depends on how well the nuisance parameters are
estimated. In the continuous-time setting, the nuisance parameters
comprise regressions of the intensity of changes and regressions of
the time-dependent means.  We have shown that we can construct HAL
estimators for all required components such as to fulfill the
conditions needed for asymptotic efficiency. To improve the estimation
of nuisance parameters, it is desirable to set up large libraries
consisting of flexible models and estimators for the conditional means
and intensities that can then be fed into the super learner. In
general, one should adapt the choices to the application at hand,
considering the length of the time interval, the marginal intensity
rates of the changes of action and covariates as well as the sample
size and outcome prevalence.

We further point out that our targeting algorithm does not depend on
the conditional density \(\mu_t\) of the covariates \(L(t)\); although
we may use \(\mu_t\) to provide an initial estimator for the
conditional expectations, we are not committed to doing so. The fact
is that it can be a big advantage to avoid estimation of \(\mu_t\)
altogether. In future work we investigate in more detail how to
construct an estimator, for instance based on inverse probability
weighted regression, that we map into an estimator that fulfills the
representation in terms of iterated expectations of Lemma
\ref{lemma:representation:Psi:Z}.

Future work will further deal with a general implementation of our
TMLE procedure, beyond the simple settings of our simulation study. In
addition to the applications outlined in Section
\ref{sec:motivating:example}, another important area of application is
that of randomized trials where subjects crossover, start additional
treatment and drop out. Here our methods can be applied to supplement
the intention-to-treat analysis.

\section*{Acknowledgements}

The authors would like to thank the anonymous reviewers for helpful
comments and suggestions that led to substantial improvement in the
presentation.

\newpage

\renewcommand{\thesection}{Appendix \Alph{section}}
\renewcommand{\thesubsection}{\Alph{section}.\arabic{subsection}}

% \newpage

% \bibliographystyle{chicago} 
% \bibliography{draft-continuous-tmle-revision2-fall-2020}
% \bibliography{refs}   

\appendix 

\section*{Supplementary material}

\setcounter{section}{0}

\section{Analysis of the estimation problem}
\label{app:canonical:gradient}

The first part of this appendix (Sections
\ref{sec:intro:canonical:gradient}--\ref{sec:deriving:eff:ic}) is
concerned with the derivation of the canonical gradient as presented
in Theorem \ref{thm:eff:ic}. In the second part, we compute the
second-order remainder (Sections
\ref{sec:cor:eff}--\ref{sec:remainder}) and prove Lemma
\ref{lemma:representation:Psi:Z} from Section
\ref{sec:target:parameter:parametrization} (Section
\ref{app:proof:of:lemma:Psi:Z}). We refer to the main text with
respect to notation; additional notation is introduced in the sections
of the appendix where necessary.

\subsection{Canonical gradient}
\label{sec:intro:canonical:gradient}

An estimator is asymptotically efficient if and only if it is regular
and asymptotically linear with influence curve equal to the canonical
gradient. This makes the canonical gradient a central component in the
construction of an asymptotically efficient estimator
\citep{bickel1993efficient,van2000asymptotic,tsiatis2007semiparametric}. We
here derive the canonical gradient for our statistical model
\(\mathcal{M}\) and target parameter
\(\Psi^{G^*} \, : \, \mathcal{M} \rightarrow \R\).  In the following,
consider the submodel \(\mathcal{M}_G \subset \mathcal{M}\) that
assumes \(G\) to be known and let \(\mathscr{T}_G(P)\) denote the
tangent space at \(P\) in the submodel \(\mathcal{M}_G\). To derive
the canonical gradient, we characterize the tangent space
\(\mathscr{T}_G(P)\) and construct an initial gradient of the pathwise
derivative of \(\Psi^{G^*} \, : \, \mathcal{M}_G\rightarrow \R\) at
\(P\). Relying on the factorization of any \(P\in\mathcal{M}\) into a
\(G\)-part and a \(Q\)-part, the canonical gradient can then be found
as the projection of the initial gradient onto the tangent space
\(\mathscr{T}_G(P)\) \citep{vanRobins2003unified}. Note that the
tangent space \(\mathscr{T}_G(P)\) is generated by fluctuations of the
conditional distributions of \(Q\) only since \(G\) is known. We deal
with this in Section \ref{sec:tangent:space}. Furthermore, to
construct an initial gradient, we can use the influence curve of any
regular and asymptotically linear (RAL) estimator of
\(\Psi^{G^*} \, : \, \mathcal{M}_G \rightarrow \R\). In Section
\ref{sec:initial:gradient}, we construct an inverse probability
weighted estimator
\citep{hernan2000marginal,hernan2002estimating,hernan2001marginal,vanRobins2003unified},
the influence curve of which we will use as an initial gradient. In
Section \ref{sec:deriving:eff:ic} we provide projection formulas and
in Section \ref{sec:proof:thm:eff} we collect the results of the
lemmas of Section \ref{sec:deriving:eff:ic} to provide the final proof
of Theorem \ref{thm:eff:ic}.

\subsection{Initial gradient}
\label{sec:initial:gradient} 

To obtain an initial gradient, we construct an asymptotically linear
estimator given the submodel \(\mathcal{M}_G \subset \mathcal{M}\)
that assumes that \(G\) is known and derive its influence function. To
this end, we note the following representation of the target
parameter, for any \(P\in \mathcal{M}\),
\begin{align*}
  \Psi^{G^*}( P)
  &= \EE_{ P^{G^*} }\big[ Y\big]
    =
    \int y\,  \Prodi_{t \in [0, \tmax]} dQ_{ t}(o) \, dG_t^*(o) \\
  &=
    \int y\, \Prodi_{t \le \tmax}
    d{G}_{t}^*(o)   \frac{d{G}_{t}(o) }{d{G}_{t}(o)} \,
    d{Q}_{t}(o)  \\
  &=
    \int \bigg( y \, \Prodi_{t \le \tmax} \frac{ d{G}_{t}^*(o) }{d{G}_{t}(o)} \bigg)
    \Prodi_{t \le \tmax} d{G}_{t}(o)\, d{Q}_{t}(o)   \\
   & = \EE \bigg[ Y \, \Prodi_{t \le \tmax} \frac{d{G}_{t}^*(O)}{d{G}_{t}(O)} \bigg] . 
\end{align*}
This suggests that we can estimate
\( \Psi^{G^*} \, :\, \mathcal{M}_G \rightarrow \R\) with an inverse
probability weighted estimator as follows:
\begin{align}
  \hat{\psi}^{G^*}_{n,\mathrm{IPW}}  := \mathbb{P}_n \bigg(
  \Prodi_{t \le \tmax} \frac{d{G}_{t}^*}{d{G}_{t}} \, Y \bigg)
  =
  \frac{1}{n} \sum_{i=1}^{n}
  \Prodi_{t \le \tmax} \frac{d{G}_{t}^* (O_i)}{d{G}_{t} (O_i)}
  \, Y_i 
  \label{eq:IPTW:est}
\end{align}
The IPW estimator as written in \eqref{eq:IPTW:est} is a sample
mean. It follows then directly that it is a linear (and thereby also
asymptotically linear) estimator at any \(P \in \mathcal{M}_G\) with
influence curve:
\begin{align}
  D^{G^*}_{\mathrm{IPW}}(P) (O) = \Prodi_{t\le \tau}
  \frac{d{G}_{t}^* (O)}{d{G}_{t} (O)}  \, Y - \Psi^{G^*}(P).
  \label{eq:ipaw}
\end{align}
Since \(D^{G^*}_{\mathrm{IPW}}(P)\) is an influence curve of a RAL
estimator of \(\Psi^{G^*}\) in \(\mathcal{M}_G\), we can derive the
canonical gradient by projecting \(D^{G^*}_{\mathrm{IPW}}(P)\) onto
\(\mathscr{T}_G(P)\).

\subsection{Tangent space}
\label{sec:tangent:space}

In the following we characterize the tangent space
\(\mathscr{T}_G(P)\). Based on the factorization of the probability
distribution of the observed data, such that the likelihood from
Display \eqref{eq:like} of Section \ref{sec:obs} consists of separate
terms for each of \(\mu_{L_0}\), \((\mu_t)_{t\ge 0}\) and
\(\Lambda^x\), \(x=a, \ell, d\), the tangent space \(\mathscr{T}_G\)
is given as the orthogonal sum of the individual tangent spaces:
\begin{align}
  \mathscr{T}_G = \mathscr{T}_{\mu_0} \oplus \mathscr{T}_{\mu} \oplus \mathscr{T}_{\Lambda^a} \oplus \mathscr{T}_{\Lambda^{\ell}} \oplus \mathscr{T}_{\Lambda^d}.
  \label{eq:orth:sum:0}
\end{align}
In the following, \(L_2(P)\) is used to denote the Hilbert space of
measurable functions \(h \, : \, \mathcal{O} \rightarrow \R\) with
\(P h^2 < \infty\).  The following lemmas characterize the tangent
spaces in \eqref{eq:orth:sum:0}. The notation \(h\in \sigma(X)\), for
some random variable \(X\), means that \(h\) is measurable with
respect to the \(\sigma\)-algebra generated by \(X\), and can thus be
written as a function of \(X\).

\begin{lemma}
  The tangent space \(\mathscr{T}_{\mu_{L_0}}\) associated with the density \(\mu_{L_0}\) of
  \(L_0\) is given by: 
  \begin{align*}
    \mathscr{T}_{\mu_{L_0}} (P)
    &= \big\lbrace
      h \in \F_0   \, : \,  \EE_P[ h ]
      = 0\big\rbrace \cap L_2(P).
  \end{align*}
  \label{lemma:tangent:q0}
\end{lemma}

\begin{proof} We put no restrictions on the density \(\mu_{L_0}\), so
  the corresponding tangent space \(\mathscr{T}_{\mu_{L_0}}\) is the
  entire Hilbert space of measurable square-integrable functions of
  \(L_0\) with mean zero (\cite{van2000asymptotic}, Example
  25.16).\end{proof}

\begin{lemma}
  The tangent space \(\mathscr{T}_{\mu}\) associated with the
  conditional density \(\mu_t\) of \(L(t)\) is given by:
  \begin{align*}
    \mathscr{T}_{\mu} (P)
    = \bigg\lbrace \int_0^{\tmax} h_t  \,  N^{\ell}(dt) \, :\,
    h_t \in   \sigma ( L(t),N^{\ell}(t), \F_{t-}) \, \wedge \, \\ 
    \EE_{P} \big[ h_t  \, \big\vert \, N^{\ell}(t), \F_{t-}\big] = 0 \bigg\rbrace  \cap L_2(P). 
  \end{align*}
    \label{lemma:tangent:qt}
  \end{lemma}

  \begin{proof} Consider the parametric submodel:
\begin{align*}
  \mu_{t, \eps, h_t} (o) = \big( 1 + \eps h_t (o) \big) \, \mu_{t} (o),
\end{align*}
for \(\eps\ge 0\), with \(h_t \, :\, \mathcal{O} \rightarrow \R\) such
that \(h_t \in \sigma ( L(t), \F_{t-}) \).  To ensure that
\( \mu_{t, \eps, h}\) for all \( \eps>0\) actually gives rise to a
well-defined probability measure, note that:
\begin{align*}
  \int_{\mathcal{O}} \,\mu_{t, \eps,h_t}\,d\nu^{\ell} =
  \underbrace{\int_{\mathcal{O}} \mu_{t}\,d\nu^{\ell}}_{=1} + \,
  \eps  \int_{\mathcal{O}} h_t\,  \mu_{t}\,d\nu^{\ell} =1 \quad \text{ if and only if } \\
  \quad
  \int_{\mathcal{O}} h_t\,\mu_{t}\,d\nu^{\ell}
  = \EE [ h_t \, \big\vert  \F_{t-}] = 0 . 
\end{align*}
The contribution to the log-likelihood is
\begin{align*}
 \log {dP_{ \eps, h_t}}(o) = \int_0^{\tmax}\log \Big( \big( 1 + \eps h_t (o) \big)
  \, \mu_{t} (o)\Big) N^{\ell}(dt), 
\end{align*}
and we derive the score accordingly
\begin{align*}
  \frac{d}{d\eps} \log {dP_{ \eps, h_t}}(o) \bigg\vert_{\eps=0}
&  = \frac{d}{d\eps} \bigg( \int_0^{\tmax}\log \Big( \big( 1 + \eps h_t (o) \big)
  \, \mu_{t} (o)\Big) N^{\ell}(dt)\bigg)\bigg\vert_{\eps=0} \\
&  = \int_0^{\tmax}h_t(o) \, N^{\ell}(dt). 
\end{align*}
This shows that the tangent space is given by:
\begin{align*}
  \mathscr{T}_{\mu} (P)
  = \bigg\lbrace \int_0^{\tmax} h_t  \,  N^{\ell}(dt) \, :\,
    h_t \in   \sigma ( L(t),N^{\ell}(t), \F_{t-})  \, \wedge \,  \\
    \EE_{P} \big[ h_t  \, \big\vert \,  N^{\ell}(t), \F_{t-}\big] = 0 \bigg\rbrace  \cap L_2(P). 
\end{align*}\end{proof}

The following lemma deals with the tangent spaces
\(\mathscr{T}_{\Lambda^a}\), \(\mathscr{T}_{\Lambda^\ell}\),
\(\mathscr{T}_{\Lambda^d}\).

\begin{lemma}
  The tangent space \(\mathscr{T}_{\Lambda^x}\) associated with the
  intensity \(\Lambda^x\) of \(N^x\) is given by: 
  \begin{align*}
    \mathscr{T}_{\Lambda^x} (P)
    & = \bigg\lbrace \int_0^{\tmax} h_t\,
      \big( N^{x}(dt) - \Lambda^x(dt)\big)\, :\,
      h_t
      \in \F_{t-}
      \bigg\rbrace \cap L_2(P), \quad x = a, \ell, d.
  \end{align*}
  \label{lemma:tangent:lambda}
\end{lemma}

\begin{proof} For \(x=a, \ell, d\), let
  \(h_t \, : \, \mathcal{O} \rightarrow \R\) be such that
  \( h_t \in \F_{t-}\) and consider the class of submodels
\begin{align*}
  \lambda^x_{\eps, h_t} (t \, \vert \, \F_{t-})
  = \exp\big( \eps h_t (o) \big) \lambda^x(t \, \vert \, \F_{t-}). 
\end{align*}
These are valid submodels since
\(\lambda^x_{0, h_t} (t \, \vert \, \F_{t-})= \lambda^x(t \, \vert \,
\F_{t-})\) and \(\lambda^x_{\eps, h_t} (t \, \vert \, \F_{t-})>0 \)
for all \(t\). The contribution to the log-likelihood is
\begin{align*}
  \log {dP_{ \eps, h_t}}(o)  =  \int_0^{\tmax} \big(  \eps h_t (o) +
  \log \lambda^x(t \, \vert\, \F_{t-}) \big)
  N^{x}(dt)\,  - \\ 
  \int_0^{\tmax}\exp \big( \eps h_t (o) \big)
  \, \lambda^x(t \, \vert \, \F_{t-}) dt, 
\end{align*}
and we derive the score
\begin{align*}
  \frac{d}{d\eps} \log {dP_{ \eps, h_t}}(o) \bigg\vert_{\eps=0}  
  & = \frac{d}{d\eps} \bigg( \int_0^{\tmax} \big(  \eps h_t (o) +
    \log \lambda^x(t \, \vert\, \F_{t-})dt \big)
    N^{x}(dt) \, - \\
  &\qquad\qquad\qquad\quad     \int_0^{\tmax}\exp \big( \eps h_t (o) \big)
    \, \lambda^x(t \, \vert \, \F_{t-})dt \bigg)\bigg\vert_{\eps=0} 
  \\
  &  = \int_0^{\tmax}h_t(o) \, N^{x}(dt) - \int_0^{\tmax}h_t(o)\, \lambda^x(t \, \vert \, \F_{t-})dt
    . 
\end{align*}
Thus, the tangent space is given by: 
\begin{align*}
  \mathscr{T}_{\Lambda^x} (P) &= \bigg\lbrace \int_0^{\tmax}h_t\,
                                \big( N^{x}(dt) - \lambda^x(t \, \vert \, \F_{t-})dt\big)\, :\,
                                h_t
                                \in \F_{t-}
                                \bigg\rbrace   , 
\end{align*}
for \(x=a, \ell, d\). \end{proof}

\subsection{Projection onto tangent spaces}
\label{sec:deriving:eff:ic}

We have characterized the tangent space \(\mathscr{T}_G\) and provided
an initial gradient \(D^{G^*}_{\mathrm{IPW}}(P)\). We can now compute
the canonical gradient by projecting \(D^{G^*}_{\mathrm{IPW}}(P)\)
onto \(\mathscr{T}_G\).

We denote the projection onto a tangent space \(\mathscr{T}\) by
\(\Pi( \cdot \, \vert\, \mathscr{T})\).  The orthogonal sum
decomposition of \(\mathscr{T}_G\) implies that:
\begin{align} 
  \begin{split}
    \Pi( D^{G^*}_{\mathrm{IPW}}(P) \, \vert\, \mathscr{T}_G(P)) =
    \Pi(D^{G^*}_{\mathrm{IPW}}(P) \, \vert\, \mathscr{T}_{\mu_{L_0}}(P)) +
    \Pi(D^{G^*}_{\mathrm{IPW}}(P) \, \vert\, \mathscr{T}_{\mu}(P)) \\ +
    \sum_{x=a,\ell,d} \Pi(D^{G^*}_{\mathrm{IPW}}(P) \, \vert\,
    \mathscr{T}_{\Lambda^x}(P)).  \end{split}
  \label{eq:sum:projections}
\end{align}
The individual projections are given in the following lemmas.

\begin{lemma}
  The projection of \(D^{G^*}_{\mathrm{IPW}} (P)\) onto the tangent
  space \(\mathscr{T}_{\mu_{L_0}}(P)\) is given by:
\begin{align*}
  \Pi ( D^{G^*}_{\mathrm{IPW}} (P) \, \vert \, \mathscr{T}_{\mu_{L_0}} (P) ) =
  \EE_{P^{G^*}} [Y  \, \vert \, \F_0] - \Psi^{G^*} (P). 
\end{align*}
\label{lemma:project:L0}
\end{lemma}

\begin{proof} Recall that
  \( \mathscr{T}_{\mu_{L_0}} (P) = \big\lbrace h \in \F_0 \, : \,
  \EE_P[ h ] = 0\big\rbrace \cap L_2(P)\). We suggest that the
  projection of \(D^{G^*}_{\mathrm{IPW}}(P)\) onto
  \(\mathscr{T}_{\mu_{L_0}}(P)\) is given by:
\begin{align}
  & \widetilde{\Pi} ( D^{G^*}_{\mathrm{IPW}} (P)\, \vert \, \mathscr{T}_{\mu_{L_0}}(P)) =
    \EE[ D^{G^*}_{\mathrm{IPW}} (P) \, \vert \, L_0 ] - \EE[ D^{G^*}_{\mathrm{IPW}} (P)  ].
    \label{eq:first:projection}
\end{align}
We verify that \eqref{eq:first:projection} is in fact the projection
by first noting that:
\begin{align*}
  & \widetilde{\Pi} ( D^{G^*}_{\mathrm{IPW}} (P)\, \vert \, \mathscr{T}_{\mu_{L_0}}(P)) \in  \mathscr{T}_{\mu_{L_0}} (P).
\end{align*}
Fix an arbitrary element \(S\in \mathscr{T}_{\mu_{L_0}}(P)\). We see
that:
\begin{align*}
  & \EE \Big[ \Big( D^{G^*}_{\mathrm{IPW}} (P)-
    \widetilde{\Pi} ( D^{G^*}_{\mathrm{IPW}} (P)\, \vert \, \mathscr{T}_{\mu_{L_0}}(P))
    \Big) S \Big]  \\
  &\qquad  =  \EE \Big[ \Big( D^{G^*}_{\mathrm{IPW}} (P)-
    \big(  \EE[ D^{G^*}_{\mathrm{IPW}} (P) \, \vert \, L_0 ] - \EE[ D^{G^*}_{\mathrm{IPW}} (P)  ]
    \big) 
    \Big) S \Big]  \\
  &\qquad  =  \EE \big[ \EE[ D^{G^*}_{\mathrm{IPW}} (P) \, \vert \, L_0 ]
    S \big] = \EE \big[ \EE[ D^{G^*}_{\mathrm{IPW}} (P) \, \vert \, L_0 ] \EE  [
    S \, \vert \, L_0] \big]   = 0. 
\end{align*}
Hence,
\(\Pi ( D^{G^*}_{\mathrm{IPW}} (P)\, \vert \,
\mathscr{T}_{\mu_{L_0}}(P)) = \widetilde{\Pi} ( D^{G^*}_{\mathrm{IPW}}
(P)\, \vert \, \mathscr{T}_{\mu_{L_0}}(P)) \).  Finally we rewrite
\eqref{eq:first:projection} as follows:
\begin{align*}
  &
    \widetilde{\Pi} ( D^{G^*}_{\mathrm{IPW}} (P)\, \vert \, \mathscr{T}_{\mu_{L_0}}(P)) \\
  & \quad = \EE\bigg[
\Prodi_{t\le \tmax}    \frac{d{G}_{t}^{*}(O)}{d{G}_{t}(O)} Y -
    \Psi^{G^*}(P) \, \bigg\vert \, L_0 \bigg] -
    \EE\bigg[ \Prodi_{t\le \tmax} \frac{d{G}_{t}^{*}(O)}{d{G}_{t}(O)} Y - \Psi^{G^*}(P) \bigg]\\
  & \quad = \EE_{P^{G^*}} [Y \, \vert\, L_0] - \Psi^{G^*}(P),
\end{align*}
noting that \(\F_0 = \sigma(L_0)\) and using the same technique as in
Section \ref{sec:initial:gradient}. This completes the
proof. \end{proof}

\begin{lemma}
  The projection of \(D^{G^*}_{\mathrm{IPW}} (P)\) onto the tangent
  space \(\mathscr{T}_{\mu}(P)\) is given by: 
\begin{align*}
  &\Pi ( D^{G^*}_{\mathrm{IPW}} (P) \, \vert \, \mathscr{T}_{\mu} (P) )
  \\
  &\,\, = \int_0^{\tmax} \Big(\EE [D^{G^*}_{\mathrm{IPW}} (P) \, \vert
    \,
    L(t),N^{\ell}(t),\F_{t-}] - \EE [D^{G^*}_{\mathrm{IPW}} (P) \, \vert \, N^{\ell}(t),\F_{t-}] \Big) N^{\ell}(dt). 
\end{align*}
\label{lemma:project:Lt}
\end{lemma}

\begin{proof} We suggest that the projection is given by:
  \begin{align*}
    \begin{split}
    &\widetilde{\Pi} ( D^{G^*}_{\mathrm{IPW}} (P) \, \vert \, \mathscr{T}_{\mu} (P) )
    \\
    &\, = \int_0^{\tmax} \Big(\EE [D^{G^*}_{\mathrm{IPW}} (P) \, \vert
      \,
      L(t), N^{\ell}(t),\F_{t-}] - \EE [D^{G^*}_{\mathrm{IPW}} (P) \, \vert \, N^{\ell}(t),\F_{t-}] \Big) N^{\ell}(dt).
    \end{split}
  \end{align*}
  To verify that
  \(\widetilde{\Pi} ( D^{G^*}_{\mathrm{IPW}} (P) \, \vert \,
  \mathscr{T}_{\mu} (P) )\) is in fact the projection, first note that
  \(\widetilde{\Pi} \big( D^{G^*}_{\mathrm{IPW}} (P) \, \vert \,
  \mathscr{T}_{\mu} (P) \big)\in \mathscr{T}_{\mu} (P)\). What remains to
  be shown is that:
\begin{align}
  \EE \left[\Big( D^{G^*}_{\mathrm{IPW}} (P)-
  \widetilde{\Pi} \big( D^{G^*}_{\mathrm{IPW}} (P) \, \vert \, \mathscr{T}_{\mu} (P) \big)\Big)  S \right]=0,
  \label{eq:covariance:0:Lt}
\end{align}
for any \(S\in \mathscr{T}_{\mu}(P)\). Since
\(S\in \mathscr{T}_{\mu}(P)\), we can write
\begin{align*}
  S = \int_0^{\tmax} h^S_t \, N^{\ell}(dt) , 
\end{align*}
for some \(h^S_t \in \sigma( L(t), N^{\ell}(t),\F_{t-})\) such that
\(\EE_{P} \big[ h^S_t \, \big\vert \, N^{\ell}(t),\F_{t-} \big] =
0\). Now, the second term of \eqref{eq:covariance:0:Lt} can be written
as follows
\begin{align*}
  &\EE\left[ \widetilde{\Pi} \big( D^{G^*}_{\mathrm{IPW}} (P) \, \vert \, \mathscr{T}_{\mu} (P) \big)
    \, S \right]
  \\
  &\quad = \EE\bigg[  \int_0^{\tmax} \EE\Big[ \Big(\EE [D^{G^*}_{\mathrm{IPW}} (P) \, \vert
    \, 
    L(t), N^{\ell}(t),\F_{t-}
    ]  \, - \\
  &\qquad\qquad\qquad\qquad\qquad\qquad\quad\,\,
    \EE [D^{G^*}_{\mathrm{IPW}} (P) \, \vert \, N^{\ell}(t),\F_{t-}] \Big)\, h^S_t
    \, \Big\vert \, N^{\ell}(t),\F_{t-} \Big] N^{\ell}(dt)  \bigg]\\
  &\quad =   \EE\bigg[ \int_0^{\tmax}\Big(\EE  \big[ \EE [h^S_t D^{G^*}_{\mathrm{IPW}} (P) \, \vert
    \, L(t), N^{\ell}(t),\F_{t-} ]\big] \, - \\
  &\qquad\qquad\qquad\qquad\qquad\quad\,
    \EE\big[\EE [D^{G^*}_{\mathrm{IPW}} (P) \, \vert \, N^{\ell}(t),\F_{t-}] \underbrace{
    \EE[h^S_t \, \vert \, N^{\ell}(t),\F_{t-}]}_{=0}
    \big]\Big) N^{\ell}(dt) \bigg] \\
  & \quad = \EE \left[ \int_0^{\tmax} h^S_t D^{G^*}_{\mathrm{IPW}} (P) \, N^{\ell}(dt)\right]
    = \EE \big[ D^{G^*}_{\mathrm{IPW}}(P)  S \big],
\end{align*}
where we have used the law of iterated expectations. Hence, we
conclude that
\(\widetilde{\Pi} \big( D^{G^*}_{\mathrm{IPW}} (P) \, \vert \,
\mathscr{T}_{\mu} (P) \big) = \Pi \big( D^{G^*}_{\mathrm{IPW}} (P) \,
\vert \, \mathscr{T}_{\mu} (P) \big)\).  \end{proof}

\begin{lemma}
  The projection of \(D^{G^*}_{\mathrm{IPW}} (P)\) onto the tangent
  space \(\mathscr{T}_{\Lambda^x}(P)\) is given by:
\begin{align*}
  &\Pi ( D^{G^*}_{\mathrm{IPW}} (P) \, \vert \, \mathscr{T}_{\Lambda^x} (P) )
    = \int_0^{\tmax} \Big(\EE [D^{G^*}_{\mathrm{IPW}} (P) \, \vert
    \, \Delta N^x(t)=1,
    \F_{t-}] \, \,- \\
  &\qquad\qquad\qquad \EE [D^{G^*}_{\mathrm{IPW}} (P) \, \vert \, \Delta N^x(t)=0,
    \F_{t-}] \Big) \big( N^{x}(dt) - \Lambda^x(dt)\big),
\end{align*}
for \(x=a,\ell,d\).
\label{lemma:project:lambda}
\end{lemma}

\begin{proof} Define:
  \begin{align}
     \begin{split}
  &\widetilde{\Pi} ( D^{G^*}_{\mathrm{IPW}} (P) \, \vert \, \mathscr{T}_{\Lambda^x} (P) )
    = \int_0^{\tmax} \Big(\EE [D^{G^*}_{\mathrm{IPW}} (P) \, \vert
    \, \Delta N^x(t)=1,
    \F_{t-}] \, \,- \\
  &\qquad\qquad\qquad \EE [D^{G^*}_{\mathrm{IPW}} (P) \, \vert \, \Delta N^x(t)=0,
  \F_{t-}] \Big) \big( N^{x}(dt) - \Lambda^x(dt)\big).
\end{split} \label{eq:tildepi:lambda}
  \end{align}
  To verify that
  \(\widetilde{\Pi} ( D^{G^*}_{\mathrm{IPW}} (P) \, \vert \,
  \mathscr{T}_{\Lambda^x} (P) )\) is in fact the projection, first
  note that
  \( \widetilde{\Pi} \big( D^{G^*}_{\mathrm{IPW}} (P) \, \vert \,
  \mathscr{T}_{\Lambda^x} (P) \big) \in
  \mathscr{T}_{\Lambda^x}(P)\). What remains to be shown is that
  \(D^{G^*}_{\mathrm{IPW}}- \widetilde{\Pi} \big(
  D^{G^*}_{\mathrm{IPW}} (P) \, \vert \, \mathscr{T}_{\Lambda^x}(P)
  \big)\) is orthogonal to all \(S\in \mathscr{T}_{\Lambda^x}(P)\),
  i.e.,
\begin{align}
  \EE \left[\Big( D^{G^*}_{\mathrm{IPW}} (P)-
  \widetilde{\Pi} \big( D^{G^*}_{\mathrm{IPW}} (P) \, \vert \, \mathscr{T}_{\Lambda^x}(P) \big)\Big)  S \right]=0.
  \label{eq:covariance:0:lambda}
\end{align}
Since \(S\in \mathscr{T}_{\Lambda^x}(P)\), we can write:
\begin{align*}
  S = \int_0^{\tmax} h^S_t \, M^{x}(dt) , 
\end{align*}
for some \(h^S_t \in \F_{t-}\).

For the proof, we will use that \cite[see, e.g., ][Theorem 2.6.1]
{fleming2011counting}:
\begin{align*}
  \EE[ M^x(dt)^2 \mid \F_{t-} ]=   \langle M^x, M^x\rangle (dt) =  (1- \Delta \Lambda^x(t))
  d\Lambda^x(t), 
\end{align*}
where \(\Delta \Lambda^x(t) = \Lambda^x(t) - \Lambda^x(t-)
\). Furthermore, by \citet[][Theorem 2.4.4] {fleming2011counting},
\begin{align*}
  \EE \bigg[ \int_0^t H_1 (s) dM^x(s) \int_0^t H_2(u) dM^x(u) \bigg]
  = \EE \bigg[ \int_0^t H_1(s) H_2(s) \langle M^x, M^x\rangle (ds) \bigg], 
\end{align*}
for \(H_1 (t), H_2(t) \in \F_{t-}\).

Applying this to the second term of \eqref{eq:covariance:0:lambda}, we
see that
\begin{align}
  \begin{split}
    &\EE \left[ \widetilde{\Pi} \big( D^{G^*}_{\mathrm{IPW}} (P) \,
      \vert \, \mathscr{T}_{\Lambda^x}(P) \big) S \right] = \EE\bigg[
    \int_0^{\tau} \Big(\EE [D^{G^*}_{\mathrm{IPW}} (P) \, \vert \,
    \Delta N^x(t)=1,
    \F_{t-}]  \\
    &\qquad\qquad\quad -\, \EE [D^{G^*}_{\mathrm{IPW}} (P) \,
    \vert \, \Delta N^x(t)=0, \F_{t-}] \Big) h^S_t \, \big( 1-\Delta
    \Lambda^x(t)\big) \Lambda^x (dt)\bigg]
    .\end{split} \label{eq:first:term:11}
\end{align}
We want to show that
\begin{align*}
  \EE \left[D^{G^*}_{\mathrm{IPW}} (P)S \right] =
  \EE \left[ 
    \widetilde{\Pi} \big( D^{G^*}_{\mathrm{IPW}} (P) \, \vert \,
    \mathscr{T}_{\Lambda^x}(P) \big)  S \right] .
\end{align*}
We now rewrite the left hand side of the previous display as follows:
\begin{align*}
  & \EE \left[D^{G^*}_{\mathrm{IPW}} (P)S \right] \\
  &\quad =
    \EE \bigg[  D^{G^*}_{\mathrm{IPW}} (P) \int_0^\tau h^S_t M^x(dt) \bigg]
    =    \EE \bigg[ \int_0^\tau  D^{G^*}_{\mathrm{IPW}} (P) h^S_t M^x(dt) \bigg]; 
    \intertext{for small \(\delta>0\) we apply iterated expectations }
  &\quad =  \EE \bigg[  \int_0^{\tau}   h^S_t 
    \EE [ D^{G^*}_{\mathrm{IPW}} (P)   \big(N^x(dt)- \Lambda^x(dt)\big) \, \big\vert
    \, N^x(t+\delta) - N^x(t-),
    \F_{t-}]  
    \bigg] \\
  &\quad =  \EE \bigg[  \int_0^{\tau}   h^S_t  \bigg(\sum_{u= 0,1}
    \EE [ D^{G^*}_{\mathrm{IPW}} (P) \, \big\vert
    \, N^x(t+\delta) - N^x(t-) =u,
    \F_{t-}]
  \\[-0.2em]
  &\qquad\qquad\qquad\qquad\qquad
    \1\lbrace N^x(t+\delta) - N^x(t-) =u\rbrace  \bigg)
    \big(N^x(dt)- \Lambda^x(dt)\big)
    \bigg] ,
    \intertext{ which in turns when letting \(\delta\rightarrow 0\) yields}
  &\quad =  \EE \bigg[  \int_0^{\tau}   h^S_t  \bigg(\sum_{u= 0,1}
    \EE [ D^{G^*}_{\mathrm{IPW}} (P)\, \big\vert
    \,  \Delta N^x(t) =u,
    \F_{t-}]  
  \\[-0.8em]
  &\qquad\qquad\qquad\qquad\quad\quad
    \big( \Delta N^x(t)\big)^u \big( 1-   \Delta N^x(t)\big)^{1-u} \bigg) \big(N^x(dt)- \Lambda^x(dt)\big)
    \bigg].
    \intertext{Now define \(H_t^u := h_t^S \EE [ D^{G^*}_{\mathrm{IPW}} (P)\, \big\vert
    \,  \Delta N^x(t) =u,
    \F_{t-}]\) for \(u=0,1\), and note that \(H_t^u \in \F_{t-}\); then we can write the above formula as }
  & \quad \begin{cases}
    \EE \big[  \int_0^{\tau}  H_t^1 \Delta N^x(t)    \big(N^x(dt)- \Lambda^x(dt)\big)\big], &\text{ if }
    u=1, \\
    \EE \big[  \int_0^{\tau}  H_t^0 \big(1-\Delta N^x(t)\big) \big(N^x(dt)- \Lambda^x(dt)\big)\big]
    , &\text{ if }
    u=0,
  \end{cases}
        \intertext{which, since  \( \int_0^{\tau} H_t^0 (N^x(dt) - \Lambda^x(dt))\) is a
        martingale such that
        \(\EE[ \int_0^{\tau} H_t^0 (N^x(dt) - \Lambda^x(dt))] =0\), is}
  &\quad  \begin{cases}
    \EE \big[  \int_0^{\tau}  H_t^1 \Delta N^x(t)    \big(N^x(dt)- \Lambda^x(dt)\big)\big], &\text{ if }
    u=1, \\
    -\EE \big[  \int_0^{\tau}  H_t^0 \Delta N^x(t) \big(N^x(dt)- \Lambda^x(dt)\big)\big]
    , &\text{ if }
    u=0,
  \end{cases}
        \intertext{so that  we can now continue and write}
  & \quad = \EE \bigg[  \int_0^{\tau} \big( H_t^1-  H_t^0\big) \Delta N^x(t) \big(N^x(dt)- \Lambda^x(dt)
    \big)\bigg] \\
  &\quad \overset{*}{=}
    \EE \bigg[  \int_0^{\tau} \big( H_t^1-  H_t^0\big)  \big(N^x(dt)-\Delta N^x(t)  \Lambda^x(dt)\big)
    \bigg] \\
  &\quad \overset{**}{=}
    \EE \bigg[  \int_0^{\tau} \big( H_t^1-  H_t^0\big)
    \big( \EE [N^x(dt) \mid \F_{t-}]
    -\EE[\Delta N^x(t) \mid \F_{t-}]  \Lambda^x(dt)\big)
    \bigg] \\
  &\quad =
    \EE \bigg[  \int_0^{\tau} \big( H_t^1-  H_t^0\big)
    \big(  \Lambda^x(dt)
    -\Delta \Lambda^x(t)   \Lambda^x(dt)\big)
    \bigg] \\
  &\quad =
    \EE \bigg[  \int_0^{\tau} \big( H_t^1-  H_t^0\big)
    \big( 1
    -\Delta \Lambda^x(t) \big)  \Lambda^x(dt)
    \bigg] \\
  &\quad=  \EE \bigg[  \int_0^{\tau}
    h_t^S \big( \EE [ D^{G^*}_{\mathrm{IPW}} (P)\, \big\vert
    \,  \Delta N^x(t) =1,
    \F_{t-}]  \\[-0.7em]
  &\qquad\qquad\quad - \,  \EE [ D^{G^*}_{\mathrm{IPW}} (P)\, \big\vert
    \,  \Delta N^x(t) =0,
    \F_{t-}] \big)  \big( 1
    -\Delta \Lambda^x(t) \big)  \Lambda^x(dt) \bigg]. 
\end{align*}
At \(*\) we used that: 
\begin{align*}
  \int_0^\tau \Delta N^x(t) \big(N^x(dt)- \Lambda^x(dt)\big)
  &= \sum_{t\le \tau}
    \Delta N^x(t) \big( \Delta N^x(t) - \Delta \Lambda^x(t) \big) \\
  &= \sum_{t\le \tau}
    \Delta N^x(t) - \Delta N^x(t) \Delta \Lambda^x(t) \\
  & = \int_0^\tau  \big(N^x(dt)- \Delta N^x(t) \Lambda^x(dt)\big),
\end{align*}
and at \(**\) we used that \(H^1_t, H^0_t \in \F_{t-}\). We see that
the right hand side displayed above is equal to the right hand side of
\eqref{eq:first:term:11}. Thus, we conclude that
\begin{align*}
  &\EE \left[ 
    \widetilde{\Pi} \big( D^{G^*}_{\mathrm{IPW}} (P) \, \vert \, \mathscr{T}_{\Lambda^x}(P) \big)  S \right]  
    = \EE \big[  D^{G^*}_{\mathrm{IPW}} (P)  S \big],
\end{align*}
which verifies \eqref{eq:covariance:0:lambda}. Hence,
\(\widetilde{\Pi} ( D^{G^*}_{\mathrm{IPW}} (P) \, \vert \,
\mathscr{T}_{\Lambda^x} (P) ) = \Pi ( D^{G^*}_{\mathrm{IPW}} (P) \,
\vert \, \mathscr{T}_{\Lambda^x} (P) )\).
\end{proof}

\subsection{Proof of Theorem \ref{thm:eff:ic}}
\label{sec:proof:thm:eff}

We are now ready to prove Theorem \ref{thm:eff:ic} by applying the
results of Section \ref{sec:deriving:eff:ic}. For this purpose, recall
that \(dG_t(O)\) and \(dQ_t(O)\) denote the conditional measures of
the interventional and non-interventional parts of the likelihood at
time \(t\) given the observed history \(\F_{t-}\). In the proof we
will use the following equation repeatedly:
\begin{align*}
  &\EE_{ P^{G^*} }\big[ Y \mid \F_{t-}\big]
    =
    \int y\,  \Prodi_{s \ge t} dQ_{s}(o) \, dG_s^*(o)  \\
  &=
    \int y\,  \Prodi_{s \ge t} d{G}_{s}^*(o)  \,\Prodi_{s \ge t} \frac{ d{G}_{s}(o)
    }{ d{G}_{s}(o)} \,
    \Prodi_{s \ge t} dQ_{s}(o)  \\
  &=
    \int \bigg( y \, \Prodi_{s \ge t} \frac{d{G}^*_{s}(o)
    }{ d{G}_{s}(o)}  \bigg)
    \Prodi_{s \ge t} dQ_{s}(o) \, dG_s(o)  
    = \EE \bigg[ Y \,\Prodi_{s \ge t}  \frac{ d{G}^*_{s}(O)
    }{d{G}_{s}(O)} \, \bigg\vert \, \F_{t-} \bigg] . 
\end{align*}

\begin{proof} (Theorem \ref{thm:eff:ic}). Collecting the results of
  Lemma \ref{lemma:project:L0}--\ref{lemma:project:lambda}, we can now
  compute the canonical gradient by the decomposition of the
  projection operator from display \eqref{eq:sum:projections} as
  follows:
  \begin{align*}
    &D^*(P) = \Pi( D^{G^*}_{\mathrm{IPW}} (P) \mid
      \mathscr{T}_G) \\
    & \,\, =   \EE_{P^{G^*}} [Y  \, \vert \, \F_0] - \Psi^{G^*} (P)  \\
    & \quad + \int_0^{\tmax} \Big(\EE [D^{G^*}_{\mathrm{IPW}} (P) \, \vert
      \,
      L(t),N^{\ell}(t),\F_{t-}] - \EE [D^{G^*}_{\mathrm{IPW}} (P) \, \vert
      \, N^{\ell}(t),\F_{t-}] \Big) N^{\ell}(dt) \\
    &  \quad +     \sum_{x\in\lbrace \ell,a,d\rbrace}
      \int_0^{\tmax} \Big(\EE [D^{G^*}_{\mathrm{IPW}} (P) \, \vert
      \, \Delta N^x(t)=1,
      \F_{t-}] \, \,- \\[-0.8em]
    &\qquad\qquad\qquad\qquad\qquad \EE [D^{G^*}_{\mathrm{IPW}} (P) \, \vert \, \Delta N^x(t)=0,
      \F_{t-}] \Big) \big( N^{x}(dt) - \Lambda^x(dt)\big).
  \label{eq:orth:sum}
  \end{align*}
  Recall that:
  \begin{align*}
    D^{G^*}_{\mathrm{IPW}}(P) (O) =
    \Prodi_{s \le \tau }  \frac{dG^*_{s}(O)}{dG_s(O)}\,
    Y - \Psi^{G^*}(P).
  \end{align*}
  Consider first:
  \begin{align*}
    &\EE [D^{G^*}_{\mathrm{IPW}} (P) \, \vert
      \,
      L(t),N^{\ell}(t),\F_{t-}] - \EE [D^{G^*}_{\mathrm{IPW}} (P) \, \vert
      \, N^{\ell}(t),\F_{t-}]  \\
    &\quad = \EE \bigg[ \bigg( \Prodi_{s <t } \frac{dG^*_s}{dG_s} \bigg) \bigg(
      \Prodi_{s \ge t } \frac{dG^*_s}{dG_s} \bigg)
      \,  Y  \, \Big\vert
      \, L(t),N^{\ell}(t),\F_{t\minus}\bigg] \\
    &\qquad\qquad\qquad - \,\,
      \EE \bigg[ \bigg( \Prodi_{s <t } \frac{dG^*_s}{dG_s} \bigg) \bigg(
      \Prodi_{s \ge t } \frac{dG^*_s}{dG_s} \bigg)
      \,  Y  \, \Big\vert
      \, N^{\ell}(t),\F_{t\minus}\bigg] \\
    &\quad =    \Prodi_{s <t } \frac{dG^*_s}{dG_s}
      \bigg(       \EE \bigg[ \bigg(
      \Prodi_{s \ge t } \frac{dG^*_s}{dG_s} \bigg)
      \,  Y  \, \Big\vert
      \, L(t),N^{\ell}(t),\F_{t\minus}\bigg] \\
    &\qquad\qquad\qquad\qquad\qquad\,\, - \,\,
      \EE \bigg[  \bigg(
      \Prodi_{s \ge t } \frac{dG^*_s}{dG_s} \bigg)
      \,  Y  \, \Big\vert
      \, N^{\ell}(t),\F_{t\minus}\bigg] \bigg) \\
    &\quad =    \Prodi_{s <t } \frac{dG^*_s}{dG_s}
      \Big(       \EE_{P^{G^*}} \big[
      \,  Y  \, \vert
      \, L(t),N^{\ell}(t),\F_{t\minus}\big] - \EE_{P^{G^*}} \big[
      \,  Y  \, \vert
      \, N^{\ell}(t),\F_{t\minus}\big]\Big), 
  \end{align*}
  and, by the same line of argument:
  \begin{align*}
    & \EE [D^{G^*}_{\mathrm{IPW}} (P) \, \vert
      \, \Delta N^x(t)=1,
      \F_{t-}] -  \EE [D^{G^*}_{\mathrm{IPW}} (P) \, \vert \, \Delta N^x(t)=0,
      \F_{t-}] \\
    &\quad = \Prodi_{s <t } \frac{dG^*_s}{dG_s}
      \Big(       \EE_{P^{G^*}} \big[
      \,  Y  \, \vert
      \, \Delta N^x(t)=1,\F_{t\minus}\big] - \EE_{P^{G^*}} \big[
      \,  Y  \, \vert
      \,\Delta N^x(t)=0,\F_{t\minus}\big]\Big). 
  \end{align*}
  Finally, note that:
  \begin{align*}
    \EE_{P^{G^*}} \big[
    \,  Y  \, \vert
    \, \Delta N^d(t)=1,\F_{t\minus}\big] = 1.
  \end{align*}
  This gives the expression for the canonical gradient displayed in
  Theorem \ref{thm:eff:ic}.  \end{proof}

\subsection{Corollary \ref{cor:eic:int:dZ}}
\label{sec:cor:eff}

The following corollary provides an alternative representation of the
canonical gradient, that is utilized to analyze the estimation problem
in Section \ref{sec:remainder}.  In the following, we use
\(\bar{\mathcal{O}}_{t}\) to denote the space where \(\bar{O}(t)\)
takes its values and \(\ubar{\mathcal{O}}_{t}\) the space where
\(\ubar{O}(t) = \lbrace O(s) \, : \, t \le s \le \tmax\rbrace\) takes
its values.

\begin{cor}[Canonical gradient]
  We can rewrite the canonical gradient from Theorem \ref{thm:eff:ic}
  for \(\Psi^{G^*} \, : \, \mathcal{M}\rightarrow \R\) as follows: 
\begin{align*}
  &  D^*(P) = \EE_{P^{G^*}} [ Y \, \vert \,  \F_0] - \Psi^{G^*} (P)  \\
  &\qquad +
    \int_0^{\tmax} \Prodi_{s <t }  \frac{dG^*_s}{dG_s} 
    \int_{\ubar{\mathcal{O}}_t} Y \bigg(  \Prodi_{s > t} dQ_s \Prodi_{s \ge t} dG^*_s  -
    \Prodi_{s \ge t} dQ_s dG^*_s\bigg). 
\end{align*}
\label{cor:eic:int:dZ}
\end{cor}

\begin{proof}
  Defining
  \begin{align*}
  Z^{G^*}(dt) := \EE_{P^{G^*}} [Y \, \vert \, L(t),
  N^{\ell}(dt),N^{a}(dt), N^{d}(dt), \F_{t-}] - \EE_{P^{G^*}} [Y \,
    \vert \, \F_{t-} ],
  \end{align*}
  the last line of the representation for \(D^{G^*}(P)\) displayed in
  the corollary can be written
\begin{align*}
\int_0^{\tmax} \Prodi_{s <t }  \frac{dG^*_s}{dG_s} 
  \int_{\ubar{\mathcal{O}}_t} Y \bigg(  \Prodi_{s > t} dQ_s \Prodi_{s \ge t} dG^*_s  -
  \Prodi_{s \ge t} dQ_s dG^*_s\bigg)
  = \int_0^{\tmax} \bigg( \Prodi_{s <t }  \frac{dG^*_s}{dG_s} \bigg)
  Z^{G^*}(dt).
\end{align*}  
To prove the corollary, it is sufficient to show that
\begin{align*}
  Z^{G^*}(dt)
  & =   \Big(  \EE[ Y \mid L(t) , \Delta N^\ell(t)=1, \F_{t-}]  \\[-0.3em]
  &\qquad\qquad\qquad\qquad\qquad\quad - \,
    \EE[ Y \mid L(t) , \Delta N^\ell(t)=1, \F_{t-}] \Big) N^\ell(dt) \\
  & \qquad +  \sum_{x\in\lbrace\ell,a,d\rbrace}
    \Big( \EE_{P^{G^*}} [Y \mid \Delta N^x(t) =1 , \F_{t-}]  \\[-0.85em]
  &\qquad\qquad\qquad\quad  -
    \, \EE_{P^{G^*}} [Y \mid \Delta N^x(t) =0 , \F_{t-}] \Big) \big(N^x (dt)-\Lambda^x(dt)\big). 
\end{align*}
We expand \(Z^{G^*}(dt)\) as follows
\begin{align*}
  & \EE_{P^{G^*}} [Y \, \vert \, L(t),
    N^{\ell}(dt),N^{a}(dt), N^{d}(dt), \F_{t-}] \\
  &\underbrace{\qquad\qquad\qquad\qquad\qquad - \,
    \EE_{P^{G^*}} [Y \, \vert \, 
    N^{\ell}(dt),N^{a}(dt), N^{d}(dt), \F_{t-}]}_{=:Z_1^{G^*} (dt)} \\
  &\qquad\, + \,\,  \underbrace{\EE_{P^{G^*}} [Y \, \vert \,
    N^{\ell}(dt),N^{a}(dt), N^{d}(dt), \F_{t-}]  - \EE_{P^{G^*}} [Y \,
    \vert \, \F_{t-} ] }_{=:Z_2^{G^*} (dt) }. 
\end{align*}
Since \(L(t)\) only changes at jumps of \(N^\ell(t)\), it follows that
\begin{align*}
  Z^{G^*}_1 (dt) =  \big(  \EE[ Y \mid L(t) , \Delta N^\ell(t)=1, \F_{t-}] -
  \EE[ Y \mid \Delta N^\ell(t)=1, \F_{t-}] \big) N^\ell(dt).
\end{align*}
Next, considering \(Z^{G^*}_2 (dt)\), we note that
\begin{align*}
  Z^{G^*}_2 (dt)
  &= \EE_{P^{G^*}} [Y \, \vert \, 
    N^{\ell}(dt),  N^a(dt),  N^d(dt), \F_{t-}] \\
  &  \qquad - \, \Lambda^\ell(dt) \, \EE_{P^{G^*}} [Y \, \vert \, \Delta N^\ell(t)=1
    , \F_{t-} ]  \\
  &\qquad -\,
    \Lambda^a(dt)\, \EE_{P^{G^*}} [Y \, \vert \, 
    \Delta N^a(t)=1, \F_{t-} ] \\
  & \qquad -\,
    \Lambda^d(dt) \,\EE_{P^{G^*}} [Y \, \vert \,\Delta N^d(t)=1, \F_{t-} ] \\
  &\qquad - \, \big(1-  \Lambda^\ell(dt) -  \Lambda^a(dt)- \Lambda^d(dt) \big) \\[-0.1em]
  &\quad\qquad\qquad
    \EE_{P^{G^*}} [Y \, \vert \,   \Delta N^\ell(t)=0, \Delta N^a(t)=0, \Delta N^d(t)=0, \F_{t-} ],
    \intertext{assuming for ease of presentation that the processes \(N^\ell\),
    \(N^a\) and \(N^d\) do not jump at the same time.  Collecting terms now  yields that}
    Z^{G^*}_2 (dt) &= \sum_{x\in\lbrace\ell,a,d\rbrace} \Big( \EE_{P^{G^*}} [Y \mid \Delta N^x(t) =1 , \F_{t-}]  \\[-0.75em]
  &\qquad\quad\qquad  -
    \, \EE_{P^{G^*}} [Y \mid \Delta N^x(t) =0 , \F_{t-}] \Big) \big(N^x (dt)-\Lambda^x(dt)\big) ,
\end{align*}
which finishes the proof.
 \end{proof}

\subsection{Representation of the second-order remainder
  \(R_2(P, P_0)\)}
\label{sec:remainder}

As in Section \ref{sec:cor:eff}, we use \(\bar{\mathcal{O}}_{t}\) to
denote the space where \(\bar{O}(t)\) takes its values and
\(\ubar{\mathcal{O}}_{t}\) the space where
\(\ubar{O}(t) = \lbrace O(s) \, : \, t \le s \le \tmax\rbrace\) takes
its values.  We now express \(P_0 D^*(P) \) for any
\(P\in \mathcal{M}\) using the representation of the canonical
gradient in Corollary \ref{cor:eic:int:dZ} as follows:
\begin{align}
  \int_{\mathcal{O}}
  & D^*(P) \, dP_0 \notag\\
  & \, = \int_{\bar{\mathcal{O}}_{t}} \int_0^{\tmax} \Prodi_{s <t }  \frac{dG^*_s}{dG_s}
    \bigg( \bigg(   \int_{\ubar{\mathcal{O}}_t} Y\Prodi_{s > t} dQ_s \Prodi_{s \ge t} dG^*_s
    \bigg) \Prodi_{s \le t} dQ_{0,s} \Prodi_{s < t} dG_{0,s}  \notag\\
  &\qquad\qquad\qquad\qquad\qquad\qquad - \,
    \bigg(   \int_{\ubar{\mathcal{O}}_t} Y    \Prodi_{s \ge t} dQ_s dG^*_s
    \bigg) \Prodi_{s <t} dQ_{0,s} dG_{0,s}\bigg)  
    \notag\\
  & \, =  \int_0^{\tmax}\int_{{\mathcal{O}}}
    \Prodi_{s <t }  \frac{dG^*_s}{dG_s}
    \bigg(
    Y \,\Prodi_{s \ge t} dG^*_s\, \Prodi_{s > t}dQ_s \Prodi_{s < t} dG_{0,s}\, \Prodi_{s \le t}dQ_{0,s}
    \, - \notag\\
  &\qquad\qquad\qquad\qquad\qquad\qquad\qquad\qquad
    Y \, \Prodi_{s \ge t} dG^*_s \,
    d Q_s \Prodi_{s < t} dG_{0,s} \,dQ_{0,s} \bigg) 
    \notag\\
  &  \, = \int_0^{\tmax}\int_{\mathcal{O}} Y\,
    \Prodi_{s <t }  \frac{dG^*_s}{dG_s}
    \Prodi_{s < t}  dG_{0,s} 
    d Q_{0,s}\,  \big( dQ_{0,t} - dQ_t \big)\Prodi_{s \ge t} dG^*_s\Prodi_{s > t} dQ_s
    \notag \\
  & \, =  \int_0^{\tmax}\int_{\mathcal{O}}  
    Y \, \Prodi_{s <t }  \frac{dG_{0,s}}{dG_s}
    \Prodi_{s < t} dG_{s}^*   dQ_{0,s} 
    \,\big( d Q_{0,t} - dQ_t \big) \Prodi_{s \ge t}  dG^*_s \Prodi_{s> t}  dQ_s 
    \notag \\
  & \, =  \int_0^{\tmax}\int_{\mathcal{O}}
    Y \, \bigg(  \Prodi_{s <t }  \frac{dG_{0,s}}{dG_s}
    -
    1\bigg)\,
    \Prodi_{s \le \tmax} dG_{s}^*  \Prodi_{s < t} dQ_{0,s} 
    \,\big( d Q_{0,t} - dQ_t \big) \Prodi_{s >t} dQ_s  \tag{\(*1\)}\\
  & \, \white =  +
    \int_0^{\tmax} \int_{\mathcal{O}}
    Y \,
    \Prodi_{s \le \tmax} dG_{s}^*  \Prodi_{s < t} dQ_{0,s} 
    \,\big( d Q_{0,t} - dQ_t \big) \Prodi_{s >t} dQ_s  \tag{\(*2\)}\\
  & \quad = \underbrace{R_2( P, P_0)}_{(*1)} +
    \underbrace{\Psi^{G^*}(P_0) - \Psi^{G^*}(P)}_{(*2)}.\notag
\end{align}
Next we apply the Duhamel equation \citep{andersen2012statistical}
which in our setting yields that:
\begin{align*}
  \Prodi_{s \le \tau} dQ_{0,s} - \Prodi_{s \le \tau} dQ_{s} =
  \int_{t \le \tau} \Prodi_{ s < t} dQ_{0,s} ( dQ_{0,t} - dQ_{t} ) \Prodi_{ s > t} dQ_s. 
\end{align*}
This allows us to establish the relation
\((*2) =\Psi^{G^*}(P_0) - \Psi^{G^*}(P) \), applying the Duhamel
equation at the fourth equality:
\begin{align*}
  & \Psi^{G^*}(P_0) - \Psi^{G^*}(P) = \EE_{P^{G^*}_0}\big[Y \big] - \EE_{P^{G^*}}\big[Y \big]\\
  &\qquad =\int_{\mathcal{O}} Y \,  \Prodi_{s\in [0,\tmax]} dG^*_{s}\, dQ_{0,s} -
    \int_{\mathcal{O}} Y \,  \Prodi_{s\in [0,\tmax]} dG^*_s\, dQ_s\\
  &\qquad =\int_{\mathcal{O}} Y \, \Prodi_{s\in [0,\tmax]} dG^*_s \,
    \bigg( \Prodi_{s\in [0,\tmax]}  dQ_{0,s}  -  \Prodi_{s\in [0,\tmax]} dQ_{s} \bigg)\\
  &\qquad =\int_{\mathcal{O}} Y \, \Prodi_{s\in [0,\tmax]} dG^*_s \, \bigg( \int_0^{\tmax}
    \Prodi_{s\in[0,t)}  dQ_{0,s} \, \big(dQ_{0,t}  - dQ_{t}  \big) \Prodi_{s\in(t,\tmax]}  dQ_{s}
    \bigg)\\
  &\qquad =\int_0^{\tmax}\int_{\mathcal{O}} Y \, \Prodi_{s\in [0,\tmax]} dG^*_s \,
    \Prodi_{s\in[0,t)}  dQ_{0,s} \, \big(dQ_{0,t}  - dQ_{t}  \big) \Prodi_{s\in(t,\tmax]}  dQ_{s}
    .
\end{align*}
This again implies that:
\begin{align}
  \Psi^{G^*}(P) - \Psi^{G^*}(P_0) = - P_0   D^*(P) + R_2(P,P_0),
  \label{eq:expansion1}
\end{align}
with the second-order remainder \(R_2(P,P_0)\) given by \((*1)\)
above:
\begin{align*}
  &  R_2(P,P_0)\\
  & \,\,=   \int_0^{\tmax}\int_{\mathcal{O}}
    Y \,  \bigg(  \Prodi_{s <t }  \frac{dG_{0,s}}{dG_s}
    -
    1\bigg)\,
    \Prodi_{s \le \tmax} dG_{s}^*  \Prodi_{s < t} dQ_{0,s} 
    \,\big( d Q_{0,t} - d{Q}_{t} \big) \Prodi_{s >t} d{Q}_{s} \\
  & \,\,= \int_0^{\tmax} \int_{\mathcal{O}}
    Y \,  \frac{1}{\bar{g}_{t}}
    \big(  \bar{g}_{0,t} - \bar{g}_{t}
    \big) 
    \Prodi_{s \le \tmax} dG_{s}^*  \Prodi_{s < t} dQ_{0,s} 
    \,\big( d Q_{0,t} - d{Q}_{t}  \big) \Prodi_{s >t} d{Q}_{s}.
\end{align*}

\subsection{Proof of Lemma \ref{lemma:representation:Psi:Z}}
\label{app:proof:of:lemma:Psi:Z}

For the proof of Lemma \ref{lemma:representation:Psi:Z} recall that:
\begin{align*}
  Z_{t}^{G^*}&= 
               \int
               Y \,\Prodi_{s \ge t} dG^*_s\, \Prodi_{s > t}dQ_s  , \\
  Z_{t,L(t)}^{G^*} 
             &= \, \int \bigg( \int
               Y \,\Prodi_{s \ge t} dG^*_s\, \Prodi_{s > t}dQ_s \bigg) d\mu_{0,t} (L(t) \, \vert \,\F_{t-}).
\end{align*}
where \(dG_t(O)\) and \(dQ_t(O)\) denote the conditional measure of
the interventional and non-interventional parts of the likelihood at
time \(t\) given the observed history \(\F_{t-}\).  Particularly,
these are defined as:
\begin{align}
  \begin{split}
    dG_{t}(O) = \big( \pi_{t} (A(t)\, \vert \, \F_{t-})
    \big)^{N^{a}(dt)} \big(d\Lambda^{c} (t \, \vert \,
    \F_{t-}) \big)^{N^c(dt)} \\
    \big( 1 - d\Lambda^{c} (t \, \vert \, \F_{t-}) \big)^{1-N^c(dt)},
  \end{split}\label{eq:G:again}
\end{align}
and
\begin{align}
  \begin{split}
    dQ_{t}(O) &= \big( d\Lambda_0^{a} (t \, \vert \, \F_{t-})
    \big)^{N^{a}(dt)} \big( 1-d\Lambda^{a} (t \, \vert \,
    \F_{t-}) \big)^{1-N^{a}(dt)}\\
    & \qquad \big( d\Lambda^{\ell} (t \, \vert \, \F_{t-}) \,
    \mu_{0,t} (L(t)\, \vert \, \F_{t-}) \big)^{N^{\ell}(dt)} \big(
    1-d\Lambda^{\ell} (t \, \vert \,
    \F_{t-}) \big)^{1-N^{\ell}(dt)}\\
    &\qquad \big(d\Lambda^{d} (t \, \vert \, \F_{t-})\big)^{N^d(dt)}
    \big( 1 - d\Lambda^{d} (t \, \vert \, \F_{t-}) \big)^{1-N^d(dt)}.
  \end{split} \label{eq:Q:again}\end{align} Recall also that our
statistical model \(\mathcal{M}\) is defined by:
\begin{align*}
  \mathcal{M} = \bigg\lbrace P \, : \, dP = dP_{Q,G}= \Prodi_{t \in [0, \tmax]}  dQ_t dG_t ,\,  \, 
  G \in \mathcal{G}, \, Q \in \mathcal{Q}\bigg\rbrace ,
\end{align*}
so that any \(P \in\mathcal{M}\) admits a factorization into a
\(G\)-part and a \(Q\)-part.

The statement of Lemma \ref{lemma:representation:Psi:Z} is that the
target parameter \(\Psi^{G^*} \,: \, \mathcal{M} \rightarrow \R\),
defined by
\begin{align}
  \Psi^{G^*} (P)    = \int Y \, \Prodi_{t \le \tau} dQ_{ t} \, dG_t^*,
  \label{eq:target:parameter:again}
\end{align}
can be represented as a functional of \(P\in \mathcal{M}\) only
through \(\bm{Z}\)
\begin{align*}
  {\bm{Z}} =  {\bm{Z}}(P) :=   \big( {Z}^{G^*}_{t}, {Z}^{G^*}_{t,L(t)},
  {\Lambda}^{\ell}(t), {\Lambda}^{a}(t), {\Lambda}^{d}(t)
  \big)_{t\in [0,\tmax]} . 
\end{align*}
The basic idea of the proof is to use Fubini's theorem repeatedly to
rearrange the order of integration and thereby rewrite the target
parameter \eqref{eq:target:parameter:again} in terms of iterated
integrals.

\begin{proof} (Lemma 1).  Fix \(P\in \mathcal{M}\) such that
  \begin{align*}
    dP= \Prodi_{t \le \tau} dQ_t dG_t, 
  \end{align*}
  with \(dQ_t\) and \(dG_t\) as defined by Equations
  \eqref{eq:G:again}--\eqref{eq:Q:again} above.  For the sequence of
  time-points,
  \(0={t}_0=t_0 <{t}_1 < \cdots <{t}_R \le t_{R+1}= \tau\), we can
  evaluate the distribution of \(Y\) to write the target parameter as:
\begin{align*}
  \Psi^{G^*} (P)
  &  = 
    \,\,\underset{0=t_0 <{t}_1 < \cdots <{t}_R \le \tau}{\int} Y \,
    \prod_{r=0}^{R} \Prodi_{t \in (t_{r-1},t_r]}
    dQ_{ t} \, dG_t^* \Prodi_{t \in (t_{R},\tau]}
    dQ_{ t} \, dG_t^* .
\end{align*}
This can also be written as:
\begin{align*}
  &\Psi^{G^*} (P)
    = 
    \!\!\!\! \underset{0=t_0 <{t}_1 < \cdots <{t}_R \le \tau}{\int} \underbrace{\bigg( \int Y  \Prodi_{t \in (t_{R},\tau]}
    dG_t^* \bigg)}_{
    \mathclap{
    E_{P^{G^*}}[ Y \mid L(\tau), N^{\ell}(\tau) , N^a(\tau), N^d(\tau),\F_{t_R}]
    = Y
    }}\,  \Prodi_{t \in (t_{R},\tau]}
    dQ_{ t} \, \prod_{r=0}^{R}  \Prodi_{t \in (t_{r-1},t_r]}
    dQ_{ t} \, dG_t^*\notag \\
  & \,\, = 
    \!\!\!\!\!\! \!\! \underset{0=t_0 <{t}_1 < \cdots <{t}_R }{\int} 
    \bigg( \underbrace{ \int Y
    \, \Prodi_{t \in (t_{R},\tau]}
    \big(d\Lambda^{d} (t )\big)^{N^d(dt)}
    \big( 1 - d\Lambda^{d} (t ) \big)^{1-N^d(dt)} }_{
    \EE_{P}[ Y \mid \F_{t_R}]
    }\bigg)
      \prod_{r=0}^{R} \Prodi_{t \in (t_{r-1},t_r]} \!\!\!\!
    dQ_{ t} \, dG_t^*\\
  & \,\, = 
    \!\!\!\!\!\!\!\! \underset{0=t_0 <{t}_1 < \cdots <{t}_R}{\int} 
    \bigg( \underbrace{\int  \EE_{P}[ Y \mid \F_{t_R}] \Prodi_{t \in (t_{R-1},t_R]} dG^*_t}_{
    \mathclap{
    E_{P^{G^*}}[ Y \mid L(t_R), N^{\ell}(t_R) , N^a(t_R), N^d(t_R),\F_{t_{R-1}}]
    = Z_{t_R}^{G^*}
    }
    }\bigg)
    \Prodi_{t \in (t_{R-1},t_R]} \!\!\! dQ_t  \, \prod_{r=0}^{R-1}  \Prodi_{t \in (t_{r-1},t_r]} \!\!\!\!
    dQ_{ t} \, dG_t^* \\
  & \,\, = \!\!\!\!
    \!\!\!\! \underset{0=t_0 <{t}_1 < \cdots <{t}_R}{\int} 
    Z_{t_R}^{G^*}
    \Prodi_{t \in (t_{R-1},t_R]} dQ_t  \, \prod_{r=0}^{R-1} \Prodi_{t \in (t_{r-1},t_r]} \!\!\!\!
    dQ_{ t} \, dG_t^* . 
\end{align*}
Next, recall that
\begin{align*}
  &\Prodi_{t \in (t_{R-1},t_R]}
    dQ_{ t} =\Prodi_{t \in (t_{R-1},t_R]} \big( d\Lambda^{a} (t )
    \big)^{N^{a}(dt)} \big( 1-d\Lambda^{a} (t ) \big)^{1-N^{a}(dt)}\\[-0.3em]
  & \qquad\qquad\qquad\qquad\qquad \big( d\Lambda^{\ell} (t ) \,
    \mu_{t} (L(t)) d\nu_L(L(t))\big)^{N^{\ell}(dt)} \big(
    1-d\Lambda^{\ell} (t ) \big)^{1-N^{\ell}(dt)}\\
  & \qquad\qquad\qquad\qquad\qquad \big(d\Lambda^{d} (t )\big)^{N^d(dt)}
    \big( 1 - d\Lambda^{d} (t ) \big)^{1-N^d(dt)},
\end{align*}
where we have suppressed the conditioning on \(\F_{t-}\) to simplify
the presentation. We use this to rewrite the following:
\begin{align*}
  &\int
  Z_{t_R}^{G^*}
  \, \Prodi_{t \in (t_{R-1},t_R]}
  dQ_{ t}
  =  \int
    Z_{t_R}^{G^*} \Prodi_{t \in (t_{R-1},t_R]} \big( 
    \mu_{t} (L(t))d\nu_L(L(t)) \big)^{N^{\ell}(dt)} \\
  &\qquad\qquad\qquad\qquad\qquad\qquad\qquad \big( d\Lambda^{\ell} (t ) \big)^{N^{\ell}(dt)} \big(
    1-d\Lambda^{\ell} (t ) \big)^{1-N^{\ell}(dt)}\\
  &\qquad\qquad\qquad\qquad\qquad\qquad\qquad    \big( d\Lambda^{a} (t )
    \big)^{N^{a}(dt)} \big( 1-d\Lambda^{a} (t ) \big)^{1-N^{a}(dt)}\\
  & \qquad\qquad\qquad\qquad\qquad\qquad\qquad \big(d\Lambda^{d} (t )\big)^{N^d(dt)}
    \big( 1 - d\Lambda^{d} (t ) \big)^{1-N^d(dt)} \\
  &\qquad\qquad=  \int \underbrace{\bigg( \int
    Z_{t_R}^{G^*} \Prodi_{t \in (t_{R-1},t_R]} \big( 
    \mu_{t} (L(t)) d\nu_L(L(t)) \big)^{N^{\ell}(dt)} \bigg)}_{=
    Z_{t_R, L(t_R)}} \\
  &\qquad\qquad\quad\qquad\qquad\quad \!\! \Prodi_{t \in (t_{R-1},t_R]}
    \big( d\Lambda^{\ell} (t ) \big)^{N^{\ell}(dt)} \big(
    1-d\Lambda^{\ell} (t ) \big)^{1-N^{\ell}(dt)}
  \\[-0.99em]
  &\qquad\qquad\qquad \quad\qquad\qquad\qquad\quad \big( d\Lambda^{a} (t )
    \big)^{N^{a}(dt)} \big( 1-d\Lambda^{a} (t ) \big)^{1-N^{a}(dt)}\\
  &\qquad\qquad\qquad \quad\qquad\qquad\qquad\quad \big(d\Lambda^{d} (t )\big)^{N^d(dt)}
    \big( 1 - d\Lambda^{d} (t ) \big)^{1-N^d(dt)} \\
  &\qquad\qquad=  \int Z^{G^*}_{t_R, L(t_R)}  \Prodi_{t \in (t_{R-1},t_R]}  \big( d\Lambda^{\ell} (t ) \big)^{N^{\ell}(dt)} \big(
    1-d\Lambda^{\ell} (t ) \big)^{1-N^{\ell}(dt)} \\[-0.39em]
  &\qquad\qquad \quad\qquad\qquad\qquad\qquad \big( d\Lambda^{a} (t )
    \big)^{N^{a}(dt)} \big( 1-d\Lambda^{a} (t ) \big)^{1-N^{a}(dt)}\\
  &\qquad\qquad \quad\qquad\qquad\qquad\qquad \big(d\Lambda^{d} (t )\big)^{N^d(dt)}
    \big( 1 - d\Lambda^{d} (t ) \big)^{1-N^d(dt)}.
\end{align*}
Combined with
\begin{align*}
  & \int \bigg( \int
    Z_{t_R}^{G^*}
    \, \Prodi_{t \in (t_{R-1},t_R]}
    dQ_{ t} \bigg) 
    \Prodi_{t \in (t_{R-2},t_{R-1}]} 
    dG^*_{ t}  \\
  & \qquad = E_{P^{G^*}}[ Y \mid L(t_{R-1}), N^{\ell}(t_{R-1}) , N^a(t_{R-1}), N^d(t_{R-1}),
    \F_{t_{R-2}}]
    = Z_{t_{R-1}}^{G^*}
\end{align*}
we have that:
\begin{align*}
  &  Z_{t_{R-1}}^{G^*} = \int \bigg( \int Z^{G^*}_{t_R, L(t_R)}  \Prodi_{t \in (t_{R-1},t_R]}
    \big( d\Lambda^{\ell} (t ) \big)^{N^{\ell}(dt)} \big(
    1-d\Lambda^{\ell} (t ) \big)^{1-N^{\ell}(dt)}
    \\[-0.99em]
  & \qquad\qquad\qquad\qquad\qquad\quad\qquad\qquad
    \big( d\Lambda^{a} (t )
    \big)^{N^{a}(dt)} \big( 1-d\Lambda^{a} (t ) \big)^{1-N^{a}(dt)}\\
  & \qquad\qquad\qquad\qquad \big(d\Lambda^{d} (t )\big)^{N^d(dt)}
    \big( 1 - d\Lambda^{d} (t ) \big)^{1-N^d(dt)} \bigg)  \Prodi_{t \in (t_{R-2},t_{R-1}]} 
    dG^*_{ t}  .
\end{align*}
Hence, \( Z_{t_R, L(t_R)}, \Lambda^{\ell}, \Lambda^a, \Lambda^{d}\)
together with the specified intervention \(G^*\) completely
characterizes \( Z_{t_{R-1}}^{G^*}\). Further, the target parameter is
now reduced to:
   \begin{align*}
     & \Psi^{G^*} (P) =
       \,\,\underset{0  <{t}_1 < \cdots <{t}_{R-1}}{\int} 
       Z_{t_{R-1}}^{G^*}
       \, \Prodi_{t \in (t_{R-2},t_{R-1}]}
       dQ_{ t} \, \prod_{r=0}^{R-2} \Prodi_{t \in (t_{r-1},t_r]}
       dQ_{ t} \, dG_t^* 
\end{align*}
By backwards induction, applying the same arguments as above, we see
that
\begin{align*}
  &  Z_{t_r}^{G^*} = \int \bigg( \int Z^{G^*}_{t_{r+1}, L(t_{r+1})}  \Prodi_{t \in (t_{r},t_{r+1}]}
    \big( d\Lambda^{\ell} (t ) \big)^{N^{\ell}(dt)} \big(
    1-d\Lambda^{\ell} (t ) \big)^{1-N^{\ell}(dt)}    \\[-0.9em]
  & \qquad\quad\qquad\qquad\qquad\qquad\qquad\quad\qquad
    \big( d\Lambda^{a} (t )
    \big)^{N^{a}(dt)} \big( 1-d\Lambda^{a} (t ) \big)^{1-N^{a}(dt)}\\
  & \qquad\qquad\qquad\qquad\quad \big(d\Lambda^{d} (t )\big)^{N^d(dt)}
    \big( 1 - d\Lambda^{d} (t ) \big)^{1-N^d(dt)} \bigg) \Prodi_{t \in (t_{r-1},t_r]} 
    dG^*_{ t}  ,
\end{align*}
for any \(r = R, \ldots, 0\), and further that
\begin{align*}
     & \Psi^{G^*} (P) =
       \,\,\underset{0=t_0 <{t}_1 < \cdots <{t}_{r}}{\int} 
       Z_{t_r}^{G^*}
       \, \Prodi_{t \in (t_{r-1},t_r]}
       dQ_{ t} \prod_{l=0}^{r-1} \Prodi_{t \in (t_{l-1},t_l]}
       dQ_{ t} \, dG_t^* .
\end{align*}
We finally note that \(Z_{t_0}^{G^*}= \EE_{P^{G^*}} [Y \mid L_0]\),
and thus:
\begin{align*}
  \Psi^{G^*}(P) = \int  Z_{t_0}^{G^*} \mu_{L_0} (L_0) d\nu_{L_0} (L_0).
\end{align*}
In total, this shows that the target parameter defined in
\eqref{eq:target:parameter:again} can be characterized entirely by the
time-sequence of conditional expectations, \(Z_t^{G^*}\),
\(Z_{t,L(t)}^{G^*}\), and the conditional intensities
\(\Lambda^{\ell}, \Lambda^a, \Lambda^d\). This was the statement of
Lemma \ref{lemma:representation:Psi:Z}.
\end{proof}

\newpage

\section{Highly adaptive lasso (HAL)}
\label{app:smoothness:hal}

This appendix is concerned with highly adaptive lasso (HAL) estimation
of the nuisance parameters of our estimation problem.  The appendix is
structured as follows.  Section \ref{sec:initial:hal} gives the
general integral representation of \cadlag functions with finite
sectional variation norm. Section \ref{app:parametrization:nuisance}
outlines parametrizations of the nuisance parameters and states the
assumptions required to prove the HAL convergence results. Section
\ref{app:sec:hal:estimator} defines the HAL estimator. Section
\ref{sec:initial:hal:theory} presents the theoretical results for
HAL. Section \ref{app:hal:proofs} provides the proofs. We need a fair
amount of extra notation that will be introduced when needed
throughout the sections of this appendix.

\subsection{Càdlàg functions with finite sectional variation norm}
\label{sec:initial:hal}

Let \(\banachk\) denote the Banach space of \(k\)-variate \cadlag
functions on \([0,\eta]\subset \R^k\), \(\eta \in \R^k_+\). We define
the sectional variation norm of \(f\in\banachk\) as the sum of the
variation norms of the sections of \(f\) \citep{gill1995inefficient}:
\begin{align}
  \Vert f \Vert_{\vv}= \vert f(0) \vert + \sum_{\mathcal{S}\subset \lbrace 1, \ldots, k\rbrace}
  \int_{(0(\mathcal{S}), \eta(\mathcal{S})]} \vert f (dx(\mathcal{S}), 0(\mathcal{S}^c)) \vert.
  \label{eq:def:variation}
\end{align}
Here, \(\sum_{\mathcal{S} \subset \lbrace 1, \ldots, k\rbrace}\) is
the sum over all subsets of \(\lbrace 1, \ldots, k\rbrace\),
\({x}(\mathcal{S})= (x_j)_{ j \in \mathcal{S}} \) corresponds to the
\(\mathcal{S}\)-specific coordinates of \(x\),
\(x({\mathcal{S}^c})= (x_j)_{ j \not\in \mathcal{S}} \) are the
coordinates in the complement of the index set \(\mathcal{S}\), and
\({x}\rightarrow f(x(\mathcal{S}), 0(\mathcal{S}^c))\) is the
\(\mathcal{S}\)-specific section of \(f\) that sets the coordinates in
the complement of \(\mathcal{S}\) equal to zero. Furthermore, let
\begin{align*}
  \FF_{\vv,\mathscr{M}} = \big\lbrace f \in \banachk \, : \, \Vert f \Vert_{\vv} \le \mathscr{M}
  \big\rbrace, 
\end{align*}
denote the subset of \cadlag functions with sectional variation norm
bounded by a constant \(\mathscr{M}<\infty\).  Any
\(f\in \FF_{\vv,\mathscr{M}} \) admits an integral representation in
terms of the measures generated by its \(\mathcal{S}\)-specific
sections \citep{gill1995inefficient} which is used to define the HAL
estimator.

\subsubsection{Representation of \cadlag functions with finite variation
  norm}
\label{sec:hal:mle}

Any \(f\in \FF_{\vv,\mathscr{M}}\) admits an integral representation
as follows \citep{gill1995inefficient}:
\begin{align*}
  f(x) = f(0) +
\sum_{\mathcal{S}\subset \lbrace 1, \ldots, k\rbrace}
  \int_{(0(\mathcal{S}), x(\mathcal{S})]}
  f(du(\mathcal{S}), 0(\mathcal{S}^c)).
\end{align*}
Particularly, consider a function \(f\in \FF_{\vv,\mathscr{M}}\)
defined on a finite set of support points
\(\lbrace s_j \rbrace_{j \in \mathcal{I}}\) indexed by
\(\mathcal{I}\).  The above integral representation of
\(f \in \FF_{\vv,\mathscr{M}}\) can be written in terms of a finite
linear combination of indicator basis functions as follows:
\begin{align}
  f(x)=  f(0) +
  \sum_{\mathcal{S}\subset \lbrace 1, \ldots, k\rbrace}
  \sum_{j \in \mathcal{I}}
  \1 \lbrace s_j ( \mathcal{S}) \le x(\mathcal{S}) \rbrace 
  f(ds_j(\mathcal{S}), 0(\mathcal{S}^c)),
  \label{eq:rep:finite:1}
\end{align}
where \( f(ds_j(\mathcal{S}), 0(\mathcal{S}^c))\) is the point-mass
that the \(\mathcal{S}\)th section function assigns to the \(j\)th
support point \(s_j(\mathcal{S})\).

Now, define the indicator basis functions
\(\phi_{\mathcal{S},j} ( x ) = \1 \lbrace s_j(\mathcal{S}) \le
x(\mathcal{S})\rbrace\) and the corresponding coefficients
\(\beta_{\mathcal{S},j} = f(ds_j(\mathcal{S}), 0(\mathcal{S}^c))
\). Then we can write \eqref{eq:rep:finite:1} as:
\begin{align}
  f(x)=  \beta_0 +
  \sum_{\mathcal{S}\subset \lbrace 1, \ldots, k\rbrace}
  \sum_{j \in \mathcal{I}}
  \phi_{\mathcal{S},j} ( x )
  \beta_{\mathcal{S},j} . 
  \label{eq:rep:finite:2}
\end{align}
We further note that the sectional variation norm (c.f.,
\eqref{eq:def:variation}) of \(f\) is the sum of the absolute values
of its coefficients:
\begin{align}
  \Vert f \Vert_{\vv} = \Vert \beta \Vert_{1} = \vert \beta_{0} \vert +
   \sum_{\mathcal{S}\subset \lbrace 1, \ldots, k\rbrace}
  \sum_{j \in \mathcal{I}(\mathcal{S})}
  \vert \beta_{\mathcal{S},j} \vert.
  \label{eq:var:norm:L1:norm}
\end{align}

Generally, the HAL estimator is defined as the minimizer of the
empirical risk over all linear combinations of indicator basis
functions under the constraint that the sectional variation norm,
i.e., the absolute value of the coefficients, is smaller than or equal
to a finite constant. We get back to this in Section
\ref{app:sec:hal:estimator}.

\subsection{Parametrization of nuisance parameters}
\label{app:parametrization:nuisance}

In this section we detail parametrizations of our nuisance parameters
to apply the tools from the previous section to define HAL estimation.
To formulate the assumptions we need, recall first that, for any
subject, \(\bar{O}(t)\) only changes at the observed event times
\(T_{1}, \ldots, T_{K({t})}\).  It is essential to our analysis that
we at any point in time \(t\) can represent the observed data
\(\bar{O}(t)\) in terms of a finite-dimensional vector.  For this
purpose, we define:
\begin{align*}
   \bar{O}_k  = 
  \big\lbrace \big(L_0, s, dN^a(s),A(s),dN^{\ell}(s)
    ,L(s),dN^d(s),dN^c(s)\big):s\in \{T_{j}\}_{j=0}^{k}
    \big\rbrace , 
\end{align*}
where \(\{T_{k}\}_{k=1}^{K}\) denotes the ordered set of unique event
times of one subject.  Then \(\bar{O}_k \in \R^{kd'+d_0}\) where
\(d'\in\N\) denotes the dimension of an observation at a monitoring
time \(T_k\), i.e.,
\begin{align*}
  \big( T_k, dN^a(T_k),A(T_k),dN^{\ell}(T_k)
  ,L(T_k),dN^d(T_k),dN^c(T_k)\big) \in \R^{d'},
\end{align*}
and \(d_0\in\N\) is the dimension of \(L_0\).  Following the notation
from Section \ref{sec:initial:hal}, where
\({x}(\mathcal{S})= (x_j)_{ j \in \mathcal{S}} \) denotes the
\(\mathcal{S}\)-specific coordinates of \(x\), we use
\(\bar{O}_{k} (\mathcal{S})\) to denote the \(\mathcal{S}\)-specific
coordinates of \(\bar{O}_{k}\) for an index set
\(\mathcal{S}\subset \lbrace 1, \ldots,{kd'+d_0}\rbrace\).

Recall that we need initial estimators for the following time-sequence
of conditional densities, conditional expectations, and conditional
intensities:
\begin{equation*}
  \begin{split}
    G &= \big(
    \pi_{t}, {\Lambda}^{c}(t)  \big)_{t\in [0,\tau]},  \qquad \text{and,}\\
    {\bm{Z}} &= \big( {Z}^{G^*}_{t}, {Z}^{G^*}_{t,L(t)},
    {\Lambda}^{\ell}(t), {\Lambda}^{a}(t), {\Lambda}^{d}(t)
    \big)_{t\in [0,\tmax]} .
  \end{split}
\end{equation*}
For estimation of \({Z}^{G^*}_{t}\) and \({Z}^{G^*}_{t,L(t)}\), we
will proceed by estimating the conditional density \(\mu_t\)
directly. Then we construct substitution estimators for
\({Z}^{G^*}_{t}\) and \({Z}^{G^*}_{t,L(t)}\) based on estimators for
\(\mu_t,\Lambda^{\ell}(t), \Lambda^{a}(t), \Lambda^{d}(t)\) by
evaluating the g-computation formula. Accordingly, we define
parametrizations and loss functions for the conditional densities
\(\pi_t\) and \(\mu_t\), and for the conditional intensities
\(\Lambda^c, \Lambda^\ell, \Lambda^a, \Lambda^d\).

\subsubsection{Parametrizations and loss functions}
\label{app:fk:loss}

We parametrize the conditional density \(\mu_t\) of \(L(t)\) in terms
of a function \(f^L(t, \bar{O}(t)) \), the conditional distribution
\(\pi_t\) of \(A(t)\) in terms of \(f^A(t, \bar{O}(t)) \), and the
intensities \(\Lambda^x(t)\) in terms of \(f^x(t, \bar{O}(t)) \),
\(x=\ell,a,c,d\).

To keep the notation simple, we assume that \(A(t)\) and \(L(t)\) are
binary such that \(\pi_t\) and \(\mu_t\) can be parametrized as
follows:
\begin{align*}
  \pi_t ( 1 \, \vert \, \bar{O}(t) )
  &= 
    \expit ( f^A(t,\bar{O}(t)) ),                        ,
\end{align*}
and
\begin{align*}
  \mu_t( 1 \, \vert \,\bar{O}(t) )
  &= 
    \expit ( f^L(t,\bar{O}(t)) ). 
\end{align*}
For the intensities, \(\Lambda^x\), \(x=\ell,a,d,c\), we consider the
continuous case.  We parametrize the intensity process \( \lambda^x\)
for which
\(\Lambda^x(t \, \vert \, \F_{t-}) = \int_0^t \lambda^x (s \,\vert\,
\F_{s-}) ds \) as follows:
\begin{align*}
  \lambda^x(t \, \vert \, \bar{O}(t) )  = 
  \exp ( f^x(t,\bar{O}(t)) ) , \qquad x= \ell, a, d, c. 
\end{align*}

Now, all required nuisance parameters are represented by a function
\(f^x\), \(x \in \lbrace L, A, c, a, \ell, d \rbrace\).  We further
parametrize each \(f^x\) in terms of a sequence of real-valued
functions \((t, \bar{O}_k )\mapsto f^x_k ( t, \bar{O}_k)\) with a
fixed\hyp{}dimensional support, such that:
\begin{align}
  f^x(t, \bar{O}(t)) = \sum_{k=0}^{K}
  \1 \lbrace K(t) = k \rbrace \, f_k^x(t, \bar{O}_{k}), \quad x
  \in \lbrace L, A, c, a, \ell, d \rbrace.
  \label{eq:rep:f:Kt:0}
\end{align}

The following assumption is central: It will allow us to define HAL
estimators and it provides the basis for establishing condition 2 of
Theorem \ref{thm:eff:estimator}.  Assumption
\ref{ass:f:all:k:cadlag:finite:variation} is formulated for \(f_k^x\)
defined by Display \eqref{eq:rep:f:Kt:0}, but translates to \( f^x\)
because the sum in \eqref{eq:rep:f:Kt:0} is
finite.

\begin{assumption}[Càdlàg and finite variation]
  We assume that \(f^x_k\in \normalfont \FF_{\vv, \mathscr{M}^x}\) for all \(k\le K\) and
  all \(x=L,A,c,a,\ell,d\), that is, \(f^x_k\) is \cadlag and has
  sectional variation norm bounded by a constant \(\mathscr{M}^x <\infty\).
  \label{ass:f:all:k:cadlag:finite:variation}
\end{assumption}

For \(x=L,A,c,a,\ell,d\), let
\((O, f^x) \mapsto \mathscr{L}_{x}(f^x)(O)\) be the log-likelihood
loss. For example, for \(x=L\),
\((O, f^L) \mapsto \mathscr{L}_{L}(f^L)(O)\) is defined as:
\begin{align*}
  \mathscr{L}_{L} (f^L) (O) &=
  \sum_{k=0}^{K}
  \Delta N^\ell(t_k)  \big( L(k)
  \log \big( 1+\exp(- f_{k}^L(  T_k, \bar{O}_{k} ) )\big)  \\
 &\qquad\qquad \,+
  \big(1-L(k)\big)
   \log \big( 1+\exp( f_{k}^L( T_k, \bar{O}_{k} ) )\big)\big),  
\end{align*}
whereas for \(x=\ell\), in the continuous case considered in Section
\ref{app:fk:loss}, it is:
\begin{align*}
  \mathscr{L}_{\ell} (f^\ell) (O) &=
                                    \sum_{k=1}^{K}
                                    \Delta N^\ell(t_{k}) f^\ell_{k}(t_{k}, \bar{O}_{k}) \\
                                  &\qquad\qquad - \,
                                    (t_{k}-t_{k-1}) \exp(f^\ell_{k-1}(t_{k-1}, \bar{O}_{k-1})).  
\end{align*}
We denote by \(f^x_0 = \text{argmin}_{f^x} P_0 \mathscr{L}_x (f^x) \),
the minimizer of the risk under \(P_0\).  Since the log-likelihood is
strictly proper \citep{gneiting2007strictly}, the data-generating
\(f_0^x\) attains this minimum.  We define the sum loss function
\((O,f^Q) \mapsto \mathscr{L}_Q (f^Q) (O)\) for the \(Q\)-factor of the
likelihood, i.e.,
\begin{align*}
\mathscr{L}_Q (f^Q) = \mathscr{L}_L (f^L) + \mathscr{L}_\ell
  (f^\ell) + \mathscr{L}_a (f^a) + \mathscr{L}_d (f^d),
\end{align*}
where \(f^{Q}=(f^{L},f^{\ell},f^a,f^d)\), and, also, the sum loss
function \((O, f^G)\mapsto \mathscr{L}_G (f^G) (O)\) for the
\(G\)-factor of the likelihood, i.e.,
\begin{align*}
  \mathscr{L}_G (f^G) = \mathscr{L}_A (f^A) + \mathscr{L}_c (f^c),
\end{align*}
where \(f^{G}=(f^A,f^c)\). Note that minimizing the sum of a set of
loss functions is the same as minimizing them separately.

Assumption \ref{ass:uniformly:bounded} below guarantees the oracle
properties of the cross-validation selector
\citep{van2003unified,van2006oracle}. We note that
\eqref{eq:standard:property} holds for most common loss
functions.

\begin{assumption}[Bounded loss functions]
  We assume that the loss functions are uniformly bounded in the sense
  that \(\sup_{f^x,O} \mathscr{L}_x (f^x)(O) < \infty \) a.s. where
  the supremum is over the support of \(P_0\), i.e., all possible
  realizations of \(O\) for a single subject, and all \(f^x\) such
  that \(f^x_k \in \FF_{\vv,\mathscr{M}^x}\). In addition, we assume
  that,
  \begin{align}
    \begin{split}
      \underset{f^G }{\sup}\,\, \frac{ \Vert \mathscr{L}_{G}({f}^G) -
        \mathscr{L}_{G}(f^G_0)
        \Vert_{P_0}^2 }{ \mathscr{D}_{G} ( {f}^G, f^G_0) } <\infty, \\
      \underset{f^Q }{\sup}\,\, \frac{ \Vert \mathscr{L}_{Q}({f}^Q) -
        \mathscr{L}_{Q}(f^Q_0) \Vert_{P_0}^2 }{ \mathscr{D}_{Q} (
        {f}^Q, f^Q_0) } <\infty,
\end{split}
  \label{eq:standard:property}
  \end{align}
  where
  \(\mathscr{D} ( {f}, f_0): = P_0 \mathscr{L}({f}) - P_0 \mathscr{L}(f_0) \)
  denotes the loss-based dissimilarity measure.
  \label{ass:uniformly:bounded}
\end{assumption}

As detailed in Appendix \ref{sec:hal:mle}, the representation of a
càdlàg function with finite sectional variation norm becomes a finite sum over
indicator basis functions when the function is defined on a discrete
support. The HAL estimator for \(f^x\) is defined as the minimizer of
the empirical risk over all measures with a particular
support defined by the actual observations \(\{O_i\}_{i=1}^n\).

\subsection{The HAL estimator}
\label{app:sec:hal:estimator}

We apply the representation of Section \ref{sec:hal:mle} to
\((t, \bar{O}_k )\mapsto f^x_k ( t, \bar{O}_k)\), which, again,
translates into a representation for \(f^x\) by
\eqref{eq:rep:f:Kt:0}. Then we can define the HAL estimator for
\(f^x\) as the minimizer of the empirical risk over all linear
combinations of indicator basis function for a specific set of support
points under the constraint that the sectional variation norm is
bounded by the constant \(\mathscr{M}^x\). By selecting the particular
support defined by the \(n\) observations \(\{O_i\}_{i=1}^n\), the
minimizer of the finite-dimensional minimization problem equals the
minimizer of the empirical risk over all measures \(f^x \) with
sectional variation norm smaller than \(\mathscr{M}^x\).

Specifically, let \(\mathcal{I}_{k}\) be the index set for the unique
observed values of \(\bar{O}_{k}\). Recall that
\(\bar{O}_{k}(\mathcal{S})\) denotes the \(\mathcal{S}\)-specific
coordinates \(\bar{O}_k\), and that \({t}_0 < \cdots <{t}_{{K}_n}\) is
the ordered sequence of unique times of changes. Consider functions
\({f}^x_{k,\beta}\) such that
\(\bar{O}_k \mapsto {f}^x_{k,\beta} (t, \bar{O}_k)\) has support
\(\lbrace s_{k,j} \rbrace_{j\in \mathcal{I}_{k}} \) and
\(t \mapsto {f}^x_{k,\beta} (t, \bar{O}_k)\) has support
\(\lbrace {t}_0, \ldots, {t}_{{K}_n}\rbrace\). Separating the
representation of \({f}^x_{k,\beta}\) along the lines of Section
\ref{sec:hal:mle} into terms involving the time axis
and terms involving \(\bar{O}_k\) gives:
\begin{align}
  &  f_{k, \beta}^x(t, \bar{O}_k)
    =   \sum_{r=0}^{{K}_n}  \1 \lbrace {t}_r \le t\rbrace
    \bigg( 
    \beta^x_{k,r,0} + \!\!\!\!\!\!\! \sum_{\mathcal{S} \subset \lbrace 1, \ldots, kd'+d_0\rbrace}
    \sum_{j\in \mathcal{I}_{k}} \phi^x_{k,\mathcal{S},j} ( \bar{O}_k)
    \beta^x_{k,r,\mathcal{S},j}\bigg).
      \label{eq:hal:beta:12}
\end{align}
Here, \(\beta_{k,r,\mathcal{S},j}\) is the measure that the
\(\mathcal{S}\)th section of \(f_{k,\beta}^x\) assigns to the \(j\)th
support point for \({t}_r \le t\) and
\(\phi_{k,\mathcal{S},j} (\bar{O}_k) = \1 \lbrace s_{k,j}(\mathcal{S})
\le \bar{O}_{k} (\mathcal{S}) \rbrace \) is the indicator that the
support point \(s_{k,j}(\mathcal{S})\) is smaller than or equal to
\(\bar{O}_{k}(\mathcal{S}) \).

Corresponding to the representation \(f^x_{k,\beta}\) for \(f^x_{k}\)
in \eqref{eq:hal:beta:12} we have an equivalent representation
\(f_{\beta}^x\) for \(f^x\).  Now we can define the HAL estimator as
follows:
\begin{align}
  \hat{f}_n^x= \underset{f_{\beta}^x \, : \,  \Vert f_{\beta}^x \Vert_{\vv} \le \mathscr{M}^x}{\argmin} \mathbb{P}_n
  \mathscr{L}_x (f_{\beta}^x).
  \label{eq:min:11}
\end{align}
Particularly, the sectional variation norm of the finite sum
representation \eqref{eq:hal:beta:12} equals the sum of the absolute
values of the coefficients by \eqref{eq:var:norm:L1:norm} so that we
can replace the constraint
\(\Vert f_\beta^x \Vert_{\vv} \le \mathscr{M}^x\) in \eqref{eq:min:11}
by \(\Vert f_\beta^x \Vert_{1} \le \mathscr{M}^x\). Hence, we can
compute the HAL estimator for each \(f^x\), \(x=L,A,\ell,a,d,c\), by:
\begin{align}
  \hat{f}_n^x= \underset{f_{\beta}^x \, : \, \Vert \beta \Vert_{1} \le \mathscr{M}^x}{\argmin}
  \mathbb{P}_n \mathscr{L}_x (f_{\beta}^x),
  \label{eq:hal:L1:min}
\end{align}
corresponding to \(L_1\)-penalized (Lasso) regression
\citep{tibshirani1996regression} with the indicator functions
\(\phi^x_{k,\mathcal{S},j} (\bar{O}_k)\) as covariates and
\(\beta^x_{k,r,\mathcal{S},j}\) as corresponding coefficients. For the
log-likelihood loss and the squared error loss functions, standard
software can be used to find the estimators \eqref{eq:hal:L1:min},
including selection of \(\mathscr{M}^x\) by cross-validation.  Table
\ref{table:hal:R} in Appendix \ref{app:overview:diagrams} gives a
brief overview.

\subsection{Theoretical results for HAL}
\label{sec:initial:hal:theory}

We collect here the theoretical results for HAL estimation. These
results rely on Assumptions \ref{ass:f:all:k:cadlag:finite:variation}
and \ref{ass:uniformly:bounded} of Section \ref{app:fk:loss}.

\begin{lemma}
  The canonical gradient can be written as
  \(D^*(P) = D^*(f^x \, : \; x=L,A,c, a,\ell,d)\). If
  \(\prodi_{s <t} d{G}_{t}^* / d{G}_{t} \) is uniformly bounded for all
  \(t\), then:
  \begin{align*}
    \big\lbrace D^*(f^x  :  x=L, A,c, a,\ell,d)
    \big\rbrace
  \end{align*}
  is a Donsker class.
  \label{thm:eic:donsker:all}
\end{lemma}

\begin{proof} See Section \ref{sec:representation:eic}.  \end{proof}

\begin{thm} 
  For the loss functions \(O \mapsto\mathscr{L}_x(f^x)(O)\), for
  \(x=Q,G\), we have, under Assumptions
  \ref{ass:f:all:k:cadlag:finite:variation} and
  \ref{ass:uniformly:bounded}, that
  \begin{align*}
    \mathscr{D}_{G} ( \hat{f}^G_n, f^G_0)
    & =   P_0 \mathscr{L}_{G}(\hat{f}^G_n) -
      P_0 \mathscr{L}_{G}(f^G_0) = o_P(n^{-1/2}) , \\
    \mathscr{D}_{Q} ( \hat{f}^Q_n, f^Q_0)
    & =   P_0 \mathscr{L}_{Q}(\hat{f}^Q_n) -
      P_0 \mathscr{L}_{Q}(f^Q_0) = o_P(n^{-1/2}) ,
  \end{align*}
  and further that
  \( \Vert \hat{f}^G_n - {f}^G_0 \Vert_{P_0} = o_P(n^{-1/4})\) and
  \( \Vert \hat{f}^Q_n - {f}^Q_0 \Vert_{P_0} = o_P(n^{-1/4})\).
  \label{thm:hal:proof}
\end{thm}

\begin{proof} See Section \ref{sec:hal:proof:Z}.  On a final note, we
  remark that the \(n^{-1/4}\) rate can be improved to
  \(n^{-1/3}\log(n)^{d/2}\) as established by
  \cite{2019arXiv190709244B}.  \end{proof}

In summary, Theorem \ref{thm:hal:proof} combined with Remark
\ref{remark:R2} implies that using HAL for initial estimation fulfills
condition 1 of Theorem \ref{thm:eff:estimator}. Moreover, under our
nonparametric smoothness assumptions (Assumption
\ref{ass:f:all:k:cadlag:finite:variation}), the first part of
condition 2 of Theorem \ref{thm:eff:estimator} holds by Lemma
\ref{thm:eic:donsker:all}. The second part of condition 2 can be seen
to hold by Lemma \ref{thm:eic:donsker:all}, the functional delta
method, and Hadamard differentiability of the product integral
\citep{gill1990survey,van2000asymptotic}.

\subsection{Proofs of results}
\label{app:hal:proofs}

We here prove the HAL results. Particularly, Section
\ref{sec:representation:eic} establishes Donsker properties of the
canonical gradient (Lemma \ref{thm:eic:donsker:all}) and the loss
functions (Lemma \ref{thm:loss:donsker:all}), and Section
\ref{sec:hal:proof:Z} presents the final HAL convergence proof.

\subsubsection{Donsker class conditions}
\label{sec:representation:eic}

Lemma \ref{thm:eic:donsker:all} (see Section
\ref{sec:initial:hal:theory}) gives the Donsker properties of the
efficient influence function. The subsequent Lemma
\ref{thm:loss:donsker:all} gives Donsker properties of the loss
functions as needed for the HAL proof of Theorem \ref{thm:hal:proof}.

\begin{proof}\noindent (Lemma \ref{thm:eic:donsker:all}).\\
Consider the following representation of the efficient
influence curve,
\begin{align*}
  D^*(P)
  &= \int_0^{\tmax}\bigg( \Prodi_{s <t } \frac{dG^*_s}{dG_s} \bigg)  \,
    \Big( \int Y \Prodi_{s > t  } dQ_s \Prodi_{s  \ge t } dG^*_s -
    \int Y \Prodi_{ s\ge t } dQ_s  dG^*_s \Big)
    . 
\end{align*}
First, consider:
\begin{align*}
  \Prodi_{s<t} d{G}_{s} (O)=
  \prod_{k=1}^{K(t)} \big( d\Lambda^c (T_k) \big)^{\Delta N^c(T_k)}\big( \pi_{T_k} (A(T_k))\big)^{\Delta N^a(T_k)}
  \Prodi_{s \in (0,t)}\big( 1- d\Lambda^c (s) \big),
\end{align*}
and, in the same way,
\begin{align*}
  \Prodi_{s\ge t} d{G}_{s} (O)=
  \prod_{k=K(t)+1}^{K} \big( d\Lambda^c (T_k) \big)^{\Delta N^c(T_k)}\big( \pi_{T_k} (A(T_k))\big)^{\Delta N^a(T_k)}
 \Prodi_{s \in (t, \tau)}
 \big( 1- d\Lambda^c (s) \big),
\end{align*}
with corresponding versions of \(\prodi_{s<t} d{G}^*_{s}\) and
\(\prodi_{s\ge t} dG^*_s\). Note that we have suppressed the
conditioning on \(\F_{t-}\) to simplify the presentation. Similarly:
\begin{align*}
  \Prodi_{s<t} d{Q}_{s} (O)=
  \prod_{k=1}^{K(t)} \big( d\Lambda^\ell(T_k) \big)^{\Delta N^\ell(T_k)}\big( \mu_{T_k}
  (L(T_k))\big)^{\Delta N^\ell(T_k)}
  \Prodi_{s \in (0,t)} \big( 1- d\Lambda^\ell (s) \big) \\
  \prod_{k=1}^{K(t)}  \big( d\Lambda^a(T_k) \big)^{\Delta N^a(T_k)}
 \Prodi_{s \in (0,t)} \big( 1- d\Lambda^a (s) \big)\\
  \prod_{k=1}^{K(t)} \big( d\Lambda^d(T_k) \big)^{\Delta N^d(T_k)}
   \Prodi_{s \in (0,t)}\big( 1- d\Lambda^d (s) \big) ,
\end{align*}
and
\begin{align*}
  \Prodi_{s\ge t} d{Q}_{s} (O)=
  \prod_{k=K(t)+1}^{K} \big( d\Lambda^\ell(T_k) \big)^{\Delta N^\ell(T_k)}\big( \mu_{T_k}
  (L(T_k))\big)^{\Delta N^\ell(T_k)}
  \Prodi_{s \in (t, \tau)} \big( 1- d\Lambda^\ell (s) \big) \\
  \prod_{k=1}^{K(t)}  \big( d\Lambda^a(T_k) \big)^{\Delta N^a(T_k)}
  \Prodi_{s \in (t, \tau)} \big( 1- d\Lambda^a (s) \big)\\
  \prod_{k=1}^{K(t)} \big( d\Lambda^d(T_k) \big)^{\Delta N^d(T_k)}
  \Prodi_{s \in (t, \tau)} \big( 1- d\Lambda^d (s) \big) .
\end{align*}
We now plug in our particular parametrizations from Section
\ref{app:fk:loss}:
\begin{align*}
  \pi_{T_k} (A(T_k)) &=  \expit( f_k^A( T_k,  \bar{O}_{k-1} ))^{A(T_k)}
                       \big(1- \expit( f_k^A( T_k,  \bar{O}_{k-1}
                       )), \\
  \mu_{T_k}
  (L(T_k)) & = \expit( f_k^L( T_k, \bar{O}_{k-1} ))^{L(T_k)} \big(1- \expit( f_k^L( T_k, \bar{O}_{k-1} )),
\end{align*}
and
\begin{align*}
  d\Lambda^x (T_k)
  &= \exp \big( f_{k-1}^x(T_{k-1}, \bar{O}_{k-1})\big) (T_k - T_{k-1}), \\
  \Prodi_{s \in (0,  t)} \big( 1- d\Lambda^x (s) \big)
  &=
    \prod_{r=1}^{K_n}    \Prodi_{s \in (t_r, \min(t, t_{r+1}))} \big( 1- d\Lambda^x (s) \big)\\
  & \!\!\!\!\!\!\!\!\!\!\!\!\!\!\!\!\!\! =  \exp \bigg(
    \sum_{r=1}^{K_n} \1 \lbrace t_r < t\rbrace  \big(\min(t, t_{r+1}) - t_r\big) \exp( f_{K(t_r)}^x(t_r, \bar{O}_{K(t_r)}))
    \bigg), \\
  \Prodi_{s \in (t,  \tmax)} \big( 1- d\Lambda^x (s) \big)
  &=
    \prod_{r=1}^{K_n}    \Prodi_{s \in (\max(t, t_r), t_{r+1})} \big( 1- d\Lambda^x (s) \big)\\
  & \!\!\!\!\!\!\!\!\!\!\!\!\!\!\!\!\!\! = \exp \bigg(  \sum_{r=1}^{K_n}  \1 \lbrace t_{r+1} \ge t\rbrace 
    \big(t_{r+1} - \max(t, t_r)\big) \exp( f_{K(t_r)}^x(t_r, \bar{O}_{K(t_r)}))
    \bigg), 
\end{align*}
for \(x=c, a, \ell, d\).  This shows the first statement of Lemma
\ref{thm:eic:donsker:all}.  By assumption, \(f^x_k\) is uniformly
bounded. Furthermore, \(\prodi_{s<t} d{G}_{s}^* / d{G}_{s} \) is
uniformly bounded away from zero which preserves the Donsker property
of the ratio.  Since the set of \cadlag functions with finite
variation norm is a Donsker class and since the Donsker property is
preserved under Lipschitz transformations and also under products and
sums \citep{van1996weak}, we conclude that
\(\big\lbrace D^*(f^x : x=L, A,c, a,\ell,d)\big\rbrace\) is a Donsker
class.
\end{proof}

\begin{lemma}
  For a set of constants \(\mathscr{M}^x< \infty\),
  \(x=L, A,a,\ell,d,c\), we have that
  \( \big\lbrace \mathscr{L}_x(f^x) \big\rbrace\) is a Donsker class.
  \label{thm:loss:donsker:all}
\end{lemma}

\begin{proof} 
  The log-likelihood loss \(\mathscr{L}_A\) for \(f^A\) can be
  written:
\begin{align*}
  \mathscr{L}_{A} (f^A) (O) =
  \sum_{k=0}^{K}
   -  \Delta N^a(T_k)  \Big( A(T_k)
  \log \big( 1+\exp(- f_{k}^A( T_k, \bar{O}_{k} ) )\big)  \\
  + \,
  \big(1-A(T_k)\big)
  \log \big( 1+\exp( f_{k}^A( T_k, \bar{O}_{k} ) )\big)\Big), 
\end{align*}
and correspondingly for log-likelihood loss \(\mathscr{L}_L\) for
\(f^L\):
\begin{align*}
  \mathscr{L}_{L} (f^L) (O) =
  \sum_{k=0}^{K}
   -  \Delta N^\ell(T_k)  \Big( L(T_k)
  \log \big( 1+\exp(- f_{k}^L(  T_k, \bar{O}_{k} ) )\big)  \\
  + \,
  \big(1-L(T_k)\big)
  \log \big( 1+\exp( f_{k}^L( T_k, \bar{O}_{k} ) )\big)\Big). 
\end{align*}
For \(x=c, a, \ell, d\), we write the loss function as:
\begin{align*}
  \mathscr{L}_{x} (f^x) (O) &=
                              \sum_{k=1}^{K}
                              \Delta N^x(t_{k})  f^x_{k}(t_{k}, \bar{O}_{k}) \\
                            &\qquad\qquad - \,
                              \sum_{r=1}^{K_n} (t_{r}-t_{r-1}) \exp(f^x_{K(t_r)-1}(t_{r-1}, \bar{O}_{K(t_r)-1})).  
\end{align*}
By Assumption \ref{ass:f:all:k:cadlag:finite:variation}, \(f^x_k\) is
uniformly bounded. Further note that
\(\inf_{f^x_k} \,(1+\exp(f^x_k)) \) \( \ge 0\). Accordingly, \(f^x_k\)
only ranges over values on which \(z \mapsto \exp(z) \) and
\(z \mapsto \log \, z\) are Lipschitz continuous functions. Since the
set of \cadlag functions with finite variation norm is a Donsker class
and since the Donsker property is preserved under Lipschitz
transformations and also under products and sums \citep{van1996weak}
we conclude that \(\mathscr{L}_x(f^x)\) for all \(x=L,A,c, a,\ell,d\)
is a Donsker class. \end{proof}

\subsubsection{HAL proof (Theorem \ref{thm:hal:proof})}
\label{sec:hal:proof:Z}

We refer to the general HAL proof \citep{van2017generally}. What we
need to show is that:
\begin{align}
  \mathscr{D}_{x} ( \hat{f}^x_n, f^x_0)
  & =   P_0 \mathscr{L}_{x}(\hat{f}^x_n) -
    P_0 \mathscr{L}_{x}(f^x_0) = o_P(n^{-1/2}) , \label{eq:loss:diss:x}
\end{align}
for \(x=Q,G\).  By the general HAL proof, we have that
\(\mathscr{D}_{x} ( \hat{f}^x_n, f^x_0)\) is bounded by: 
\begin{align*}
  - ( \mathbb{P}_n - P_0 ) \big(\mathscr{L}_{x}(\hat{f}^x_n) - \mathscr{L}_{x}(f^x_0)\big).  
\end{align*}
Since \(\mathscr{L}_{x}(\hat{f}^x_n) - \mathscr{L}_{x}(f^x_0)\) falls
in a Donsker class (Lemma \ref{thm:loss:donsker:all}) it follows that
\( \mathscr{D}_{x} ( \hat{f}^x_n, f^x_0) = O_P(n^{-1/2})\) which,
again, implies that
\(P_0 (\mathscr{L}_{x}(\hat{f}^x_n) - \mathscr{L}_{x}(f^x_0) )^2 =
O_P(n^{-1/2}) \). The latter implication relies on
\eqref{eq:standard:property} stated in Assumption
\ref{ass:uniformly:bounded}, and holds for the squared error loss and
the log-likelihood loss as long as these are uniformly bounded (c.f.,
Assumption \ref{ass:uniformly:bounded}). It then follows by the
Donsker theorem that
\(- ( \mathbb{P}_n - P_0 ) (\mathscr{L}_{x}(\hat{f}^x_n) -
\mathscr{L}_{x}(f^x_0)) = o_P(n^{-1/2})\) which gives
\(\mathscr{D}_{x} ( \hat{f}^x_n, f^x_0) = o_P(n^{-1/2})\).

Finally, since the loss-based dissimilarity (i.e., the
Kullback-Leibler divergence) behaves as the squared
\(L_2 (P_0 )\)-norm \citep[see, for example,][Lemma
4]{van2017generally}, it follows that
\( \mathscr{D}_{x} ( \hat{f}^x_n, f^x_0)= o_P(n^{-1/2})\) implies
\(\Vert \hat{f}^x_n - {f}^x_0 \Vert_{P_0} = o_P(n^{-1/4})\).

\newpage

\section{Overview}
\label{app:overview:diagrams}

\subsection{Overview of notation}
\label{app:sub:overview:notation}

\quad

\begin{table}[!h]
\centering
\begin{tabular}{ l l r }
  \hline \\[-0.9em]
  \(L_0\in \R^{d_0}\) & baseline covariate vector & \\[0.0em]
  \(A(t)\in \mathcal{A}\) & treatment decision at time \(t\) & \\[0.0em]
  \(L(t) \in \R^d\) & covariate vector \(t\) & \\[0.0em]
  \(N^a(t)\) &  \multicolumn{2}{l}{counting process recording changes in treatment}  \\[0.0em]
  \(N^{\ell}(t)\) &  \multicolumn{2}{l}{counting process recording changes in covariates}  \\[0.0em]
  \(N^{c}(t)\) &  \multicolumn{2}{l}{counting process recording changes in censoring status}  \\[0.0em]
  \(N^{d}(t)\) &  \multicolumn{2}{l}{counting process recording changes in death status}  \\[0.0em]
  \(\Lambda^x\) &  \multicolumn{2}{l}{the
                  cumulative intensity
                  characterizing the compensator of \(N^x\), and }  \\[0.0em]
  \(M^x=N^x - \Lambda^x\) &  \multicolumn{2}{l}{the corresponding martingale, \(x=a,\ell,c,d\)}  \\[0.0em]
  \(\pi_t\) & \multicolumn{2}{l}{the conditional density of \(A(t)\); \(\pi_{0,t}(a \,|\, \F_{t-} ) =P(A(t)=a \,|\,
              \F_{t-})\)} \\[0.0em]
  \(\mu_t\) & \multicolumn{2}{l}{the conditional density of \(L(t)\);
              \(\mu_{0,t}(\ell \, \vert\, \F_{t-} ) \)} \\[0.0em]
  \(K(t)\) & the subject-specific total number of unique events in \([0,t]\) & p. \pageref{eq:K:t} \\[0.0em]
  \(K= K({\tau})\) & the subject-specific total number of unique events in \([0,\tau]\) &  \\[0.0em]
  \(T^a_{1} <\cdots <T^a_{N^a(\tau)}\) & the subject-specific jump times of \(N^a\) & \\[0.0em]
  \(T^{\ell}_{1} <\cdots <T^{\ell}_{N^{\ell}(\tau)}\) & the subject-specific jump times of \(N^{\ell}\) & \\[0.0em]
  \(T_{1} <\cdots <T_{K}\) &  subject-specific unique event times  & \\[0.0em]
  \(O=\bar{O}(\tau)\) & subject-specific observed data in \([0,\tau]\) & p. \pageref{eq:observed:data} \\[0.0em]
  \(P_0 \in \mathcal{M}\) & the distribution of \(O\)& p. \pageref{eq:like} \\[0.0em]
  \( \mathcal{M}\) & the statistical model containing \(P_0\)& p. \pageref{eq:statistical:model} \\%[0.5em]
  \(G\), \(g\) & interventional part of \(P \in \mathcal{M}\) and its density & p. \pageref{eq:g:0} \\[0.0em]
  \(Q\), \(q\) & non-interventional part of \(P \in \mathcal{M}\) and its density & \\[0.0em]
  \(G^*\), \(g^*\) & intervention and its density & \\[0.0em]
  \(P_{Q,G^*}=P^{G^*}\) & \multicolumn{2}{l}{the post-interventional distribution defined
                          by the g-computation formula} \\%[0.0em]  
  \(\Psi^{G^*} \, : \, \mathcal{M} \rightarrow \R\) & target parameter for fixed
                                                      intervention \(G^*\)& p. \pageref{eq:target:parameter} \\[0.0em]
  \(\bm{Z}\) &  \(\bm{Z}= ( {Z}^{G^*}_{t}, {Z}^{G^*}_{t,L(t)},
               {\Lambda}^{\ell}(t), {\Lambda}^{a}(t), {\Lambda}^{d}(t)
               )_{t\in [0,\tmax]}  \)
                                                  & p.
                                                    \pageref{eq:rel:part:Z} \\[0.0em]
  \(Z^{G^*}_t\) &\(Z^{G^*}_t = \EE_{P^{G^*}} \big[ Y \, \big\vert \, L(t),
                  N^{\ell}(t),N^a(t),
                  N^d(t), \F_{t-}\big] \), \(t\in [0,\tau]\)&  p. \pageref{eq:def:Zt:Lt} \\[0.0em]
  \(Z^{G^*}_{t,L(t)}\) &\(Z^{G^*}_{t,L(t)}=  \EE_{P^{G^*}} \big[ Y \, \big\vert \,
                         N^{\ell}(t),N^a(t),
                         N^d(t), \F_{t-}\big]\), \(t\in [0,\tau]\) & p. \pageref{eq:def:Zt:Lt} \\[0.0em]
  \(h^{G^*}_t\) & clever weights, \(t\in (0,\tau]\) & p. \pageref{eq:clever:weight} \\[0.0em]
  \(h^{\ell}_t, h^a_t, h^d_t\) & clever covariates, \(t\in (0,\tau]\) & p. \pageref{eq:h:ell} \\%[0.5em]
  \(O_1, \ldots, O_n \sim P_0\) & observed data &  p. \pageref{eq:observed:data} \\[0.0em]
  \(\mathbb{P}_n\) & the empirical distribution of \(\{O_i\}_{i=1}^n\) &  \\[0.0em]
  \(t_0 <t_1<\dots<t_{{K}_n} \) &the ordered sequence
                                  of unique times of changes \(\cup_{i=1}^n \{T_{i,k}\}_{k=1}^{K_{i}}\)&  p. \pageref{eq:observed:event:times} \\[0.0em]
  \({K}_n= \sum_{i=1}^n K_{i}\) & the total number of observation times for the data \(\{O_i\}_{i=1}^n\) &\\[0.0em]
  \(\hat{G}_n\) &  estimator for
                  \(G_0\) on \( [0,\tmax]\)& \\[0.0em]
  \(\hat{\bm{Z}}^{m}_n\) & estimator for  \(\bm{Z}\)& \\[0.0em]
  \(\hat{P}_n^*\) & targeted estimator, characterized by  \((\hat{\bm{Z}}^*_n, \hat{G}_n)\), where
                    \(\hat{\bm{Z}}^*_n= \hat{\bm{Z}}^{\mm=\mm^*}_n\)   &\\[0.0em]
  \( \hat{\psi}^{G^*}_n=\Psi^{G^*} (\hat{P}^{*}_{n})\)&  TMLE estimator for the target parameter&\\[0.3em]
  \hline
\end{tabular}
\end{table}

\newpage

\subsection{Details of targeting algorithm}
\label{app:overview:targeting}

We here provide a more detailed description of the targeting procedure
(Section \ref{sec:TMLE}). As in Section \ref{sec:TMLE}, we here use
the notation:
\begin{align}
 &{Z}_{t_{r}}^{G^*}
   = \EE_{P^{G^*}} [ Y \, \vert \, L(t_{r}),
    N^{\ell}(t_{r}), N^{a}(t_{r}), N^{d}(t_{r}),  \F_{t_{r-1}}] , \label{eq:Z:tilde:1}  \\
& {Z}_{t_{r},L(t_{r})}^{G^*}
   = \EE_{P^{G^*}} [ Y \, \vert \,
    N^{\ell}(t_{r}), N^{a}(t_{r}), N^{d}(t_{r}),  \F_{t_{r-1}}] ,\label{eq:Z:tilde:L:1}
\end{align}
for \(r=1, \ldots, {K}_n\). Further, we now introduce:
\begin{align}
  {Z}_{t_{r},N^{\ell}(t_{r})}^{G^*}
  &=                  \EE_{P^{G^*}} [ Y \, \vert \,  
    N^{a}(t_{r}), N^{d}(t_{r}),  \F_{t_{r-1}}] , \label{eq:Z:tilde:ell:1}\\
  {Z}_{t_{r},N^{a}(t_{r})}^{G^*}
  &= 
    \EE_{P^{G^*}} [ Y \, \vert \,
    N^{d}(t_{r}),  \F_{t_{r-1}}] \label{eq:Z:tilde:a:1} \\
  {Z}_{t_{r},N^{d}(t_{r})}^{G^*}
  &=
    \EE_{P^{G^*}} [ Y \, \vert \,
    \F_{t_{r-1}}] ,\label{eq:Z:tilde:d:1}
\end{align}
where the subscript `\(N^x(t_r)\)', \(x=\ell,a,d\), tells us what was
last integrated out over the interval \((t_{r-1},t_r]\). Given current
estimators \(\hat{\bm{Z}}_n^{\mm}\) for \(\bm{Z}\), we construct
estimators,
\begin{align*}
  \hat{Z}^{G^*}_{t_r,\mm}, \hat{Z}^{G^*}_{t_r,L(t_r),\mm},
  \hat{Z}^{G^*}_{t_r,N^{\ell}(t_r),\mm},\hat{Z}^{G^*}_{t_r,N^{a}(t_r),\mm},
  \hat{Z}^{G^*}_{t_r,N^{d}(t_r),\mm},
\end{align*}
for \eqref{eq:Z:tilde:1}--\eqref{eq:Z:tilde:d:1}. Moreover, given
estimators for \eqref{eq:Z:tilde:1}--\eqref{eq:Z:tilde:d:1}, we
construct estimators
\(\hat{h}^{\ell}_{t_r,\mm},\hat{h}^{a}_{t_r,\mm},\hat{h}^{d}_{t_r,\mm}\)
for the clever covariates \(h^{\ell}_{t_r},h^{a}_{t_r},h^{d}_{t_r}\)
by:
\begin{align*}
  \hat{h}^{\ell}_{t_r,\mm} & =  \sum_{\delta= 0,1} (2\delta-1) \hat{Z}^{G^*}_{t_r,L(t_r),\mm} ( N^{\ell}(t_r) = N^{\ell}(t_{r-1}) +
                         \delta),  \\
  \hat{h}^{a}_{t_r,\mm} & =  \sum_{\delta= 0,1} (2\delta-1) \hat{Z}^{G^*}_{t_r,N^{\ell}(t_r),\mm} ( N^{a}(t_r) = N^{a}(t_{r-1}) +
                      \delta) , \\
  \hat{h}^{d}_{t_r,\mm} & = 1-  \hat{Z}^{G^*}_{t_r,N^{a}(t_r),\mm} ( N^{d}(t_r) = 0)  .
\end{align*}
Now we can carry out the individual targeting update steps as
described in Sections \ref{sec:7:update:Z:L} and
\ref{sec:7:update:intensity}. This gives:
\begin{align*}
  &\hat{Z}^{G^*}_{t_r,L(t_r),k} &&\mapsto \hat{Z}^{G^*}_{t_r,L(t_r),k+1} \\
  & \hat{\Lambda}^{\ell}_{\mm} &&\mapsto \hat{\Lambda}^{\ell}_{\mm+1}  \qquad
                                  \qquad\qquad\qquad\qquad\qquad\qquad\qquad\qquad\qquad \\
  & \hat{\Lambda}^{a}_{\mm} &&\mapsto \hat{\Lambda}^{a}_{\mm+1}  \qquad
                               \qquad\qquad\qquad\qquad\qquad\qquad\qquad\qquad\qquad \\
  & \hat{\Lambda}^{d}_{\mm} &&\mapsto \hat{\Lambda}^{d}_{\mm+1}  \qquad
                               \qquad\qquad\qquad\qquad\qquad\qquad\qquad\qquad\qquad
\end{align*}
Starting with \(\hat{Z}^{G^*}_{t_r,L(t_r),\mm+1}\), we use
\(\hat{\Lambda}^{\ell}_{\mm+1}\) to integrate out \(N^{\ell}(t_r)\) to
obtain the updated
\(\hat{Z}^{G^*}_{t_r,N^{\ell}(t_r),\mm+1}\). Similarly, we use
\(\hat{\Lambda}^{a}_{\mm+1}\) to obtain the updated
\(\hat{Z}^{G^*}_{t_r,N^{a}(t_r),\mm+1}\) based on
\(\hat{Z}^{G^*}_{t_r,N^{\ell}(t_r),\mm+1}\) and lastly
\(\hat{\Lambda}^{d}_{\mm+1}\) to obtain
\(\hat{Z}^{G^*}_{t_r,N^{d}(t_r),\mm+1}\) based on
\(\hat{Z}^{G^*}_{t_r,N^{a}(t_r),\mm+1}\).

To proceed, recall that \(\hat{Z}^{G^*}_{t_r,N^{d}(t_r),\mm+1}\)
estimates \( \EE_{P^{G^*}} [ Y \, \vert \, \F_{t_{r-1}}] \), i.e.,
\begin{align*}
 \EE_{P^{G^*}} [ Y \, \vert \, N^{c}(t_{r-1}), A(t_{r-1}),
   L(t_{r-1}),
    N^{\ell}(t_{r-1}), N^{a}(t_{r-1}), N^{d}(t_{r-1}),  \F_{t_{r-2}}] . 
\end{align*}
The distributions of \(N^{c}(t_{r-1}), A(t_{r-1})\) are specified by
our intervention \(G^*\), so that we can now further obtain an updated
\(\hat{Z}^{G^*}_{t_{r-1},\mm+1}\) from
\(\hat{Z}^{G^*}_{t_r,N^{d}(t_r),\mm+1}\) by integrating out
\(N^{c}(t_{r-1}), A(t_{r-1})\) according to \(dG_t^*\) over
\((t_{r-2},t_{r-1}]\). For example, if our intervention imposes no
censoring and sets \(A(t)\) to \(a^*\) throughout (see Equation
\eqref{eq:static} in Section \ref{sec:interventions}), then we have:
\begin{align*}
  \hat{Z}^{G^*}_{t_{r-1},\mm+1} = \hat{Z}^{G^*}_{t_r,N^{d}(t_r),\mm+1}( N^{c}(t_{r-1})=0,
  A(t_{r-1})= a^*) .
\end{align*}
This means that we now have updated estimators:
\begin{align*}
  \hat{Z}^{G^*}_{t_r,\mm+1}, \hat{Z}^{G^*}_{t_r,L(t_r),\mm+1},
  \hat{Z}^{G^*}_{t_r,N^{\ell}(t_r),\mm+1},\hat{Z}^{G^*}_{t_r,N^{a}(t_r),\mm+1},
  \hat{Z}^{G^*}_{t_r,N^{d}(t_r),\mm+1},
\end{align*}
for the entire sequence
\eqref{eq:Z:tilde:1}--\eqref{eq:Z:tilde:d:1}. We can now proceed with
the next update \(\mm+1\) to \( \mm+2\) in the same way as outlined
above. \\

\newpage

\subsection{Overview: Super learning}
\label{app:sub:overview:super:learning}

In this section we briefly describe super learning
\citep{van2007super} adapted to our setting. The super learner is a
machine learning method also known as stacked regression
\citep{wolpert1992stacked,breiman1996stacked}. In the following, let
\(x\) run through \(\lbrace L, A, c, a, \ell, d \rbrace\). In Section
\ref{app:fk:loss} we outlined how all the nuisance parameters can be
represented by a function \(f^x\). Super learning for \(f^x\) involves
the following ingredients:
\begin{enumerate}
\item A collection of candidate estimators
  \(\lbrace \hat{f}^x_j \rbrace_{1 \le j\le J}\) for \(f^x\).
\item A loss function \((O, f^x) \mapsto \mathscr{L}_x(f^x) (O) \).
\end{enumerate}
In Section \ref{app:fk:loss} we defined log-likelihood loss functions
\((O, f^x) \mapsto \mathscr{L}_x(f^x) (O)\).  The super learner
identifies the best performing of the estimators in the collection
\(\lbrace \hat{f}^x_j\rbrace_{1 \le j\le J}\) by minimizing the
empirical risk obtained via \(V\)-fold cross-validation.  A \(V\)-fold
cross-validation scheme defines \(v=1,\ldots,V\) sample splits into a
training sample,
\( \lbrace 1, \ldots, n\rbrace \setminus \mathrm{Val}(v) \), used to
construct the estimator, and a validation sample,
\(\mathrm{Val}(v)\subset \lbrace 1, \ldots, n\rbrace \), used to
evaluate it.

Let \(\mathbb{P}^{v}_{n},\mathbb{P}^{-v}_{n}\) denote the empirical
distribution of the validation and the training sample,
respectively. Let \(\hat{f}^x_j ( \mathbb{P}^{-v}_{n})\) denote the
\(j\)th estimator obtained on the training sample. We define the loss
function based cross-validation selector as follows:
\begin{align*}
  \hat{j}^x_n = \underset{j \in \lbrace 1, \ldots, J\rbrace}{\mathrm{argmin}}
  \frac{1}{V} \sum_{v=1}^V \mathbb{P}^{v}_n \mathscr{L}_x ( \hat{f}^x_j ( \mathbb{P}^{-v}_{n})),
\end{align*}
which yields the super learner
\(\hat{f}^J_n = \hat{f}^x_{\hat{j}^x_n}\).

\newpage

\subsection{Preliminaries on HAL implementation}
\label{app:sub:overview:diagrams}

\quad

\begin{table}[!h]
\begin{center}
  \begin{tabular}{l l l}
    \hline
    \multicolumn{3}{l}{\textbf{HAL estimation using \texttt{R} software.}}\\ \hline\\[0.5em]
    \multicolumn{3}{l}{ \texttt{hal9001::fit\_hal(X, Y, family=family)}} \\[0.7em]
    & \multicolumn{2}{l}{generates a HAL design matrix consisting of basis functions corresponding}
    \\
    & \multicolumn{2}{l}{to covariates and interactions}
    \\[0.7em]
    & \multicolumn{2}{l}{makes a call to:}
    \\
    &\multicolumn{2}{l}{\texttt{glmnet::glmnet(x=X, y=Y, family=family)}} \\[0.7em]
    & \multicolumn{2}{l}{automatically selects a CV-optimal value of this regularization parameter:}
    \\
    &\multicolumn{2}{l}{\texttt{glmnet::cv.glmnet(x=X, y=Y, family=family)}} \\[0.7em]
    \\
    \multicolumn{3}{l}{\underline{Logistic regression:}}\\[0.7em]
    & \texttt{X} & HAL design matrix  \\
    & \texttt{Y} & outcome variable (factor with two levels) \\
    & \texttt{family} & \texttt{'binomial'} \\[0.7em]
    \\ 
    \multicolumn{3}{l}{\underline{Squared error loss:}} \\[0.7em]
    & \texttt{X} & HAL design matrix  \\
    & \texttt{Y} & outcome variable (real-valued) \\
    & \texttt{family} & \texttt{'gaussian'} \\[0.7em]
    \\
    \multicolumn{3}{l}{\underline{Intensity estimation with Poisson likelihood:}}\\[0.7em]
    & \texttt{X} & HAL design matrix  \\
    & \texttt{Y} & Number of events observed for each combination of
                   covariates  \\
    & \texttt{R} & The amount of risk time for each combination of covariates \\
    & &             \texttt{log(R)} is included as an offset in the regression formula   \\
    & \texttt{family} & \texttt{'poisson'} \\
    \\
    \multicolumn{3}{l}{\underline{Intensity estimation with partial likelihood:}}\\[0.7em]
    & \texttt{X} & HAL design matrix  \\
    & \texttt{Y} & outcome variable (two-column matrix with \texttt{'time'} and
                   \texttt{'status'}) \\
    & \texttt{family} & \texttt{'cox'} \\
    \\
    \hline
  \end{tabular}
  \end{center}\caption{}
  \label{table:hal:R}
\end{table}

\newpage

\section{Webappendix}

We here show the claim from Section \ref{ssec:tmle:cond:means} that:
\begin{equation*}
  \frac{d}{d\eps} \,\logloss{{\QQ}_{t}^{G^*}}\big({\QQ}_{t,L(t)}^{G^*}(\eps)\big)\,\Big\vert_{\eps=0} =
  {h}^{G^*}_t \, \big( {\QQ}^{G^*}_{t} - {\QQ}^{G^*}_{t,L(t)}  \big) , 
\end{equation*}
when
\begin{equation*}
  \logloss{{\QQ}_{t}^{G^*}}\big({\QQ}_{t,L(t)}^{G^*}\big) =
  {\QQ}^{G^*}_{t} \log {\QQ}^{G^*}_{t, L(t)} +
  \big(1- {\QQ}^{G^*}_{t} \big) \log \big( 1-  {\QQ}^{G^*}_{t, L(t)}\big) .
\end{equation*} 
and
\begin{equation*}
  \logit \,  {\QQ}^{G^*}_{t,L(t)}(\eps) =\logit\,
  {\QQ}^{G^*}_{t,L(t)} + \eps h_t^{G^*}. 
\end{equation*}

Recall that: 
\begin{align*}
  \expit (x ) = \frac{e^x}{1+e^x} = \frac{1}{1+e^{-x}},
  \qquad \logit (x) =  \log \left( \frac{x}{1-x} \right),
\end{align*}
such that
\begin{align*}
  \expit ( \logit(x)) = x,
\end{align*}
and
\begin{align*}
\expit ( - \logit (x) ) &= \expit(-\log(x) + \log(1-x)) \\
                          &= \expit( \log (1-x) - \log(x) ) \\
                          &= \expit (\logit (1-x)) = 1-x.
\end{align*}
Furthermore, it is easily seen that
\begin{align*}
  \log ( \expit ( x) ) &= - \log (1+ e^{-x}), \\
  \log ( 1-\expit ( x) ) &= - \log (1+ e^{x}), \\
\end{align*}
so that
\begin{align*}
  \frac{d}{d x}  \log ( \expit ( x) ) &= \frac{e^{-x}}{1+e^{-x}} =\expit(-x), \\
  \frac{d}{d x}  \log ( 1-\expit ( x) ) &= -\frac{e^{x}}{1+e^{x}} =- \expit(x). 
\end{align*}
Now we can differentiate the composite functions
\begin{align*}
  \frac{d}{d \eps}  \log( \expit (\logit ( x) + \eps h))
  & = h \, \expit (-(\logit ( x) + \eps h)) ,\\
  \frac{d}{d \eps}  \log(1- \expit (\logit ( x) + \eps h))
  & = -h \, \expit (\logit (x) + \eps h) , 
\end{align*}
where setting \(\eps = 0\) gives
\begin{align*}
  \frac{d}{d \eps} \log( \expit (\logit ( x) + \eps h))  \Big\vert_{\eps=0}
  & = h \, \expit (-(\logit ( x))) = 1-x ,\\
  \frac{d}{d \eps}   \log(1- \expit (\logit ( x) + \eps h))  \Big\vert_{\eps=0}
  & =- h \, \expit (\logit (x) ) = x . 
\end{align*}
Applying these steps to
\(\logloss{{\QQ}_{t}^{G^*}}\big({\QQ}_{t,L(t)}^{G^*}(\eps)\big)\) now
gives:
\begin{align*}
  \frac{d}{d\eps} \,\logloss{{\QQ}_{t}^{G^*}}\big({\QQ}_{t,L(t)}^{G^*}(\eps)\big)\,\Big\vert_{\eps=0}
  &=  {\QQ}^{G^*}_{t} h_t^{G^*} (1- {\QQ}^{G^*}_{t,L(t)})  - (1- {\QQ}^{G^*}_{t})
    h_t^{G^*}  {\QQ}^{G^*}_{t,L(t)} \\
  &=  {\QQ}^{G^*}_{t} h_t^{G^*}   -
    h_t^{G^*}  {\QQ}^{G^*}_{t,L(t)}  = h_t^{G^*} \big(  {\QQ}^{G^*}_{t}
    - {\QQ}^{G^*}_{t,L(t)} \big) , 
\end{align*}
which completes the proof.


\newpage

\begin{thebibliography}{}

\bibitem[\protect\citeauthoryear{Andersen, Petersen, Jimenez-Solem, Broedbaek,
  Andersen, Andersen, Afzal, Torp-Pedersen, Keiding, and Poulsen}{Andersen
  et~al.}{2013}]{andersen2013trimethoprim}
Andersen, J.~T., M.~Petersen, E.~Jimenez-Solem, K.~Broedbaek, E.~W. Andersen,
  N.~L. Andersen, S.~Afzal, C.~Torp-Pedersen, N.~Keiding, and H.~E. Poulsen
  (2013).
\newblock Trimethoprim use in early pregnancy and the risk of miscarriage: a
  register-based nationwide cohort study.
\newblock {\em Epidemiology \& Infection\/}~{\em 141\/}(8), 1749--1755.

\bibitem[\protect\citeauthoryear{Andersen, Borgan, Gill, and Keiding}{Andersen
  et~al.}{1993}]{andersen2012statistical}
Andersen, P.~K., O.~Borgan, R.~D. Gill, and N.~Keiding (1993).
\newblock {\em Statistical models based on counting processes}.
\newblock Springer, New York.

\bibitem[\protect\citeauthoryear{Bang and Robins}{Bang and
  Robins}{2005}]{bang2005doubly}
Bang, H. and J.~M. Robins (2005).
\newblock Doubly robust estimation in missing data and causal inference models.
\newblock {\em Biometrics\/}~{\em 61\/}(4), 962--973.

\bibitem[\protect\citeauthoryear{{Bibaut} and {van der Laan}}{{Bibaut} and {van
  der Laan}}{2019}]{2019arXiv190709244B}
{Bibaut}, A.~F. and M.~J. {van der Laan} (2019, July).
\newblock {Fast rates for empirical risk minimization over c{\`a}dl{\`a}g
  functions with bounded sectional variation norm}.
\newblock {\em arXiv e-prints\/}, arXiv:1907.09244.

\bibitem[\protect\citeauthoryear{Bickel, Klaassen, Ritov, and Wellner}{Bickel
  et~al.}{1993}]{bickel1993efficient}
Bickel, P.~J., C.~A.~J. Klaassen, Y.~Ritov, and J.~A. Wellner (1993).
\newblock Efficient and adaptive inference in semiparametric models.

\bibitem[\protect\citeauthoryear{Breiman}{Breiman}{1996}]{breiman1996stacked}
Breiman, L. (1996).
\newblock Stacked regressions.
\newblock {\em Machine learning\/}~{\em 24\/}(1), 49--64.

\bibitem[\protect\citeauthoryear{Chakraborty and Moodie}{Chakraborty and
  Moodie}{2013}]{chakraborty2013statistical}
Chakraborty, B. and E.~E. Moodie (2013).
\newblock {\em Statistical methods for dynamic treatment regimes}.
\newblock Springer.

\bibitem[\protect\citeauthoryear{Dawid and Didelez}{Dawid and
  Didelez}{2010}]{dawid2010identifying}
Dawid, A.~P. and V.~Didelez (2010).
\newblock Identifying the consequences of dynamic treatment strategies: A
  decision-theoretic overview.
\newblock {\em Statistics Surveys\/}~{\em 4}, 184--231.

\bibitem[\protect\citeauthoryear{Fleming and Harrington}{Fleming and
  Harrington}{2011}]{fleming2011counting}
Fleming, T.~R. and D.~P. Harrington (2011).
\newblock {\em Counting processes and survival analysis}, Volume 169.
\newblock John Wiley \& Sons.

\bibitem[\protect\citeauthoryear{Gill}{Gill}{1994}]{gill1994lectures}
Gill, R.~D. (1994).
\newblock Lectures on survival analysis.
\newblock In {\em Lectures on Probability Theory}, pp.\  115--241. Springer.

\bibitem[\protect\citeauthoryear{Gill and Johansen}{Gill and
  Johansen}{1990}]{gill1990survey}
Gill, R.~D. and S.~Johansen (1990).
\newblock A survey of product-integration with a view toward application in
  survival analysis.
\newblock {\em The annals of statistics\/}~{\em 18\/}(4), 1501--1555.

\bibitem[\protect\citeauthoryear{Gill and Robins}{Gill and
  Robins}{2001}]{gill2001causal}
Gill, R.~D. and J.~M. Robins (2001).
\newblock Causal inference for complex longitudinal data: the continuous case.
\newblock {\em Annals of Statistics\/}, 1785--1811.

\bibitem[\protect\citeauthoryear{Gill, van~der Laan, and Wellner}{Gill
  et~al.}{1995}]{gill1995inefficient}
Gill, R.~D., M.~J. van~der Laan, and J.~A. Wellner (1995).
\newblock {\em Inefficient estimators of the bivariate survival function for
  three models}, Volume~31.
\newblock Annales de l’Institut Henri Poincaré.

\bibitem[\protect\citeauthoryear{Gneiting and Raftery}{Gneiting and
  Raftery}{2007}]{gneiting2007strictly}
Gneiting, T. and A.~E. Raftery (2007).
\newblock Strictly proper scoring rules, prediction, and estimation.
\newblock {\em Journal of the American statistical Association\/}~{\em
  102\/}(477), 359--378.

\bibitem[\protect\citeauthoryear{Hern{\'a}n}{Hern{\'a}n}{2010}]{hernan2010hazards}
Hern{\'a}n, M.~A. (2010).
\newblock The hazards of hazard ratios.
\newblock {\em Epidemiology (Cambridge, Mass.)\/}~{\em 21\/}(1), 13.

\bibitem[\protect\citeauthoryear{Hern{\'a}n, Brumback, and Robins}{Hern{\'a}n
  et~al.}{2000}]{hernan2000marginal}
Hern{\'a}n, M.~{\'A}., B.~Brumback, and J.~M. Robins (2000).
\newblock Marginal structural models to estimate the causal effect of
  zidovudine on the survival of hiv-positive men.
\newblock {\em Epidemiology\/}, 561--570.

\bibitem[\protect\citeauthoryear{Hern{\'a}n, Brumback, and Robins}{Hern{\'a}n
  et~al.}{2001}]{hernan2001marginal}
Hern{\'a}n, M.~A., B.~Brumback, and J.~M. Robins (2001).
\newblock Marginal structural models to estimate the joint causal effect of
  nonrandomized treatments.
\newblock {\em Journal of the American Statistical Association\/}~{\em
  96\/}(454), 440--448.

\bibitem[\protect\citeauthoryear{Hern{\'a}n, Brumback, and Robins}{Hern{\'a}n
  et~al.}{2002}]{hernan2002estimating}
Hern{\'a}n, M.~A., B.~A. Brumback, and J.~M. Robins (2002).
\newblock Estimating the causal effect of zidovudine on cd4 count with a
  marginal structural model for repeated measures.
\newblock {\em Statistics in medicine\/}~{\em 21\/}(12), 1689--1709.

\bibitem[\protect\citeauthoryear{Hern{\'a}n, Lanoy, Costagliola, and
  Robins}{Hern{\'a}n et~al.}{2006}]{hernan2006comparison}
Hern{\'a}n, M.~A., E.~Lanoy, D.~Costagliola, and J.~M. Robins (2006).
\newblock Comparison of dynamic treatment regimes via inverse probability
  weighting.
\newblock {\em Basic \& clinical pharmacology \& toxicology\/}~{\em 98\/}(3),
  237--242.

\bibitem[\protect\citeauthoryear{Hernan and Robins}{Hernan and
  Robins}{2020}]{hernanrobins}
Hernan, M.~A. and J.~M. Robins (2020).
\newblock {\em Causal Inference}.
\newblock Boca Raton, Fl: Chapman \& Hall/CRC.

\bibitem[\protect\citeauthoryear{Karim, Gustafson, Petkau, Tremlett, the
  Long-Term~Benefits, and of~Beta-Interferon~for Multiple Sclerosis (BeAMS)
  Study~Group}{Karim et~al.}{2016}]{karim2016comparison}
Karim, M.~E., P.~Gustafson, J.~Petkau, H.~Tremlett, the Long-Term~Benefits, and
  A.~E. of~Beta-Interferon~for Multiple Sclerosis (BeAMS) Study~Group (2016).
\newblock Comparison of statistical approaches for dealing with immortal time
  bias in drug effectiveness studies.
\newblock {\em American journal of epidemiology\/}~{\em 184\/}(4), 325--335.

\bibitem[\protect\citeauthoryear{Kessing, Rytgaard, Gerds, Berk, Ekstr{\o}m,
  and Andersen}{Kessing et~al.}{2019}]{kessing2019depression}
Kessing, L.~V., H.~C. Rytgaard, T.~A. Gerds, M.~Berk, C.~T. Ekstr{\o}m, and
  P.~K. Andersen (2019).
\newblock New drug candidates for depression -- a nationwide population-based
  study.
\newblock {\em Acta Psychiatrica Scandinavica\/}~{\em 139\/}(1), 68--77.

\bibitem[\protect\citeauthoryear{Lendle, Schwab, Petersen, and {{v}an {d}er
  Laan}}{Lendle et~al.}{2017}]{ltmleRpackage}
Lendle, S.~D., J.~Schwab, M.~L. Petersen, and M.~J. {{v}an {d}er Laan} (2017).
\newblock {ltmle}: An {R} package implementing targeted minimum loss-based
  estimation for longitudinal data.
\newblock {\em Journal of Statistical Software\/}~{\em 81\/}(1), 1--21.

\bibitem[\protect\citeauthoryear{Lok}{Lok}{2008}]{lok2008statistical}
Lok, J.~J. (2008).
\newblock Statistical modeling of causal effects in continuous time.
\newblock {\em The Annals of Statistics\/}~{\em 36\/}(3), 1464--1507.

\bibitem[\protect\citeauthoryear{Martinussen, Vansteelandt, and
  Andersen}{Martinussen et~al.}{2018}]{martinussen2018subtleties}
Martinussen, T., S.~Vansteelandt, and P.~K. Andersen (2018).
\newblock Subtleties in the interpretation of hazard ratios.
\newblock {\em arXiv preprint arXiv:1810.09192\/}.

\bibitem[\protect\citeauthoryear{Murphy}{Murphy}{2003}]{murphy2003optimal}
Murphy, S.~A. (2003).
\newblock Optimal dynamic treatment regimes.
\newblock {\em Journal of the Royal Statistical Society: Series B (Statistical
  Methodology)\/}~{\em 65\/}(2), 331--355.

\bibitem[\protect\citeauthoryear{Murphy}{Murphy}{2005}]{murphy2005experimental}
Murphy, S.~A. (2005).
\newblock An experimental design for the development of adaptive treatment
  strategies.
\newblock {\em Statistics in medicine\/}~{\em 24\/}(10), 1455--1481.

\bibitem[\protect\citeauthoryear{Murphy, van~der Laan, Robins, and
  Group}{Murphy et~al.}{2001}]{murphy2001marginal}
Murphy, S.~A., M.~J. van~der Laan, J.~M. Robins, and C.~P. P.~R. Group (2001).
\newblock Marginal mean models for dynamic regimes.
\newblock {\em Journal of the American Statistical Association\/}~{\em
  96\/}(456), 1410--1423.

\bibitem[\protect\citeauthoryear{Petersen, Schwab, Gruber, Blaser, Schomaker,
  and {v}an~{d}er Laan}{Petersen et~al.}{2014}]{petersen2014targeted}
Petersen, M., J.~Schwab, S.~Gruber, N.~Blaser, M.~Schomaker, and M.~{v}an~{d}er
  Laan (2014).
\newblock Targeted maximum likelihood estimation for dynamic and static
  longitudinal marginal structural working models.
\newblock {\em Journal of causal inference\/}~{\em 2\/}(2), 147--185.

\bibitem[\protect\citeauthoryear{Robins}{Robins}{1986}]{robins1986new}
Robins, J. (1986).
\newblock A new approach to causal inference in mortality studies with a
  sustained exposure period—application to control of the healthy worker
  survivor effect.
\newblock {\em Mathematical modelling\/}~{\em 7\/}(9-12), 1393--1512.

\bibitem[\protect\citeauthoryear{Robins}{Robins}{1989a}]{robins1989control}
Robins, J. (1989a).
\newblock The control of confounding by intermediate variables.
\newblock {\em Statistics in medicine\/}~{\em 8\/}(6), 679--701.

\bibitem[\protect\citeauthoryear{Robins}{Robins}{1992}]{robins1992estimation}
Robins, J. (1992).
\newblock Estimation of the time-dependent accelerated failure time model in
  the presence of confounding factors.
\newblock {\em Biometrika\/}~{\em 79\/}(2), 321--334.

\bibitem[\protect\citeauthoryear{Robins, Orellana, and Rotnitzky}{Robins
  et~al.}{2008}]{robins2008estimation}
Robins, J., L.~Orellana, and A.~Rotnitzky (2008).
\newblock Estimation and extrapolation of optimal treatment and testing
  strategies.
\newblock {\em Statistics in medicine\/}~{\em 27\/}(23), 4678--4721.

\bibitem[\protect\citeauthoryear{Robins}{Robins}{1987}]{robins1987addendum}
Robins, J.~M. (1987).
\newblock Addendum to “a new approach to causal inference in mortality
  studies with a sustained exposure period—application to control of the
  healthy worker survivor effect”.
\newblock {\em Computers \& Mathematics with Applications\/}~{\em 14\/}(9-12),
  923--945.

\bibitem[\protect\citeauthoryear{Robins}{Robins}{1989b}]{robins1989analysis}
Robins, J.~M. (1989b).
\newblock The analysis of randomized and non-randomized aids treatment trials
  using a new approach to causal inference in longitudinal studies.
\newblock {\em Health service research methodology: a focus on AIDS\/},
  113--159.

\bibitem[\protect\citeauthoryear{Robins}{Robins}{1998}]{robins1998marginal}
Robins, J.~M. (1998).
\newblock Marginal structural models. 1997 proceedings of the american
  statistical association, section on bayesian statistical science (pp. 1--10).
\newblock {\em Retrieved from\/}.

\bibitem[\protect\citeauthoryear{Robins}{Robins}{2000a}]{robins2000marginal}
Robins, J.~M. (2000a).
\newblock Marginal structural models versus structural nested models as tools
  for causal inference.
\newblock In {\em Statistical models in epidemiology, the environment, and
  clinical trials}, pp.\  95--133. Springer.

\bibitem[\protect\citeauthoryear{Robins}{Robins}{2000b}]{robins2000robust}
Robins, J.~M. (2000b).
\newblock Robust estimation in sequentially ignorable missing data and causal
  inference models.
\newblock In {\em Proceedings of the American Statistical Association}, Volume
  1999, pp.\  6--10.

\bibitem[\protect\citeauthoryear{Robins}{Robins}{2002}]{robins2002analytic}
Robins, J.~M. (2002).
\newblock Analytic methods for estimating hiv-treatment and cofactor effects.
\newblock In {\em Methodological Issues in AIDS Behavioral Research}, pp.\
  213--288. Springer.

\bibitem[\protect\citeauthoryear{Robins, Hernan, and Brumback}{Robins
  et~al.}{2000}]{robins2000marginala}
Robins, J.~M., M.~A. Hernan, and B.~Brumback (2000).
\newblock Marginal structural models and causal inference in epidemiology.
\newblock {\em Epidemiology\/}~{\em 11\/}(5), 550--560.

\bibitem[\protect\citeauthoryear{Robins, Hern{\'a}n, and Siebert}{Robins
  et~al.}{2004}]{robins2004effects}
Robins, J.~M., M.~A. Hern{\'a}n, and U.~Siebert (2004).
\newblock Effects of multiple interventions.
\newblock {\em Comparative quantification of health risks: global and regional
  burden of disease attributable to selected major risk factors\/}~{\em 1},
  2191--2230.

\bibitem[\protect\citeauthoryear{R{\o}ysland}{R{\o}ysland}{2011}]{roysland2011martingale}
R{\o}ysland, K. (2011).
\newblock A martingale approach to continuous-time marginal structural models.
\newblock {\em Bernoulli\/}~{\em 17\/}(3), 895--915.

\bibitem[\protect\citeauthoryear{R{\o}ysland}{R{\o}ysland}{2012}]{roysland2012counterfactual}
R{\o}ysland, K. (2012).
\newblock Counterfactual analyses with graphical models based on local
  independence.
\newblock {\em The Annals of Statistics\/}~{\em 40\/}(4), 2162--2194.

\bibitem[\protect\citeauthoryear{Stitelman, De~Gruttola, and van~der
  Laan}{Stitelman et~al.}{2012}]{stitelman2012general}
Stitelman, O.~M., V.~De~Gruttola, and M.~J. van~der Laan (2012).
\newblock A general implementation of tmle for longitudinal data applied to
  causal inference in survival analysis.
\newblock {\em The international journal of biostatistics\/}~{\em 8\/}(1).

\bibitem[\protect\citeauthoryear{Tibshirani}{Tibshirani}{1996}]{tibshirani1996regression}
Tibshirani, R. (1996).
\newblock Regression shrinkage and selection via the lasso.
\newblock {\em Journal of the Royal Statistical Society: Series B
  (Methodological)\/}~{\em 58\/}(1), 267--288.

\bibitem[\protect\citeauthoryear{Tsiatis}{Tsiatis}{2007}]{tsiatis2007semiparametric}
Tsiatis, A. (2007).
\newblock {\em Semiparametric theory and missing data}.
\newblock Springer Science \& Business Media.

\bibitem[\protect\citeauthoryear{{v}an~{d}er Laan}{{v}an~{d}er
  Laan}{2010a}]{van2010targeted}
{v}an~{d}er Laan, M.~J. (2010a).
\newblock Targeted maximum likelihood based causal inference: Part {I}.
\newblock {\em The International Journal of Biostatistics\/}~{\em 6\/}(2).

\bibitem[\protect\citeauthoryear{{v}an~{d}er Laan}{{v}an~{d}er
  Laan}{2010b}]{van2010targetedII}
{v}an~{d}er Laan, M.~J. (2010b).
\newblock Targeted maximum likelihood based causal inference: Part {II}.
\newblock {\em The international journal of biostatistics\/}~{\em 6\/}(2).

\bibitem[\protect\citeauthoryear{{v}an~{d}er Laan}{{v}an~{d}er
  Laan}{2017}]{van2017generally}
{v}an~{d}er Laan, M.~J. (2017).
\newblock A generally efficient targeted minimum loss based estimator based on
  the highly adaptive lasso.
\newblock {\em The International Journal of Biostatistics\/}~{\em 13\/}(2).

\bibitem[\protect\citeauthoryear{{v}an~{d}er Laan and Dudoit}{{v}an~{d}er Laan
  and Dudoit}{2003}]{van2003unified}
{v}an~{d}er Laan, M.~J. and S.~Dudoit (2003).
\newblock Unified cross-validation methodology for selection among estimators
  and a general cross-validated adaptive epsilon-net estimator: Finite sample
  oracle inequalities and examples.

\bibitem[\protect\citeauthoryear{{v}an~{d}er Laan and Gruber}{{v}an~{d}er Laan
  and Gruber}{2012}]{van2012targeted}
{v}an~{d}er Laan, M.~J. and S.~Gruber (2012).
\newblock Targeted minimum loss based estimation of causal effects of multiple
  time point interventions.
\newblock {\em The international journal of biostatistics\/}~{\em 8\/}(1).

\bibitem[\protect\citeauthoryear{{v}an~{d}er Laan and Petersen}{{v}an~{d}er
  Laan and Petersen}{2007}]{van2007causal}
{v}an~{d}er Laan, M.~J. and M.~L. Petersen (2007).
\newblock Causal effect models for realistic individualized treatment and
  intention to treat rules.
\newblock {\em The international journal of biostatistics\/}~{\em 3\/}(1).

\bibitem[\protect\citeauthoryear{{v}an~der Laan, Polley, and Hubbard}{{v}an~der
  Laan et~al.}{2007}]{van2007super}
{v}an~der Laan, M.~J., E.~C. Polley, and A.~E. Hubbard (2007).
\newblock Super learner.
\newblock {\em Statistical applications in genetics and molecular
  biology\/}~{\em 6\/}(1).

\bibitem[\protect\citeauthoryear{{v}an~{d}er Laan and Robins}{{v}an~{d}er Laan
  and Robins}{2003}]{vanRobins2003unified}
{v}an~{d}er Laan, M.~J. and J.~M. Robins (2003).
\newblock {\em Unified methods for censored longitudinal data and causality}.
\newblock Springer Science \& Business Media.

\bibitem[\protect\citeauthoryear{{v}an~{d}er Laan and Rose}{{v}an~{d}er Laan
  and Rose}{2011}]{van2011targeted}
{v}an~{d}er Laan, M.~J. and S.~Rose (2011).
\newblock {\em Targeted learning: causal inference for observational and
  experimental data}.
\newblock Springer Science \& Business Media.

\bibitem[\protect\citeauthoryear{{v}an~{d}er Laan and Rose}{{v}an~{d}er Laan
  and Rose}{2018}]{van2018targeted}
{v}an~{d}er Laan, M.~J. and S.~Rose (2018).
\newblock {\em Targeted learning in data science: causal inference for complex
  longitudinal studies}.
\newblock Springer.

\bibitem[\protect\citeauthoryear{{v}an~{d}er Laan and Rubin}{{v}an~{d}er Laan
  and Rubin}{2006}]{van2006targeted}
{v}an~{d}er Laan, M.~J. and D.~Rubin (2006).
\newblock Targeted maximum likelihood learning.
\newblock {\em The International Journal of Biostatistics\/}~{\em 2\/}(1).

\bibitem[\protect\citeauthoryear{{v}an~{d}er Vaart}{{v}an~{d}er
  Vaart}{2000}]{van2000asymptotic}
{v}an~{d}er Vaart, A.~W. (2000).
\newblock {\em Asymptotic statistics}, Volume~3.
\newblock Cambridge university press.

\bibitem[\protect\citeauthoryear{{v}an~{d}er Vaart, Dudoit, and {v}an~{d}er
  Laan}{{v}an~{d}er Vaart et~al.}{2006}]{van2006oracle}
{v}an~{d}er Vaart, A.~W., S.~Dudoit, and M.~J. {v}an~{d}er Laan (2006).
\newblock Oracle inequalities for multi-fold cross validation.
\newblock {\em Statistics \& Decisions\/}~{\em 24\/}(3), 351--371.

\bibitem[\protect\citeauthoryear{{v}an~{d}er Vaart and Wellner}{{v}an~{d}er
  Vaart and Wellner}{1996}]{van1996weak}
{v}an~{d}er Vaart, A.~W. and J.~A. Wellner (1996).
\newblock Weak convergence.
\newblock In {\em Weak convergence and empirical processes}, pp.\  16--28.
  Springer.

\bibitem[\protect\citeauthoryear{Wolpert}{Wolpert}{1992}]{wolpert1992stacked}
Wolpert, D.~H. (1992).
\newblock Stacked generalization.
\newblock {\em Neural networks\/}~{\em 5\/}(2), 241--259.

\bibitem[\protect\citeauthoryear{Zhang, Tsiatis, Laber, and Davidian}{Zhang
  et~al.}{2012}]{zhang2012robust}
Zhang, B., A.~A. Tsiatis, E.~B. Laber, and M.~Davidian (2012).
\newblock A robust method for estimating optimal treatment regimes.
\newblock {\em Biometrics\/}~{\em 68\/}(4), 1010--1018.

\bibitem[\protect\citeauthoryear{Zhang, Tsiatis, Laber, and Davidian}{Zhang
  et~al.}{2013}]{zhang2013robust}
Zhang, B., A.~A. Tsiatis, E.~B. Laber, and M.~Davidian (2013).
\newblock Robust estimation of optimal dynamic treatment regimes for sequential
  treatment decisions.
\newblock {\em Biometrika\/}~{\em 100\/}(3), 681--694.

\end{thebibliography}
\end{document}